%% file: OpInf_with_localized_nonlinearity.tex
\documentclass[preprint]{elsarticle} 

\input{ListPackages}

\usepackage{fullpage}
\usepackage[textsize=tiny]{todonotes}

\newcommand{\cstate}{s}      
\newcommand{\dstate}{\mathbf{s}}	
\newcommand{\Mstate}{\mathbf{S}}   

\begin{document}

\begin{frontmatter}
	
	\title{Operator inference for non-intrusive model reduction of systems with non-polynomial nonlinear terms}
	
	
	\author[mpi]{Peter Benner}
	
	\author[mpi]{Pawan Goyal}
		
	\author[ucsd]{Boris Kramer\corref{cor1}}
	\ead{bmkramer@ucsd.edu; +1 858-246-5327}
	
	\author[courant]{Benjamin Peherstorfer}
	
	\author[ut]{Karen Willcox}
	
	\cortext[cor1]{Corresponding author}
	
	\address[mpi] {Max Planck Institute for Dynamics of Complex Technical Systems, Magdeburg, Germany}
		
	\address[ucsd]{Department of Mechanical and Aerospace Engineering, University of California San Diego, CA, United States}
	
	\address[courant]{Courant Institute of Mathematical Sciences, New York University, NY, United States} 
	
	\address[ut]{Oden Institute for Computational Engineering \& Sciences, The University of Texas at Austin, TX, United States}

\begin{abstract}
	This work presents a non-intrusive model reduction method to learn low-dimensional models of dynamical systems with non-polynomial nonlinear terms that are spatially local and that are given in analytic form. In contrast to state-of-the-art model reduction methods that are intrusive and thus require full knowledge of the governing equations and the operators of a full model of the discretized dynamical system, the proposed approach requires only the non-polynomial terms in analytic form and learns the rest of the dynamics from snapshots computed with a potentially black-box full-model solver. The proposed method learns operators for the linear and polynomially nonlinear dynamics via a least-squares problem, where the given non-polynomial terms are incorporated in the right-hand side.  The least-squares problem is linear and thus can be solved efficiently in practice. %
	The proposed method is demonstrated on three problems governed by partial differential equations, namely the diffusion-reaction Chafee-Infante model, a tubular reactor model for reactive flows, and a batch-chromatography model that describes a chemical separation process. The numerical results provide evidence that the proposed approach learns reduced models that achieve comparable accuracy as models constructed with state-of-the-art intrusive model reduction methods that require full knowledge of the governing equations. 
\end{abstract}	
	
	
\end{frontmatter}

\section{Introduction} \label{sec:intro}

Model reduction constructs computationally efficient reduced models of large-scale dynamical systems by approximating the high-dimensional states in low-dimensional subspaces of the state space. Reduced models are especially beneficial in the many-query and outer-loop setting, such as in optimization, design, uncertainty quantification, and control, where models are evaluated many times for different inputs. An overview of the field of model reduction can be found in, e.g., \cite{BCOW2017morBook,quarteroni2014reduced}.  
Most traditional model reduction techniques are of intrusive nature, which means that full knowledge of the governing equations and the discretization of the full models of the dynamical systems of interest are required. The intrusive nature limits the scope of traditional model reduction methods. For example, consider proprietary software that implements the simulation of a physical process. Details, and access to, the governing equations, discretization, and solver typically are unavailable when working with proprietary software and thus traditional, intrusive model reduction methods are not applicable. 
In contrast, non-intrusive model reduction techniques aim to learn reduced models from snapshots, i.e., either numerical approximations or measurements of the states and the outputs of the dynamical systems. 
There is a variety of non-intrusive model reduction techniques for linear systems, yet methods for learning reduced-order nonlinear systems---which is the focus of this paper---are more scarce.

For linear systems, the Loewner framework constructs reduced models from input-output measurements by fitting a linear-time invariant model to frequency or time-domain data~\cite{morMayA07,antoulas14parametricLoewner,PSW16TLoewner,schulze2018data}. System identification techniques such as the matrix pencil approach~\cite{antoulas12pencilSYSID}, eigensystem realization algorithm~\cite{kung78new,rowley_ma11BTera,KG16TangentialInterpolationERA,KG18ERA_CUR_Hankel}, and vector fitting~\cite{gustavsen99VF,drmac15VF} can accurately approximate dynamical systems and in some cases even retain structure. For nonlinear systems, non-intrusive model reduction generally requires a choice of parametrization of the nonlinear term. For example, the Loewner approach has been extended from linear to bilinear \cite{doi:10.1137/15M1041432} and quadratic-bilinear systems \cite{gosea2018LoewnerQB}. Techniques based on libraries select the nonlinear functions that compose the right-hand-side of a dynamical system from candidate functions that are chosen either by expert knowledge~\cite{brunton2016discovering} or through sparse approximation techniques~\cite{rudy2017data,schaeffer2018extracting}. Dynamic mode decomposition~\cite{rowley09spectralDMD,schmid2010dynamic} learns reduced models for nonlinear dynamical systems, yet its success depends on knowledge of suitable observables that make the input-output map low-dimensional~\cite{williams2015data}. 
A promising approach to consider DMD modes together with the fluctuation-dissipation theorem allows the authors in~\cite{khodkar_hassanzadeh_2018} to learn a linear ROM for nonlinear turbulent flows. 
There are also hybrid approaches to learn only part of the nonlinear system, e.g., in data-driven closure modeling, where a projection-based framework is pursued, but correction terms are learned from data~\cite{xie2018data}, or the Galerkin-ROMs can be re-calibrated \cite{couplet2005calibration,cordier2010calibration} to achieve greater accuracy.
Another approach to learning nonlinear reduced models is operator inference~\cite{Peherstorfer16DataDriven} that is applicable if the nonlinear terms are polynomials in the state.  The operator inference framework has been extended to general nonlinear systems by utilizing variable transformations~\cite{QKMW_2019_Transform_and_Learn}.

In practice, one is often confronted with a mixture of the intrusive and non-intrusive setting, such as when the governing equations of the system of interest are known yet details about the discretization are unavailable. For example, if one considers a reaction-diffusion process (such as the tubular reactor example in \Cref{sec:numerics_tubular}), then the reaction term might be given analytically, whereas the discretization of the terms corresponding to the diffusion are unavailable.
We thus propose a novel non-intrusive model reduction method that incorporates terms that are available analytically and that learns reduced operators for all other terms from snapshot data. 
To this end, we build on the operator inference method and extend it in multiple directions. First, we explicitly incorporate  knowledge about non-polynomial terms that are given analytically. Second, we combine operator inference with empirical interpolation \cite{morBarMNetal04,grepl2007efficient,deim2010} to make the learned models online efficient in case of non-polynomial nonlinear terms. 
Third, we demonstrate that structure in coupled (partial) differential equations can be retained in the learned reduced models, which increases the predictive capabilities of the reduced models.
We show that under certain conditions the learned model converges to the same model that is derived with intrusive model reduction methods that require full knowledge of the governing equations and the full-model solvers. Numerical results demonstrate that our approach learns predictive reduced models for  diffusion-reaction systems and simulations of chemical processes. 

The manuscript is structured as follows. \Cref{sec:problemformulation} describes the problem setting and presents background material on model reduction. \Cref{sec:use_nonlin} presents the proposed operator inference problem for nonlinear non-polynomial systems. \Cref{sec:numerics} presents numerical results that provide evidence of the wide scope of our approach. Conclusions are drawn in \Cref{sec:conclusion}.

\section{Problem setting and preliminaries} \label{sec:problemformulation}
In \Cref{sec:PDE}, we introduce the partial differential equation (PDE) model considered herein, followed by its discretized version in \Cref{sec:ODE}. In \Cref{sec:traditional}, we describe the standard projection-based model reduction technique and in \Cref{sec:problemsetting} we state the problem setting of this paper.

\subsection{Partial differential equation model} \label{sec:PDE}
Let $0 < T \in \mathbb{R}$ and consider a PDE of the form
\begin{equation}\label{eq:PDE}
	\frac{\partial \cstate}{\partial t } = \cA(\cstate) + \cH(\cstate) + f(t,\cstate) + \cB(u)
\end{equation}
for time $t \in (0, T]$, input $u(t)$, state $\cstate(x,t)$ with $x \in \Omega \subseteq \R^d, \ d=1,2,3$, and $\Omega$ a bounded open domain with boundary $\partial \Omega$. We assume boundary conditions $\cstate_{| \partial \Omega}$ and initial conditions $\cstate(x,0)$ are given such that the corresponding initial boundary value problem with governing PDE~\eqref{eq:PDE} is well posed up to time $T$.
In equation~\eqref{eq:PDE}, $\cA$ is a linear operator (satisfying $\cA(a\cstate) = a\cA(\cstate)$ for all $a\in \R$ and $\cA (\cstate_1 + \cstate_2) = \cA(\cstate_1) + \cA(\cstate_2)$, e.g., a difference or integral operator), the operator $\mathcal{H}$ is quadratic in $\cstate$ (satisfying $\cH(a\cstate) = a^2 \cH(\cstate)$ for all $\cstate$, e.g., $\mathcal{H}(\cstate) = \cstate \cdot \nabla \cstate$), the nonlinear function is $f(t,\cstate)$, and $\mathcal{B}$ is a linear input operator.

\subsection{Discretized model} \label{sec:ODE}
A semi-discrete numerical model for the PDE~\eqref{eq:PDE} is given by
\begin{align}\label{eq:FOM}
	\dot \dstate(t)& = \bA\dstate(t) + \bH(\dstate(t)\otimes \dstate(t)) + \boldf(t,\dstate(t)) +  \bB\bu(t),
\end{align}
with finite-dimensional state $\dstate(t)\in \R^n$, input $\bu(t)\in \R^m$, linear operator  $\bA\in \R^{n\times n}$, quadratic operator $\bH \in \R^{n\times n^2}$, input operator $\bB \in \R^{n \times m}$. 
The initial condition for the semi-discrete model~\eqref{eq:FOM} is $\dstate(0) = \dstate_0$ and the input at time $t = 0$ is $\bu(0)=\bu_0$.  Throughout this paper, we call \eqref{eq:FOM} the full-order model (FOM). 
Here, $\otimes$ denotes the column-wise Kronecker product, which for a column vector $\dstate = [\cstate_1, \cstate_2, \ldots, \cstate_n]^\top$ is given by 
$$
\dstate \otimes \dstate = [\cstate_1^2 \ \ \cstate_1 \cstate_2 \ \ldots \cstate_1 \cstate_n \ \ \cstate_2 \cstate_1 \   \ \cstate_2^2 \ \ldots \ \cstate_2\cstate_m \ \ldots \cstate_m^2] \in \R^{n^2},
$$
and for a matrix $\Mstate = [\dstate_1, \dstate_2, \ldots, \dstate_k] \in \R^{n\times k}$ is given by 
$$
\Mstate \otimes \Mstate = [\dstate_1 \otimes \dstate_1 \ \ \dstate_2 \otimes \dstate_2  \ \ \ldots \ \ \dstate_k \otimes \dstate_k ] \in \R^{n^2 \times k}.
$$

In the following, we assume that each component function $f_1, \dots, f_n: [0, T] \times \mathbb{R} \to \mathbb{R}$ of $\boldf$ requires evaluating $f$ from the continuous model \eqref{eq:PDE} at a single component of the state vector $\dstate(t) = [\cstate_1(t), \dots, \cstate_n(t)]^\top$. Formally, this means that $\boldf$ is 
\begin{equation} \label{eq:def_f}
	\boldf(t, \dstate) = \begin{bmatrix}
		f(t, \cstate_1)\\
		\vdots\\
		f(t, \cstate_n)
	\end{bmatrix}\,.
\end{equation}
We introduced the PDE model~\eqref{eq:PDE} and the definition of $\boldf$ in~\eqref{eq:def_f} for a scalar model (describing the evolution of one physical quantity, e.g., velocity) for ease of notation.  For PDEs with $\ell$ physical variables, that means that each component function in~\eqref{eq:def_f} can depend on other physical variables at the same spatial location. 
Next, we give a concrete example of a PDE and its discretization with non-polynomial nonlinear terms as in~\eqref{eq:def_f}.

\begin{example}
	Consider the PDE
	$$
	\frac{\partial}{\partial t} \cstate(x,t) = \frac{\partial^2}{\partial x^2} \cstate(x,t) + e^{-\beta t}  \cstate(x,t)^{-\alpha} + b(x) u(t)
	$$
	with state $\cstate(x,t)$, diffusion operator $\frac{\partial}{\partial x^2}$, time-dependent reaction term $f(t,\cstate) = e^{-\beta t}  \cstate^{-\alpha}$, one-dimensional input $u(t)$, input function $b(x)$ and constants $\alpha, \beta \in \R_+$. Let $\dstate$ denote the discretized, finite-dimensional state vector obtained through a finite difference approximation. A discretization of the diffusion operator results in a matrix representation $\bA$, whose structure depends on the discretization scheme, approximation order, and boundary conditions. Likewise, the discretized version of $b(x)$ is $\bB$. For a vector $\dstate$, we define $\dstate^{-\alpha}:= [s_i]^{-\alpha}, \ i=1,2, \ldots, n$ as the componentwise power of the vector entries. The discretized system reads as 
	$$
	\dot{\dstate}(t) = \bA \dstate(t) + e^{-\beta t} \dstate(t)^{-\alpha} + \bB u(t).
	$$
	Note that the nonlinear function $\boldf(t,\dstate(t)) = e^{-\beta t}  \dstate(t)^{-\alpha}$ does not require approximation of spatial derivatives.  In particular, evaluating the semi-discrete nonlinear function $\boldf(t,\dstate)$ only requires application of $f(t,\cstate)$ at every component of $\dstate$.
\end{example}

In this work, we consider functions $f(t,\cstate)$ that are continuous in $\cstate$ and that are \textit{spatially local}, which means they involve nonlinear expressions in the state $\cstate$; the class of problems not considered in this paper are functions $f$ that include differentials or integrals of $\cstate$.
Other examples of spatially local nonlinear terms are the Arrhenius reaction model $\exp{(\gamma - \frac{\gamma}{\cstate} )}$ in the tubular reactor model in \Cref{sec:numerics_tubular} and the function $\dfrac{\cstate}{\alpha+\cstate}$ in the Batch Chromatography example in \Cref{sec:numerics_batchChrom}; see also \cite{morBarMNetal04,deim2010}.

\subsection{Traditional projection-based model reduction for dynamical systems} \label{sec:traditional}
In projection-based model reduction, the semi-discrete model~\eqref{eq:FOM} is projected onto a low-dimensional subspace. Let $r \ll n$ and let $\bV = [\bv_1, \dots, \bv_{r}] \in \mathbb{R}^{n \times r}$ be an orthonormal matrix, where the columns of $\bV$ span an $r$-dimensional subspace of $\mathbb{R}^{n}$. A basis matrix $\bV$ of such a subspace can be computed with, e.g., proper orthogonal decomposition (POD)~\cite{lumley1967structure,sirovich87turbulence,holmes_lumley_berkooz_1996}.  Let $\Mstate = [\dstate_1, \ldots, \dstate_k]$ be snapshot matrix containing solution snapshots $\dstate_i \approx \dstate(t_i)$, where the $\dstate_i$ are solutions to~\eqref{eq:FOM} computed with a time-stepping scheme. POD computes the basis matrix $\bV$ from $\Mstate$ via singular value decomposition (SVD).

Consider now the reduced operators
\begin{equation}
	\widetilde{\bA} = \bV^\top\bA\bV \in \R^{r \times r}\,,\quad \check{\bH} = \bV^\top \bH (\bV \otimes' \bV) \in \R^{r \times r^2}\,,\quad \widetilde{\bB} = \bV^\top\bB \in \R^{r \times m}\, ,
	\label{eq:Prelim:RedOperators}
\end{equation}
which are the restrictions of the matrices of the FOM from~\eqref{eq:FOM} to the image of $\bV$. 
To remove redundant terms in the original Kronecker product, we use throughout the manuscript the product $\dstate \otimes' \dstate$ that contains as components $\dstate_i \dstate_j, i,j = 1,2, \ldots, n, j\geq i$, which is the Kronecker product yet the duplicate terms are removed. For instance, for $\dstate = [\dstate_1, \dstate_2]^\top$, the standard Kronecker product yields $\dstate\otimes\dstate = [\dstate_1^2 \ \dstate_1 \dstate_2 \ \dstate_2 \dstate_1 \ \dstate_2^2]^\top $ and without redundant terms we have $\dstate\otimes'\dstate = [\dstate_1^2 \ \dstate_1 \dstate_2 \ \dstate_2^2]^\top$.  To this end, we define the matrix $\widetilde{\bH} \in \R^{r\times (r^2+r)/2}$ (without redundant terms) as
$$
\widetilde{\bH} = \check{\bH}_{:,r(i-1) + j}, \qquad i=1, \ldots r, \  i \leq j \leq r,
$$
where $\check{\bH}_{:,j}$ denotes $j$th column of the matrix $\check{\bH}$.
Moreover, we define the function $\widetilde{\boldf}: [0, T] \times \R^r \to \R^r$ as
\[
\widetilde{\boldf}(t, \widetilde{\dstate}(t)) = \bV^\top\boldf(t, \bV\widetilde{\dstate}(t))\,,
\]
where $\widetilde\dstate \in \mathbb{R}^{r}$. Note that there are efficient reduction strategies that approximately compute $\widetilde{\boldf}$ so that $\boldf$ does not have to be evaluated at all $n$ components of $\bV\widetilde{\dstate}$ \cite{astrid2008missing,morBarMNetal04,carlberg2013gnat,deim2010,nguyen2008best}; we will return to such reduction strategies of $\widetilde{\boldf}$ in \Cref{sec:DEIM}. The reduced operators and $\hat{\boldf}$ define the reduced-order model (ROM)
\begin{align}\label{eq:ROM}
	\dot{\widetilde{\dstate}}(t)& = \widetilde \bA\widetilde \dstate(t) + \widetilde \bH(\widetilde \dstate(t) \otimes' \widetilde \dstate(t)) + \widetilde\boldf(t,\widetilde\dstate(t)) +  \widetilde \bB\bu(t)
\end{align}
with the reduced state $\widetilde{\dstate}(t) \in \mathbb{R}^r$ at time $t$ and with initial condition $\widetilde \dstate(0) = \widetilde\dstate_0$. 
Constructing reduced operators \eqref{eq:Prelim:RedOperators} via projection requires that the operators $\bA, \bB, \bH$ of the system \eqref{eq:FOM} are available in assembled form or as matrix-vector products. Furthermore, each evaluation of $\widetilde{\boldf}$ requires evaluating components of $\boldf$.
%

\subsection{Problem setting} \label{sec:problemsetting}
In this work, we consider the situation that we have a PDE model~\eqref{eq:PDE} with spatially local nonlinear term $f$ as described in \Cref{sec:PDE}, and that we can obtain snapshot data $\dstate(t_i),\bu(t_i)$ of the model~\eqref{eq:FOM} at times $0=t_0< t_1 <\cdots < t_k$. 
However, the discretized operators $\bA, \bH, \bB$ are unavailable. 
The goal of this work is to learn a nonlinear ROM of the form~\eqref{eq:ROM}.

\section{Operator inference for nonlinear systems} \label{sec:use_nonlin}
In \Cref{sec:OpInf_knownF}, we extend the operator inference framework for low-order polynomial systems from~\cite{Peherstorfer16DataDriven} to the setting described in \Cref{sec:PDE} with general non-polynomial nonlinear function $f$. In \Cref{sec:OPINF_analysis}, we present a result that shows that under certain conditions, the learned operators converge to their projection-based counterparts. In \Cref{sec:DEIM}, we discuss empirical interpolation \cite{morBarMNetal04,deim2010} to efficiently reduce the non-polynomial nonlinear terms in the learned ROMs.  

\subsection{Operator inference method} \label{sec:OpInf_knownF}
We now extend the operator inference method introduced in~\cite{Peherstorfer16DataDriven} for polynomially nonlinear systems to systems with non-polynomial nonlinear terms that are spatially local. Let $\dstate_1, \ldots, \dstate_k$ be the solutions of the FOM~\eqref{eq:FOM} at time steps $t_1, \ldots, t_k$ computed with a time stepping scheme and initial condition $\dstate_0$.  Let the inputs $\bu(t_0), \ldots, \bu(t_k)$ be sampled at the same time steps. Then, the snapshot and input trajectory are defined as
\begin{equation} \label{eq:Xdot_X}
	{\Mstate}  : = 
	\left[
	\begin{array}{llll}
		~ \vertbar & ~\vertbar &        & ~\vertbar \\
		\dstate_0    & \dstate_1    & \cdots & \dstate_k   \\
		~\vertbar & ~\vertbar &        & ~\vertbar 
	\end{array}
	\right], \qquad 
	{\bU}  : = 
	\left[
	\begin{array}{llll}
		~ \vertbar & ~\vertbar &        & ~\vertbar \\
		\bu({t_0})    & \bu({t_1})    & \cdots & \bu({t_k})    \\
		~\vertbar & ~\vertbar &        & ~\vertbar 
	\end{array}
	\right].
\end{equation}
Note that we can evaluate $\boldf(t_i,\dstate(t_i))$ using the solution $\dstate(t)$ and the analytical form of the spatially local nonlinear term $f$. That is, we can post-process the snapshots to obtain
\begin{equation}\label{eq:F}
	{\bF}  = 
	\left[
	\begin{array}{cccc}
		\vertbar & \vertbar &        & \vertbar \\
		\boldf(t_0,\dstate(t_0))    & \boldf(t_1,\dstate(t_1))    & \cdots & \boldf(t_k,\dstate(t_k))    \\
		\vertbar & \vertbar &        & \vertbar 
	\end{array}
	\right].
\end{equation}
In the operator inference framework, the goal is to learn a ROM of the form~\eqref{eq:ROM} directly from data. One starts with the projected  trajectories of $\dstate(t)$, $\dot \dstate(t)$ and $\boldf(t,\dstate)$, which are obtained via the following projections onto the basis matrix $\bV\in \R^{n\times r}$:
\begin{equation} \label{eq:Shat}
	\hat{\Mstate} = \bV^\top{\Mstate},\qquad 
	{\hat{\bF}} = \bV^\top {\bF}.
\end{equation}
We denote with $\dot{\widehat{\dstate}}_k$ the time derivative approximation of $\frac{\text{d}}{\text{d}t} \hat{\dstate}(t_k)$, which can be computed from $\hat{\dstate}$ using a time derivative approximation (see, e.g.,~\cite{martins2013review,knowles2014methodsDifferentiation,chartrand2017numericalDifferentiation}). We store the time-derivative approximations in the matrix
\begin{equation} \label{Xdot}
	\dot{\widehat{\Mstate} } : = 
	\left[
	\begin{array}{llll}
		~ \vertbar & ~\vertbar &        & ~\vertbar \\
		\dot{\widehat{\dstate}}({t_0})    & \dot{\widehat{\dstate}}({t_1})    & \cdots & \dot{\widehat{\dstate}}({t_k})    \\
		~\vertbar & ~\vertbar &        & ~\vertbar 
	\end{array}
	\right].
\end{equation}
The optimization problem to compute $\hat \bA$, $\hat \bB$ and $\hat \bH$ from the above projected data is:
\begin{equation} \label{eqn:optimizationProb}
	\min_{\hat \bA,\hat \bB,\hat \bH} \  \| \underbrace{\dot{\hat{\Mstate}} - \hat{\bF}}_{:=\hat{\bR}}- \hat \bA \hat{\Mstate} - \hat \bB \bU -\hat \bH (\hat{\Mstate}\otimes' \hat{\Mstate})\|_F
\end{equation}
where $\hat{\bR}$ contains only known terms. 
%
%
Taken together, the solution of the operator inference problem for $\hat \bA,\hat \bB,\hat \bH$ yields a learned ROM of the form 
\begin{align}  \label{eq:ROM_OPINF}
	\dot{\hat{\dstate}}(t)& = \hat \bA\hat \dstate(t) + \hat \bH(\hat \dstate (t )\otimes' \hat \dstate (t)) + \bV^\top \boldf(t,\bV\hat \dstate(t)) +  \hat \bB\bu(t), 
\end{align} 
which can then be used for predictive simulations. \Cref{algo:OpInfNL} summarizes the algorithmic steps of this section. 
\begin{algorithm}[!tb]
	\caption{Operator inference for nonlinear systems.}
	\begin{algorithmic}[1]
		\State {\bf Input:} $\dstate_0,\Mstate,\bU,\bF$ and a user-specified tolerance $\texttt{tol}$.
		\State Collect derivative data $\dot{\Mstate}$ by numerical approximation or evaluation of a right-hand side residual on the snapshot matrix $\Mstate$. 
		\State Compute the $r$  dominant POD basis vectors of $\Mstate$, resulting in $\bV$ such that 
		$$\dfrac{\|\Mstate - \bV\bV^\top\Mstate\|}{\|\Mstate\|} \leq \texttt{tol}.$$
		\State Determine projected initial condition $\hat{\dstate}_0$, reduced states $\hat{\Mstate}$, reduced derivative $\dot{\hat{\Mstate}}$, and reduced nonlinear terms $\hat{\bF}$:
		$$\hat{\dstate}_0 := \bV^\top\dstate_0, \quad\hat{\Mstate} := \bV^\top\Mstate,\quad \dot{\hat{\Mstate}} = \bV^\top\dot{{\Mstate}},\quad  \hat{\bF} := \bV^\top\bF.$$
		\State Solve the optimization problem \eqref{eqn:optimizationProb}, yielding $\hat{\bA}$, $\hat{\bB}$, $\hat{\bH}_\text{sq}$. 
		\State  {\bf Output:} $\hat{\dstate}_0$, $\hat{\bA}$, $\hat{\bB}$, $\hat{\bH}$.
	\end{algorithmic}\label{algo:OpInfNL}
\end{algorithm}
In \Cref{fig:compare_PODDEIM_NonIntru}, we illustrate the differences of the  proposed approach of using operator inference together with knowledge of spatially nonlinear functions to learn ROMs as compared to an intrusive projection-based ROM. 
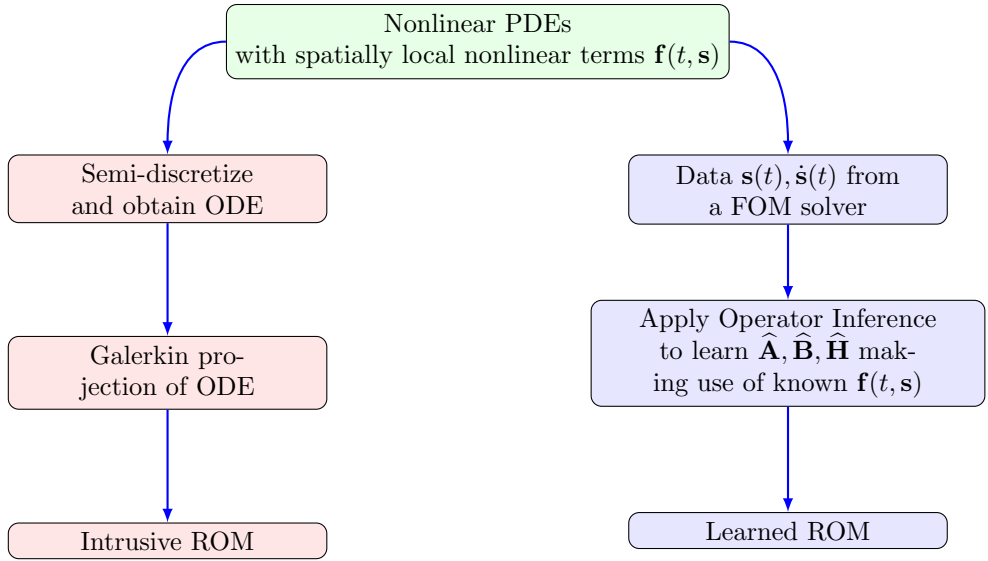
\begin{figure}[!tb]
	\centering
	\tikzsetnextfilename{FlowChart}%
	\input{FlowChart.tikz}%

	\caption{A comparison of a fully intrusive ROM approach with the proposed learning approach.}
	\label{fig:compare_PODDEIM_NonIntru}
\end{figure}

\begin{remark}
	We introduced the operator inference framework in the context of continuous-time nonlinear systems. However, the framework carries over directly to discrete time systems of the form 
	\begin{equation}
		\dstate_{k+1} = \bA \dstate_k + \bH (\dstate_k \otimes' \dstate_k) + \boldf(t_k, \dstate_k) + \bB \bu_k.
	\end{equation}	
	In this case, we replace $\Mstate$ with $\Mstate_0 = [\dstate_0, \dstate_1, \ldots, \dstate_{k-1}]$, set $\bU = [\bu_0, \bu_1, \ldots, \bu_{k-1}]$, and replace the time-derivative matrix $\dot{\Mstate}$ with $\Mstate_1 = [\dstate_1, \dstate_2, \ldots, \dstate_k]$.
\end{remark}

\subsection{Analysis of operator inference framework} \label{sec:OPINF_analysis}
We extend the results from \cite[Sec~3.2]{Peherstorfer16DataDriven} to show that under certain conditions on the time discretization, the learned operators converge to the intrusively obtained reduced operators. For this, we make two assumptions on the time discretization of the FOM and ROM.

\begin{assumption} \label{ass1}
	The time stepping scheme for the FOM~\eqref{eq:FOM} is convergent, i.e., 
	$$
	\max_{i \in \{ 1,\ldots T/\Delta t \}} \Vert \dstate_i - \dstate(t_i) \Vert_2 \rightarrow 0 \quad \text{as} \quad \Delta t\rightarrow 0.
	$$ 
\end{assumption}

\begin{assumption} \label{ass2} 
	The derivatives approximated from projected states, $\dot{\hat{\dstate}}_k$, converge to $\frac{\text{d}}{\text{d}t} \hat{\dstate}(t_k)$ as the discretization time step $\Delta t \rightarrow 0$, i.e., 
	$$
	\max_{i \in \{ 1,\ldots T/\Delta t \}} \Vert \dot{\hat{\dstate}}_i - \frac{\text{d}}{\text{d}t} \hat{\dstate}(t_i) \Vert_2 \rightarrow 0 \quad \text{as} \quad \Delta t\rightarrow 0.
	$$ 
\end{assumption}
Assumption~\ref{ass2} relates to the specific derivative approximation scheme that is used, its approximation order and the time-step size. In this work, we employ a fourth-order time-stepping scheme, see \Cref{sec:numerics} for details. Other choices of time discretizations are possible, see e.g.,~\cite{martins2013review,knowles2014methodsDifferentiation,chartrand2017numericalDifferentiation} and the references therein.

\begin{theorem}
	Let Assumptions~\ref{ass1}--\ref{ass2} hold and let a basis matrix $\bV = [ \bv_1, \bv_2, \ldots, \bv_r] \in \R^{n\times r}$ be given. Let $\widetilde{\bA}, \widetilde{\bB}, \widetilde{\bH}$ be the intrusively projected ROM operators from \eqref{eq:Prelim:RedOperators}. If the data matrix $\hat{\bD} = [\hat{\Mstate}, \ \bU, \ \hat{\Mstate} \otimes' \hat{\Mstate}]$ has full column rank, then for every $\varepsilon>0$, there exists $r\leq n$ and a time step size $\Delta t>0$ such that for the difference between the learned operators $\hat{\bA}, \hat{\bH}, \hat{\bB}$ and the (intrusive) projection-based $\widetilde{\bA}, \widetilde{\bH}, \widetilde{\bB}$, we have
	
	\begin{equation*}
		\Vert \hat{\bA} - \widetilde{\bA} \Vert_F \leq \varepsilon, \quad \Vert \hat{\bB} - \widetilde{\bB} \Vert_F \leq \varepsilon, \qquad \Vert \hat{\bH} - \widetilde{\bH} \Vert_F \leq \varepsilon.
	\end{equation*}
\end{theorem}

\begin{proof}
	The proof follows the ideas of the proof of \cite[Theorem~1]{Peherstorfer16DataDriven}. With the data matrices from \eqref{eq:Shat}--\eqref{Xdot}, we set $\hat{\bD} = [\widehat{\Mstate}, \ \bU, \ \widehat{\Mstate} \otimes' \widehat{\Mstate}]$ as the data matrix, $\widehat{\bR} = \dot{\widehat{\bS}} - \widehat{\bF}$ as the residual, and $\hat{\bO} = [ \widehat{\bA}, \widehat{\bB}, \widehat{\bH}]$ are the operators to be inferred, so the operator inference problem~\eqref{eqn:optimizationProb} can be formulated as
	\begin{equation} \label{eqn:optimizationProbDO}
		\min_{\hat \bO} \  \| \widehat{\bD} \widehat{\bO}^\top - \widehat{\bR}^\top \|_F.
	\end{equation} 	
	Denote by $\widetilde{\Mstate} = [ \widetilde{\dstate}_1,\widetilde{\dstate}_2, \ldots, \widetilde{\dstate}_k  ]$ the data matrix that is collected from taking $k$ snapshots of the intrusive ROM~\eqref{eq:ROM}, and let $\widetilde{\bD} = [\widetilde{\Mstate}, \ \bU, \ \widetilde{\Mstate} \otimes' \widetilde{\Mstate}]$ and $\widetilde{\bR} = \dot{\widetilde{\bS}} - \widetilde{\bF}$. With $\widetilde{\bD},\widetilde{\bR}$ as data, the projected operators $\widetilde{\bA}, \widetilde{\bB}, \widetilde{\bH}$ are a solution of the minimization problem~\eqref{eqn:optimizationProbDO}, and if $\widetilde{\bD}$ has full column rank, then the solution is also unique. 
	Next, we observe that $\widehat{\bS} = \widetilde{\bS} + \Delta\widetilde{\bS}$, i.e., the projected FOM data can be interpreted as a perturbation of the ROM-simulated data, and as the dimension $r$ increases, we have $\Vert \Delta \widetilde{\bS}  \Vert_F \rightarrow 0$ because of Assumption~\ref{ass1}, and the fact that for the limiting case $r=n$ the FOM~\eqref{eq:FOM} and ROM~\eqref{eq:ROM} are equivalent up to a change of coordinates. 
	Since $\widehat{\bS} \rightarrow \widetilde{\bS}$ we also have by the continuity of $\boldf_r$ that $\widetilde{\bF} \rightarrow \widehat{\bF}$.
	Moreover, Assumption~\ref{ass2} ensures convergence of $\dot{\hat{\dstate}}_k \rightarrow \frac{\text{d}}{\text{d}t} \widehat{\dstate}(t_k)$ and using again that for the limiting case $r=n$ the FOM~\eqref{eq:FOM} and ROM~\eqref{eq:ROM} are equivalent up to a change of coordinates, we have that $\frac{\text{d}}{\text{d}t} \widehat{\dstate}(t_k) \rightarrow \frac{\text{d}}{\text{d}t} \widetilde{\dstate}(t_k)$ as $r\rightarrow n$ so $\Vert \Delta \widetilde{\bR} \Vert_F \rightarrow 0$.  Taken together we have $\Vert \widetilde{\bR} + \Delta \widetilde{\bR} \Vert_F \rightarrow \Vert \widetilde{\bR} \Vert_F$. 
	Therefore, we have as a limiting result that
	\begin{equation}
		\min_{\widehat \bO} \left [ \lim_{\substack{r\rightarrow n  \\ \Delta t \rightarrow 0}}    \  \| \widehat{\bD} \widehat{\bO}^\top - \widehat{\bR}^\top \|_F \right ]
		= 
		\min_{\widehat \bO} \left [ \lim_{\substack{r\rightarrow n  \\ \Delta t \rightarrow 0}}    \  \| [\widetilde{\bD} + \Delta \widetilde{\bD} ]  \widehat{\bO}^\top - [\widetilde{\bR} + \Delta \widetilde{\bR}]^\top \|_F \right ]
		= 
		\min_{\widehat \bO}  \| \widetilde{\bD} \widehat{\bO}^\top - \widetilde{\bR}^\top \|_F, 
	\end{equation} 	
	so the learned operators $\widehat{\bO}$ converge to the intrusively projected operators $\widetilde{\bO}$. In the pre-asymptotic case, we then get the stated result in the theorem, which also uses the full-rank assumption on $\widehat{\bD}$ to deduce $\Vert \hat{\bA} - \widetilde{\bA} \Vert_F \leq \varepsilon, \ \Vert \hat{\bB} - \widetilde{\bB} \Vert_F \leq \varepsilon, \ \Vert \hat{\bH} - \widetilde{\bH} \Vert_F \leq \varepsilon$. 
\end{proof}	

The theorem shows that the learned ROM operators converge to the projected ROM operators as $r\rightarrow n$, in the coordinate system spanned by the columns of $\bV$ as $\widehat{\bA} \rightarrow \bV^\top \bA \bV$. However, we do not obtain a convergence rate, and in our practical experience there are examples where the difference in operators might not monotonically decrease for low orders of the ROM, such as the tubular reactor example in Section~\ref{sec:numerics_tubular}. Even for $r=50$ the relative error in the inferred linear system operators is $1.66\times 10^{-2}$ for the tubular reactor example. For the other examples in Sections~\ref{sec:numerics_Chaffee} and \ref{sec:numerics_batchChrom}, the norm did decrease for low orders of the ROM. Most importantly, however, we see in the numerical results in Section~\ref{sec:numerics} that in all cases we can obtain accurate learned ROMs with convergent state-errors for $r\ll n$ via the operator inference framework.

The cost of the proposed framework of operator inference for nonlinear systems is similar to a standard projection-based POD framework when $k<n$, i.e., we have more states than snapshots. Both the projection-based ROM and the learned ROM require computing the POD basis from the snapshot set as well as time integration of $r$-dimensional ROMs. The methods differ in that the operator inference requires projecting the data $\bS, \dot{\bS}$ onto the subspace $\bV$ (see step~4 of Algorithm~\ref{algo:OpInfNL}), forming $\bS\otimes'\bS$ and then solving a least-squares problem in reduced dimension. The cost scales as $\mathcal{O}(rn)$. In contrast, the projection-based ROM requires projecting the model matrices onto the subspace $\bV$, see \eqref{eq:Prelim:RedOperators}, which also scales as $\mathcal{O}(rn)$. Note, that in that last step we need to form $\bV^\top \bH (\bV\otimes' \bV)$ which is potentially expensive.

\subsection{Empirical interpolation} \label{sec:DEIM}
In this section, we discuss empirical interpolation \cite{morBarMNetal04,deim2010} to accelerate the evaluation of $\bV^\top \boldf (t, \bV \hat{\dstate})$ in the learned ROM. We employ the discrete empirical interpolation method (DEIM) from \cite{deim2010}. To this end, we approximate 
\begin{equation} \label{eq:DEIMapprox}
	\hat{\boldf}(t,\bV \hat{\dstate}(t)) \approx \hat{\boldf}_r (t,\bV \hat{\dstate}(t)) = \bV^\top \bW (\mathbb{S}^\top \bW)^{-1} \mathbb{S}^\top \boldf(t,\bV \hat{\dstate}(t)).
\end{equation}
The matrix $\bW$ is computed by taking the SVD of the nonlinear snapshot matrix $\bF$ in equation~\eqref{eq:F}, and setting $\bW$ to the leading $m$ left singular vectors of $\bF$. Here, $\mathbb{S}$ is an $n\times m$ matrix obtained by selecting certain columns of the $n\times n$ identity matrix. The choice of the selected entries in $\mathbb{S}$ follows \Cref{algo:DEIM}, which enforces interpolation $\mathbb{S}^\top \boldf(t,\cdot)  = \mathbb{S}^\top [\bW (\mathbb{S}^\top \bW)^{-1} \mathbb{S}^\top \boldf(t,\cdot) ]$ and simultaneously limits the local growth of the spectral norm $\Vert(\mathbb{S}^\top \bW)^{-1}\Vert_2$, as this term occurs in the DEIM error bound \cite{deim2010}.
\begin{algorithm}[!tb]
	\caption{Discrete empirical interpolation method \cite[Algorithm 1]{deim2010}}
	\begin{algorithmic}[1]
		\State {\bf Input:} $\bW = [\bw_1, \ldots, \bw_m]$.
		\State $p_1 = \arg \max_i( | \bw_1(i) | ); \ \bW_1 = [\bw_1]; \ \mathbb{S}_1 = [e_{p_1}]; \ \rho_1= [p_1]$
		\For{$j=2:m$}
		\State Solve $\mathbb{S}^\top_{j-1} \bW_{j-1} z = \mathbb{S}^\top _{j-1} \bw_j$ for $z$;
		\State  $r_j = \bw_j - \bW_{j-1} z;  \ p_j = \arg \max_i ( | r_j (i)| );$
		\State $U_j = [\bW_{j-1}, \bw_{j}]; \ \mathbb{S}_j = [\mathbb{S}_{j-1}, e_{p_j}]; \ \rho_j = (\rho_{j-1}, p_j)$;
		\EndFor
		\State  {\bf Output:}  Selection operator $\mathbb{S} = \mathbb{S}_m$
	\end{algorithmic}\label{algo:DEIM}
\end{algorithm}
The learned ROM with DEIM interpolation is written as
\begin{align}\label{eq:ROMDEIM}
	\dot{\hat{\dstate}}(t)& = \hat \bA\hat \dstate(t) + \hat \bH(\hat \dstate(t) \otimes' \hat \dstate(t)) + \hat \boldf_r(t,\hat \dstate(t)) +  \hat \bB\bu(t), 
\end{align}
where only $m\ll n$ entries of $\boldf$ are evaluated, thus achieving computational speedup.

\section{Numerical results} \label{sec:numerics}
In this section, we test the efficiency of the proposed operator inference approach by means of several problems, arising in different areas of science and engineering. We compare the quality of the learned ROM with an intrusive ROM. Here, we use a projection-based POD method. Below we give some information that applies to all our numerical simulations:
\begin{itemize}
\item All models are simulated using the routine \texttt{ode15s} in \matlab ~with relative error and absolute error tolerances of $10^{-10}$.
\item In the operator inference framework, we approximate the time derivative $\dot{\widehat{\dstate}}_i$ with a five-point stencil $\dot{\widehat{\dstate}}_i \approx (-{\widehat{\dstate}}_{i+2} + 8 {\widehat{\dstate}}_{i+1} - 8{\widehat{\dstate}}_{i-1} +{\widehat{\dstate}}_{i-2} )/(12\Delta t)$, which has fourth-order accuracy. The first two and last two time derivatives are computed using first-order forward and backward Euler approximations, respectively.
\end{itemize}

We present three examples with increasing complexity to illustrate our approach.

\subsection{Chafee-Infante model} \label{sec:numerics_Chaffee}
We consider the one-dimensional Chafee-Infante equation \cite{chafee1974bifurcation}, which is closely related to the Cahn-Hillard equation \cite{chen2004generation}.  The Chafee-Infante model has been studied extensively as a benchmark problem for nonlinear MOR techniques, and has been reduced by means of various methods, see, e.g.~\cite{morBenB15,morBenGG18}. 
The Chafee-Infante equation is a  diffusion-reaction model with the  governing equation, boundary and initial conditions as follows:
\begin{equation}\label{eq:ChafeeGovEq}
	\begin{aligned}
		\dfrac{dv}{dt} + v^3 &= \frac{\partial^2 v}{\partial x^2} +  u, \quad  &&(x,t) \in (0,1)\times (0,T], &\quad v(0,t) &= u(t), \quad &&t \in (0,T),\\
		\frac{\partial v}{\partial x}(1,t) &= 0, \quad  &&t \in (0,T), &\quad v(x,0) &= 0, \quad &&x \in (0,1).
	\end{aligned}
\end{equation}
As a quantity of interest, we observe the output at the right boundary, i.e., $y(t) = v(L,t)$. The model~\eqref{eq:ChafeeGovEq} contains a cubic nonlinear term, which fits our discussion of spatially local nonlinear terms of \Cref{sec:PDE}. Hence, we can use our proposed non-intrusive approach to construct a learned ROM. 

The governing equations are spatially discretized via a finite difference scheme using $500$ grid points. We generate snapshots of the discretized state $\dstate(t)$ via numerical simulations of the semi-discrete system in the time interval $T = \left[0,10\right]$s with time step $10^{-3}$.  We also collect snapshot data for the input $u(t) = 10(\sin(\pi t)+1)$. This allows us to build a snapshot matrix $\Mstate$ and input matrix $\bU$. Furthermore, we construct a nonlinear snapshot matrix $\bF$ from the snapshot matrix $\Mstate$ by applying the nonlinear function $\boldf$ from equation~\eqref{eq:def_f}, i.e., $\bF(:,i) = \boldf(\Mstate(:,i),t_i)$ where $(:,i)$ denotes the $i$th column of a matrix.

In \Cref{fig:chafee_decaySV}, we show the decay of the singular values of the snapshot matrix and nonlinear snapshot matrix, indicating exponential decay for both. By taking the dominant $12$  POD basis vectors, we construct an intrusive ROM by explicitly computing the projected matrices, and also employ our non-intrusive framework from \Cref{algo:OpInfNL} to obtain a learned ROM. 
For the operator inference method, we leave out the first $10$ recorded snapshots, as the transient response in the first 10 state vectors is very fast. A much smaller time step could solve this issue, but we are often faced with a situation where a data set with a fixed time step is given to us, and we have little control over the time step selection, which is the situation we consider here.
We employ DEIM approximation for the reduced nonlinear terms as discussed in \Cref{sec:DEIM}. We select DEIM basis vectors for the nonlinear term corresponding to relative singular values up to $10^{-8}$.

In \Cref{fig:chafee_response}, we compare the output of interest for the FOM and both ROMs for the same input. Note that the system is simulated for $T=20$s, thus predicting the output for 10s, i.e., 100\% past the training interval. We observe that the learned ROM is more accurate at most time steps. While this appears counter-intuitive at first, it might be possible that unresolved features (sometimes called closure terms) are learned by the ROM. This remains a topic of further investigation.

In \Cref{fig:chafee_stateOpComp}, we compare the state approximation error of both ROMs. Both POD and our learned ROM are accurate, with the learned ROM being slightly more accurate. We found that after $r = 14$, the non-intrusive approach starts yielding unstable ROMs. This could be due to the fact that the singular values corresponding to higher basis vectors are relatively small (order $10^{-10}$ in this example). Moreover, since the training data remains the same as $r$ is increased, there can be a situation where there is not enough information in the training data to learn a more accurate ROM that inherits the stability from the FOM. We do not observe instability in non-intrusive ROMs in general.
Nevertheless, at this point the ROMs already predict the output with accuracy $10^{-7}$, which is sufficient in most practical applications. 

\begin{figure}[!tb]
	\centering
	\setlength\fheight{3.0cm}
	\setlength\fwidth{.4\textwidth}
	{\footnotesize
	\tikzsetnextfilename{Chafee_DecaySV}%
	\input{Chafee_DecaySV.tikz}%

	}
	\caption{Chafee-Infante example: Decay of singular values of the snapshot matrix and the nonlinear snapshot matrix.}
	\label{fig:chafee_decaySV}
\end{figure}
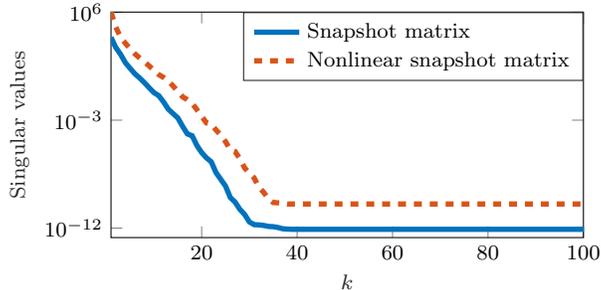

\begin{figure}[!tb]
	\centering
	\setlength\fheight{3.0cm}
	\setlength\fwidth{.4\textwidth}
	{\footnotesize
	\tikzsetnextfilename{Chafee_Response}%
	\input{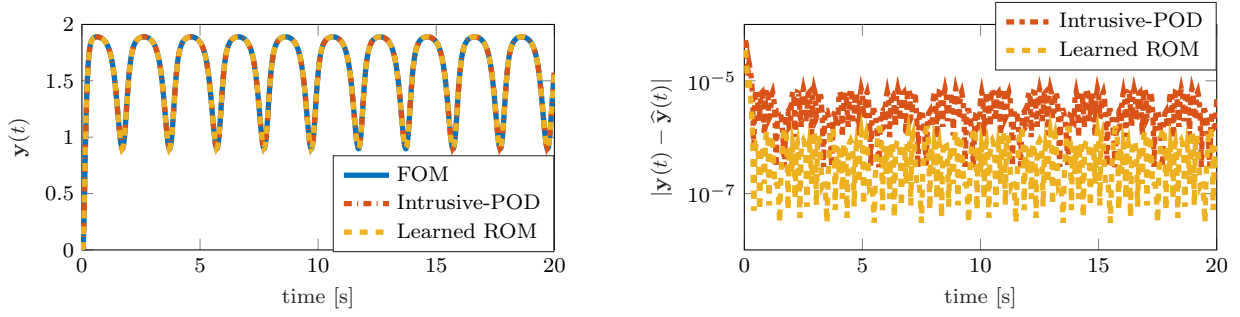}%
\hfill
	\tikzsetnextfilename{Chafee_Error}%
	\input{Chafee_Error.tikz}%

	}
	\caption{Chafee-Infante example: Output (left) and output errors (right) of the POD intrusive and the learned ROMs of order $r = 12$.}
	\label{fig:chafee_response}
\end{figure}

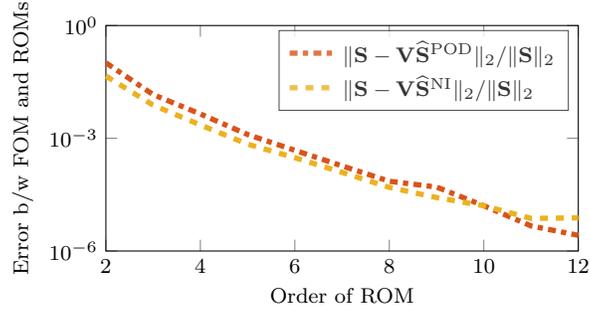
\begin{figure}[!tb]
	\centering
	\setlength\fheight{3.0cm}
	\setlength\fwidth{.4\textwidth}
	{\footnotesize
	\tikzsetnextfilename{Chafee_DiffFullState}%
	\input{Chafee_DiffFullState.tikz}%

		\caption{Chafee-Infante example: Convergence of relative state errors with increasing order of the POD intrusive ROMs and the learned ROMs.}
		\label{fig:chafee_stateOpComp}
	}
\end{figure}

\subsection{Tubular reactor model} \label{sec:numerics_tubular}
We consider a one-dimensional non-adiabatic tubular reactor model with a single reaction, describing the evolution of the species concentration $\psi(x,t)$ and temperature $\theta(x,t)$. The dynamics of the model are given by the PDEs~\cite{heinemann1981tubular} as follows:
\begin{subequations}
	\begin{align*}
		\frac{\partial \psi}{\partial t}   &= \dfrac{1}{\Pe} \frac{\partial^2 \psi}{\partial x^2} - \frac{\partial \psi}{\partial x} - \mathcal D \mathcal{F}(\psi,\theta;\gamma),\\
		\frac{\partial \theta}{\partial t } &= \dfrac{1}{\Pe}\frac{\partial^2 \theta}{\partial x^2} - \frac{\partial \theta}{\partial x}  -\beta(\theta-\theta_{\text{ref}}) + \mathcal B\mathcal D \mathcal{F}(\psi,\theta;\gamma),
	\end{align*}
\end{subequations} 
with spatial variable $x\in (0,1)$, time $t >0$ and Arrhenius reaction term 
\begin{equation} \label{eq:F_tubular}
	\mathcal{F}(\psi,\theta;\gamma) = \psi \exp  \left( \gamma - \dfrac{\gamma}{\theta} \right). 
\end{equation}
The parameters are the Damk{\"o}hler number $\cD$, the P{\'e}clet number  $\Pe$, and the reaction rate $\gamma$. We set $\mathcal{D}=0.167$, $ \Pe=5$, and  $\gamma=25$. The reference temperature is $\theta_{\text{ref}}(x,t) \equiv 1$,  and the constants are $\mathcal{B}=0.5$ and  $\beta=2.5$. Robin boundary conditions are imposed on the left boundary of the domain as follows:
\begin{equation*}
\frac{\partial \psi}{\partial x}(0,t) = \Pe (\psi(0,t)-1), \qquad \frac{\partial \theta}{\partial x}(0,t) = \Pe (\theta(0,t)-1), 
\end{equation*}
and the Neumann boundary conditions on the right boundary are
\begin{equation*}
\frac{\partial \psi}{\partial x}(1,t) = 0, \quad \frac{\partial \theta}{\partial x} (1,t) =0.
\end{equation*}
Moreover, the initial conditions are prescribed as
\begin{equation*}
	\psi(x,0) = \psi_0(x), \qquad \theta(x,0) = \theta_0(x).
\end{equation*}	
The quantity of interest is the temperature oscillation at the reactor exit, i.e., the quantity
\begin{equation*}
	y(t) = \theta(x=1,t).
\end{equation*}
The governing PDE is discretized with a finite difference approximation (see \cite{zhou2012thesis} for details) leading to a discretized state $\dstate(t) \in \mathbb{R}^{198}$ which takes the form
\begin{align}\label{eq:fullDiscretedModel}
	\dot \dstate(t) & = \bA \dstate(t) + \boldf(\dstate(t)) +  \bB.
\end{align}
Observe that in this example the input $\bu(t)$ is constant, i.e., $\bu(t) \equiv 1$. 
The Arrhenius nonlinear term~\eqref{eq:F_tubular} requires pointwise evaluations (local in space), which fits to our framework presented  in \Cref{sec:use_nonlin}. Thus, we can apply the proposed non-intrusive operator inference framework to obtain a ROM by including PDE-level information about the nonlinear terms into the learning framework. We build the snapshot matrix by collecting snapshots in the time interval $T  = \left(0,30\right]$s which are $\delta t = 10^{-3}$ apart. 

\Cref{fig:TR_decaySV} shows the decay of the singular values of the snapshot matrix $\bS$ and the nonlinear snapshot matrix $\bF$. 
\begin{figure}[!tb]
	{	\centering
		\setlength\fheight{3.0cm}
		\setlength\fwidth{.4\textwidth}
	\tikzsetnextfilename{decay_singularValues}%
	\input{decay_singularValues.tikz}%

		\caption{Tubular reactor example: Decay of the singular values of the snapshot matrix  and the nonlinear snapshot matrix.}
		\label{fig:TR_decaySV}
	}
\end{figure}
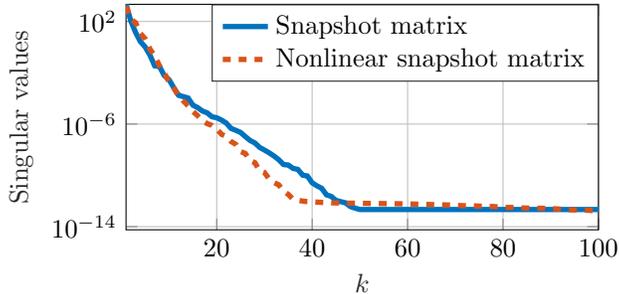
We use the leading $r=10$ POD modes in the projection matrix $\bV$. We construct ROMs both via explicitly computing the projected model terms, and via our non-intrusive approach. As in the previous example, we make use of DEIM to efficiently evaluate the reduced nonlinear terms in both ROMs. We select DEIM basis vectors for the nonlinear term corresponding to relative singular values up to $10^{-8}$.

Both ROMs are  simulated until $T = 60$s, which is 100\% longer than the training interval.   \Cref{fig:tubularresponse_a} shows the quantity of interest obtained from the FOM and the two ROMs. Note, that for this set of parameters, the model enters a self-excited limit-cycle oscillation. \Cref{fig:tubularresponse_b} shows the output error $| \by(t)- \hat\by(t) |$ for the two ROMs. Both models produce similar output errors. This competitiveness of the learned ROM with the projection-based POD model is remarkable, since after $t=30$s outputs are purely predictive as the ROMs exit the range of training data.  

\Cref{fig:stateOpComp} shows the relative error in the approximated states for both ROMs with increasing ROM order. The learned ROM is more accurate than the POD-ROM in most cases. The errors for both models level-off after $r = 18$ at a relative state error of $10^{-7}$, which is again accurate enough for most applications, certainly for the temperature field considered in this example.  

\begin{figure}[!tb]
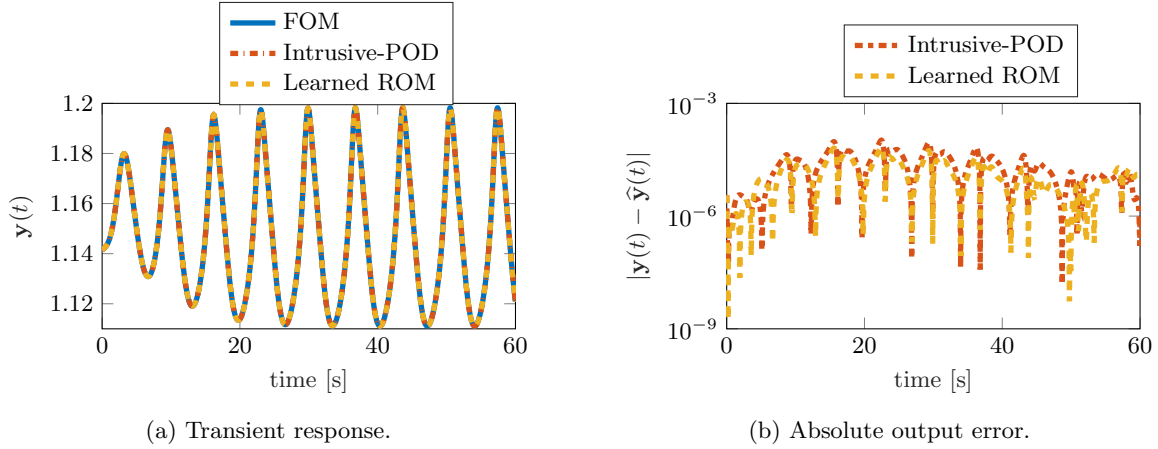

	{\small
		\centering
		\begin{subfigure}[t]{0.5\textwidth}
			\centering
			\setlength\fheight{3.0cm}
			\setlength\fwidth{.7\textwidth}
	\tikzsetnextfilename{outputs}%
	\input{outputs.tikz}%

			\caption{Transient response.}
			\label{fig:tubularresponse_a}
		\end{subfigure}%
		\begin{subfigure}[t]{0.5\textwidth}
			\centering
			\setlength\fheight{3.0cm}
			\setlength\fwidth{.7\textwidth}
	\tikzsetnextfilename{output_error}%
	\input{output_error.tikz}%

			\caption{Absolute output error.}
			\label{fig:tubularresponse_b}
		\end{subfigure}
	}
	\caption{Tubular reactor example: A comparison of the transient behavior of the projection-based POD ROM and the learned ROM.}
	\label{fig:tubularresponse}
\end{figure}

\begin{figure}[!tb]
	{\footnotesize
		\centering
		\setlength\fheight{3.0cm}
		\setlength\fwidth{.4\textwidth}
	\tikzsetnextfilename{state_error_convergence}%
	\input{state_error_convergence.tikz}%

		\caption{Tubular reactor example: Relative state errors produced by both ROMs, i.e., $\|\Mstate -\bV \hat{\Mstate}^{\text{POD}}\|/ \| \Mstate \|$ and $\|\Mstate -\bV \hat{\Mstate}^{\text{NI}}\|/ \| \Mstate \|$.}
		\label{fig:stateOpComp}
	}
\end{figure}
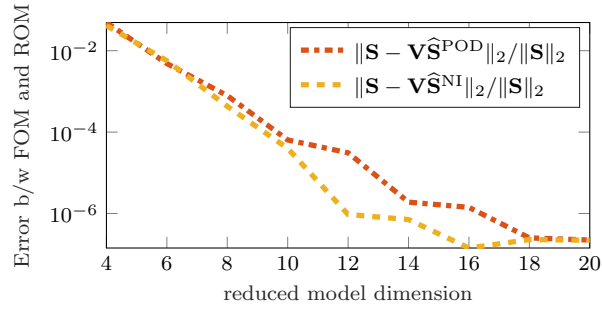

\subsection{Batch Chromatography} \label{sec:numerics_batchChrom}

In \Cref{sec:BC_setup} we describe the PDE model and its discretization. \Cref{sec:BC_intrusive} presents the structure-preserving ROM generation in the projection-based setting, and \Cref{sec:BC_nonintrusive} describes how we enforce the coupling structure in the non-intrusive learning framework. \Cref{sec:BC_results} presents our numerical results.

\subsubsection{Problem setup} \label{sec:BC_setup}
Chromatography is a separation and purification process used in the chemical and pharmaceutical industry, and we consider one of the simplest chromatography processes, the so-called batch chromatography.  For a detailed description of the process, we refer to \cite{guiochon2006fundamentals,morZhaFLetal14}. The primary principle of the batch chromatography process for binary separation is shown in~\Cref{fig:schematic_bc}. A mixture of products A and B is injected at the inlet of the column and then the feed mixture flows through the column. Since the to-be-separated solutes exhibit different adsorption affinities to the stationary phase, they move at different velocities, and thus separate from each other when exiting the column. At the column outlet, component A, which moves faster, is collected between $t_1$ and $t_2$, and later, component B is collected between $t_3$ and $t_4$. Here, time $t_1$ and $t_4$ are determined by a minimum concentration threshold that the detector can resolve, and time $t_2$ and $t_3$ are determined by the purity specifications imposed on the products. 
\begin{figure}[!tb]
	\centering
	\includegraphics[width = 9cm]{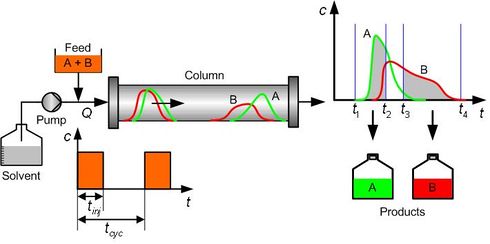}
	\caption{The figure shows a chromatography process, separating a mixture, c.f.~\cite{morwiki_bc}.}
	\label{fig:schematic_bc}
\end{figure}

The dynamics of the batch chromatographic column can be described by the following dimensionless coupled PDE system:
\begin{subequations}\label{eq:BC_pde}
	\begin{align}
		&\dfrac{\partial c_i}{\partial t} +  \dfrac{1-\epsilon}{\epsilon}\dfrac{\partial q_i}{\partial t} + \dfrac{\partial c_i}{\partial x} - \dfrac{1}{\Pe}\dfrac{\partial^2 c_i}{\partial x^2} = 0, \label{eq:BC_pde_c} \\ 
		&\dfrac{\partial q_i}{\partial t} =\dfrac{L}{Q/\epsilon A_c} \kappa_i \left(q_i^{Eq} - q_i\right),\label{eq:BC_pde_q}
	\end{align}
\end{subequations}
where $i \in \{1,2\}$, $x\in (0,L)$ is the spatial variable, and the liquid and solid phase concentrations of the compound $i$ are denoted by $c_i$ and $q_i$, respectively; the constants $\epsilon$, $L$ are the interstitial liquid velocity, the column porosity, and the column length, respectively; $\Pe$ denotes the P\'eclet number; $A_c$ denotes the cross-sectional area, defined as $\tfrac{\pi D^2}{4}$, where $D$ is the diameter of the tube. The constant $\kappa_{i}$ is the mass-transfer coefficient of component~$i$; moreover, $q^{\text{Eq}}_i$ is the adsorption equilibrium concentration calculated by the isotherm equation for the component~$i$, given as follows:
\begin{equation} \label{eq:BC_qi}
	q^{\text{Eq}}_i=\frac{H_{i,1}\,c_i}{1+\sum\limits_{j=1,2}K_{j,1}c_j^f\,c_j}+\frac{H_{i,2}\,c_i}{1+\sum\limits_{j=1,2}K_{j,2}c_j^f\,c_j},  
\end{equation}
where $H_{i,1}$ and $H_{i,2}$ are the Henry constants, and $K_{j,1}$ and $K_{j,2}$ the thermodynamic coefficients.  Furthermore, the boundary conditions are specified by the Danckwerts relations:
\begin{equation}
	\left.\frac{\partial c_i}{\partial x}\right|_{x=0} = \Pe\,(\left.c_i\right|_{x=0}-c^{\text{in}}_i), \quad\quad \left.\frac{\partial c_i}{\partial x}\right|_{x=L}=0,  
\end{equation}
where $c^{\text{in}}_i$ is the concentration of component $i$ at the inlet of the column.  A sigmoid-type injection profile is assumed, which is a differentiable approximation of a rectangular profile, given by
\begin{equation}
	c^{\text{in}}_i(t)= \dfrac{1}{1 + e^{-5(t-t_{\text{inj}  } )  }  },
\end{equation}
where $t_{\text{inj}}$ is the injection period. In addition, we have the following initial conditions:
\begin{equation}
	c_i(t=0,x)=q_i(t=0,x)=0,\quad x\in[0,\;L],\;i \in \{1,2\}.
\end{equation}

System \eqref{eq:BC_pde} is a coupled set of PDEs, and more importantly, \eqref{eq:BC_pde_c} involves the derivatives of $c_i$ and $q_i$. However, we can modify \eqref{eq:BC_pde_c} such that the resulting equation does not have a  $\partial q_i/\partial t$ term. For this, we set $\epsilon_c := \dfrac{\epsilon-1}{\epsilon}$ and insert \eqref{eq:BC_pde_q} by into \eqref{eq:BC_pde_c}, yielding
\begin{subequations}\label{eq:BC_pde1}
	\begin{align}
		\dfrac{\partial c_i}{\partial t} + \dfrac{\partial c_i}{\partial x} - \dfrac{1}{\Pe}\dfrac{\partial^2 c_i}{\partial x^2} & = \epsilon_c\kappa_i \left(q_i^{Eq} - q_i\right),\label{eq:BC_pde1c}\\
		\dfrac{\partial q_i}{\partial t} & = \kappa_i \left(q_i^{\text{Eq}} - q_i\right).\label{eq:BC_pde1q}
	\end{align}
\end{subequations}
A finite volume discretization of the governing equations yields a discretized model of the form:
\begin{equation}\label{eq:discretizedSys}
	\bbm \dot \bc_1 \\ \dot \bq_1 \\ \dot \bc_2 \\ \dot \bq_2 \ebm  = \bbm \bA_1 & 0 & 0 & 0 \\ 0 & 0 & 0 & 0 \\ 0 & 0 & \bA_2  & 0 \\ 0 & 0 & 0 & 0 \ebm \bbm \bc_1 \\ \bq_1 \\ \bc_2 \\ \bq_2 \ebm +  \bbm \bB \\ 0 \\ \bB \\ 0 \ebm u(t) + \bbm \epsilon_c \\ 1 \ebm \otimes \bbm \bff_1(\bc_1,\bq_1,\bc_2,\bq_2) \\ \bff_2(\bc_1,\bq_1,\bc_2,\bq_2) \ebm,
\end{equation}
where $\bc_1,\bq_1,\bc_2,\bq_2 \in \Rn$, $\bA_1,\bA_2\in \Rnn$, $\bB\in \Rn$, and $n$ is the number of degrees of freedom in the discretized field. The term $\bff_i(\bc_1,\bq_1,\bc_2,\bq_2)\in \Rn$ is the evaluation of the nonlinear term $\kappa_i \left(q_i^{\text{Eq}} - q_i\right)$ on the right-hand side of \eqref{eq:BC_pde1q} on the spatial grid. Note that this nonlinear term involves a rational function, see \eqref{eq:BC_qi}, which again translates to a nonlinear term that is local in space, and hence the nonlinear function fits the setting described in \Cref{sec:PDE}.

\subsubsection{Maintaining the coupling structure: Projection-based ROM} \label{sec:BC_intrusive}
Classically, to obtain the basis for POD, one would stack all variables in one vector, i.e., $[\bc_1^\top, \bq_1^\top,\bc_2^\top,\bq_2^\top]^\top$ and determine the common subspace $\bV$ by taking the SVD of that data. However, by doing so, we observe that both methods produce unstable ROMs. This is potentially due to the fact that the governing PDEs are highly coupled and the resulting ROMs do not preserve the coupling topology structure. We present an approach that uses the knowledge about the topology structure of a coupled system in the ROM. 

To preserve the coupled topology structure, we first compute a basis for each subsystem, i.e., $\bc_i \approx \bV_{\bc_i} \hat \bc_i$ and $\bq_i \approx \bV_{\bq_i} \hat \bq_i$, $i \in \{1,2\}$.  Next, we construct a projection matrix with a block structure, i.e., we approximate
\begin{equation} \label{eq:BC_V}
	\begin{bmatrix} \bc_1\\ \bq_1\\\bc_2\\\bq_2\end{bmatrix} \approx \bbm \bV_{\bc_1}  & & & \\ & \bV_{\bq_1} & & \\ & & \bV_{\bc_2} & \\ & & & \bV_{\bq_2} \ebm \bbm \hat{\bc}_{1} \\ \hat{\bq}_{1}\\\hat{\bc}_{2}\\\hat{\bq}_{2}\ebm.
\end{equation}
With the above projection matrix $\bV$, a projection-based ROM preserves this structure, i.e., 
\begin{equation}\label{eq:discretizedSysROM}
	\bbm \dot{\hat{\bc}}_1 \\ \dot{\hat{\bq}}_1 \\ \dot{\hat{\bc}}_2 \\ \dot{\hat{\bq}}_2 \ebm  = \bbm \hat{\bA}_1 & 0 & 0 & 0 \\ 0 & 0 & 0 & 0 \\ 0 & 0 & \hat{\bA}_2  & 0 \\ 0 & 0 & 0 & 0 \ebm \bbm \hat{\bc}_1 \\ \hat{\bq}_1 \\ \hat{\bc}_2 \\ \hat{\bq}_2 \ebm +  \bbm \hat{\bB}_1 \\ 0 \\ \hat{\bB}_2 \\ 0 \ebm u(t) + \bbm \epsilon_c \\ 1 \ebm \otimes \bbm \hat{\bff}_1(\hat{\bc}_1,\hat{\bq}_1,\hat{\bc}_2,\hat{\bq}_2) \\ \hat{\bff}_2(\hat{\bc}_1,\hat{\bq}_1,\hat{\bc}_2,\hat{\bq}_2) \ebm,
\end{equation}
where $\hat{\bc}_1,\hat{\bq}_1,\hat{\bc}_2,\hat{\bq}_2 \in \Rr$, $\hat{\bA}_1,\hat{\bA}_2\in \Rrr$, $\hat{\bB}_1,\hat{\bB}_2 \in \Rr$.
Note, that the dimension of the coupled ROM~\eqref{eq:discretizedSysROM} is $4r$, however, it has sparse structure, e.g., the linear system matrix has only $2r^2$ nonzero entries and the linear input matrix has only $2r$ nonzero entries. Moreover, it is also possible to use different ROM dimensions for each of the four physical states. Thus, enforcing the coupling structure does not necessarily lead to drastically more expensive ROMs.

\begin{remark} \label{rem:Stable}
	Maintaining the coupling structure in a ROM  of the  aggregate system is important for physical interpretability, and can have numerical implications as well~\cite{morBenF15,liao2007important,morReiS07,reis2008survey}. 
	For instance, structure-preserving ROMs can  lead to convergent feedback controllers for coupled systems~\cite{kramer16MORcontrolBurgers} and have implications on stability.
	We observe in our numerical experiments on the batch chromatography example, that unstable ROMs arise when we do not preserve this structure. Conversely, when the coupling structure of the FOM is maintained in the ROMs, we obtain stable ROMs. A theoretical justification of this observation for nonlinear systems remains an open problem, though, for linear systems, stability analysis for coupled systems is performed in \cite{morReiS07}. 
\end{remark}

\subsubsection{Maintaining the coupling structure: Learned ROM} \label{sec:BC_nonintrusive}
Having derived the form of the projection-based structure-preserving ROM leads to a strategy to enforce the coupling structure in the operator inference framework as follows. First, we compute the block-diagonal projection matrix $\bV$ from~\eqref{eq:BC_V}. Second, we collect projected state and time-derivative data:
\begin{equation}
	\hat{\Mstate} 
	= \begin{bmatrix} \bV_{\bc_1}^\top \bC_1  \\ \bV_{\bq_1}^\top \bQ_1 \\\bV_{\bc_2}^\top \bC_2 \\ \bV_{\bq_2}^\top \bQ_2 \end{bmatrix} 
	=: \begin{bmatrix} \hat{\bC}_1 \\ \hat{\bQ}_1 \\ \hat{\bC}_2 \\ \hat{\bQ}_2\end{bmatrix},
	\qquad 
	\dot{\hat{\Mstate}}
	= \begin{bmatrix} \bV_{\bc_1}^\top \dot\bC_1  \\ \bV_{\bq_1}^\top \dot\bQ_1 \\ \bV_{\bc_2}^\top \dot\bC_2 \\ \bV_{\bq_2}^\top \dot\bQ_2 \end{bmatrix} 
	=: \begin{bmatrix} \dot{\hat{\bC}}_1 \\ \dot{\hat{\bQ}}_1 \\ \dot{\hat{\bC}}_2 \\ \dot{\hat{\bQ}}_2\end{bmatrix}.
\end{equation}
The input matrix $\bU$ is assembled as given in~\eqref{eq:Xdot_X} and the nonlinear term matrix $\bF$ is built in a similar way as given in \eqref{eq:F}. We then solve \textit{separate} least-squares problems of the form 
\begin{equation}\label{eq:separateLS}
	\min_{\hat \bA_{i},\hat{\bB_i}}  \ \| \dot{\hat\bC}_i -\epsilon_c\bV_{\bc_i}^\top \bF - \hat \bA_{i} \hat{\bC}_i - \hat \bB_{i} \bU \|_F,\quad i \in \{1,2\}.
\end{equation}
Recall the structure of the projection-based ROM in \eqref{eq:discretizedSysROM}. For $\dot{\hat\bq}_i$, $i \in \{1,2\}$, there are no linear terms on the right hand side, and so there is no need to solve a least-squares problem. The expressions for $\dot{\hat\bq}_i$ are
\begin{align*}
	\dot{\hat\bq}_i = \boldf_i(\hat\bc_1,\hat\bq_1,\hat\bc_2,\hat\bq_2) =  \bV^\top_{\bq_i}\boldf_i(\bV_{\bc_1}\bc_1,\bV_{\bq_1}\bq_1,\bV_{\bc_2}\bc_2,\bV_{\bq_2}\bq_2).
\end{align*}

\subsubsection{Results} \label{sec:BC_results}

%

For the numerical experiment, we choose the model parameters as follows: $\epsilon=0.4$, $L=10.5$, $\Pe =2000$, $D = 2.6$, $\kappa_{1}=\kappa_{2}=0.1$, $H_{1,1} = 3.728,  H_{1,2} = 0.3, H_{2,1}= 2.688,  H_{2,2} = 0.1$, $K_{1,1} = 46.6,  K_{1,2} = 33.6, K_{2,1}= 3000,  K_{2,2} =1000$, $Q=0.1018$, $c_1^f = c_2^f =  2.9\times 10^{-3}$, and $t_\text{inj}=1.3$s. For spatial discretization we choose $n= 400$ (thus the overall state is $4n=1600$). 
The model is simulated until $T=10$s and snapshots are collected every $\delta t = 5\times 10^{-5}$s.
Such a fine time stepping allows us to obtain an accurate approximation of the time derivative; however, we take every $100$th snapshot of the state and time derivative while inferring the operators in  the least squares procedure.
\Cref{fig:BC_decaySV} shows the decay of the singular values of each component, showing a rather slow decay, which can be expected due to the transport nature of the problem. Next, we compute ROMs via intrusive POD and our non-intrusive approach with the same matrix $\bV$ as shown in~\eqref{eq:BC_V}, where each component is reduced to order $r = 22$. As in the previous two examples, we incorporate a DEIM approximation of the reduced nonlinear terms for both ROMs. We select DEIM basis vectors for the nonlinear term corresponding to relative singular values up to $10^{-10}$.
\Cref{fig:BC_response}, left, shows the quantity of interest obtained from the time-domain simulations of the FOM and the two ROMs.  \Cref{fig:BC_response}, right, shows the (mean) relative error of the two outputs. The results indicate that the non-intrusive and intrusive approaches yield ROMs of comparative accuracy. 
\Cref{fig:BC_stateOpComp} shows the relative error in the approximated states for both ROMs with increasing ROM order. For ROM orders of $r<30$, both the intrusive POD ROM and the learned ROMs perform similarly. However, as $r$ increases further, the error for the non-intrusive models does not decrease as favorably as for POD intrusive models. This could be due to the condition number of the operator inference problem \eqref{eq:ROM_OPINF}. Adding regularization to the least-squares problem could improve this as seen in~\cite{SKHW2020_learning_ROMs_combustor} and also using data from several different simulations, similar to \cite[Sec 3.6]{Peherstorfer16DataDriven}.

\begin{figure}[!tb]
	{	\centering
		\setlength\fheight{3.0cm}
		\setlength\fwidth{.4\textwidth}
	\tikzsetnextfilename{BC_DecaySV}%
	\input{BC_DecaySV.tikz}%

		\caption{Batch chromatography example: Decay of singular values of the collected snapshots for all four components separately and the nonlinear snapshot matrix.}
		\label{fig:BC_decaySV}
	}
\end{figure}
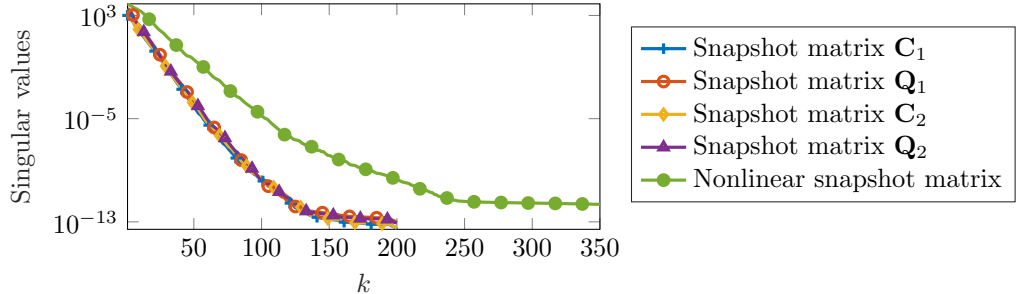

\begin{figure}[!tb]
	{\footnotesize
		\centering
		\setlength\fheight{3.0cm}
		\setlength\fwidth{.4\textwidth}
	\tikzsetnextfilename{BC_Response}%
	\input{BC_Response.tikz}%
\hfill
	\tikzsetnextfilename{BC_Error}%
	\input{BC_Error.tikz}%

	}
	\caption{Batch chromatography example: A comparison of the POD intrusive model with the learned model of order $r = 4\times 22$.}
	\label{fig:BC_response}
	
\end{figure}
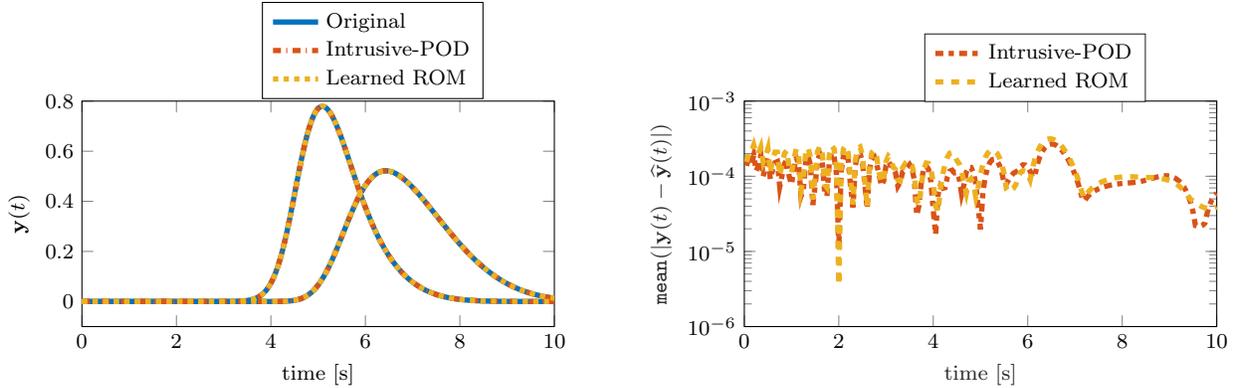

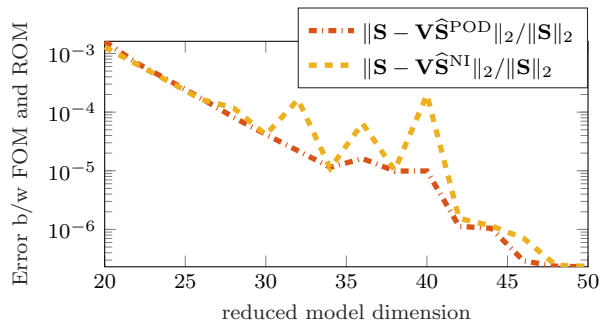
\begin{figure}[!tb]
	\centering
	\setlength\fheight{3.0cm}
	\setlength\fwidth{.4\textwidth}
	{\footnotesize
	\tikzsetnextfilename{BC_DiffFullState_approxResidual}%
	\input{BC_DiffFullState_approxResidual.tikz}%

	}
	\caption{Batch Chromatography example: Relative errors $\|\Mstate -\bV \hat{\Mstate}^{\text{POD}}\|/ \| \Mstate \| $ and $\|\Mstate - \bV\hat{\Mstate}^{\text{NI}}\| / \| \Mstate \|$.}
	\label{fig:BC_stateOpComp}
\end{figure}

\section{Conclusions} \label{sec:conclusion}
We have presented a data-driven model reduction method that non-intrusively learns reduced models from data collected of the full model states. The presented method applies to non-polynomial nonlinear models that are spatially local, and extends previous results for the polynomial case~\cite{Peherstorfer16DataDriven}. The method exploits the spatially local structure of the nonlinear terms to generate nonlinear snapshots from state data without requiring access to the discretization of the nonlinear term explicitly. Moreover, we presented a convergence result that shows that under mild assumptions on the time stepping scheme, and if the time step of the sampled data is sufficiently small, the non-intrusively learned reduced models converge to the same reduced models as obtained with intrusive model reduction methods. The numerical results show that the learned reduced models provide accurate predictions when compared to the full-model predictions. The three numerical experiments with reactive flows and chemical purification processes demonstrate that the method works well for strongly nonlinear systems, and further information in form of known structure can also be embedded into the learning framework. 
Future research directions motivated by this work are: addressing the effect of measurement noise (from real data) on the learning procedure and ultimately the learned ROMs; developing a suitable regularization strategy (possibly problem dependent) that could improve the learned ROMs accuracy and stability; and including additional information, such as physical constraints in the least-squares problem as suggested by \cite{loiseau_brunton_2018}.

\section*{Acknowledgments}
This work was supported in part by the US Air Force Center of Excellence on Multi-Fidelity Modeling of Rocket Combustor Dynamics award FA9550-17-1-0195, the Air Force Office of Scientific Research MURI on managing multiple information sources of multi-physics systems awards FA9550-15-1-0038 and FA9550-18-1-0023, and the AEOLUS center under US Department of Energy Applied Mathematics MMICC award  DE-SC0019303. The fourth author was partially supported by the US Department of Energy, Office of Advanced Scientific Computing Research, Applied Mathematics Program, DOE Award DESC0019334 and the National Science Foundation under Grant No. 1901091.

\bibliographystyle{elsarticle-num}
\bibliography{bibliography_OpInf_localized.bib,mor.bib}

\end{document}

%% file: ListPackages.tex
\usepackage{amsmath}
\usepackage{array}
\usepackage{url}
\usepackage{cleveref}
\usepackage{amsfonts}
\usepackage{amsthm}
\usepackage{graphicx}
\usepackage{algorithm} 
\usepackage{algpseudocode}
\usepackage{tikz}
\usepackage{pgfplots}
\usepackage{subcaption}
\usepackage{soul}

\usetikzlibrary{external}
\usetikzlibrary{trees, decorations,positioning,external}
\usetikzlibrary{shapes.misc, fit}
\usetikzlibrary{calc}
\pgfplotsset{compat = newest}

\newcommand*{\vertbar}{\rule[-1ex]{0.5pt}{2.5ex}}

\newcommand{\Pe}{\ensuremath{\mathop{\mathrm{Pe}}}}

\ifpdf
  \DeclareGraphicsExtensions{.eps,.pdf,.png,.jpg}
\else
  \DeclareGraphicsExtensions{.eps}
\fi

\title{{\TheTitle}\thanks{Submitted to the editors \today.
}}

\usepackage{amsopn}
\usepackage[software,hardware]{mymacros}

\newcommand{\bbm}{\begin{bmatrix}}
	\newcommand{\ebm}{\end{bmatrix}}
\newcommand{\bpm}{\begin{pmatrix}}
	\newcommand{\epm}{\end{pmatrix}}

\def\addlegendimage{\csname pgfplots@addlegendimage\endcsname}

\newlength\figureheight
\newlength\figurewidth
\newlength\fheight
\newlength\fwidth

\usepackage{adjustbox}
\usepackage{pdflscape}

\usepackage[ntheorem, fleqn]{empheq}
\usepackage[american]{babel}
\usepackage[title,titletoc,toc]{appendix}

\usepackage{nicefrac}
\newtheorem{theorem}{Theorem}
\newtheorem{example}{Example}
\newtheorem{remark}{Remark}
\newtheorem{assumption}{Assumption}
\usepackage{enumitem}
\renewcommand\hat[1]{\widehat#1}
\renewcommand\tilde[1]{\widetilde#1}

\makeatletter
\newcommand{\pushright}[1]{\ifmeasuring@#1\else\omit\hfill$\displaystyle#1$\fi\ignorespaces}
\newcommand{\pushleft}[1]{\ifmeasuring@#1\else\omit$\displaystyle#1$\hfill\fi\ignorespaces}
\makeatother

%% file: FlowChart.tikz
\begin{tikzpicture}
  \node (NL_PDE) [rectangle,rounded corners,draw,align = center, fill = green!10] at (0,0) {Nonlinear PDEs \\ with spatially local nonlinear terms $\boldf(t,\bs)$};
  \node (Discrete) [rectangle,rounded corners,draw, below left = 1cm and 2cm, text width = 4cm,align=center,fill = red!10] at (NL_PDE.south) {Semi-discretize and obtain ODE};
  \node (PODEIM) [rectangle,rounded corners,draw, below = 1.5cm, text width = 4cm,align=center,fill = red!10] at (Discrete.south) {Galerkin projection of ODE};
  \node (ROM1) [rectangle,rounded corners,draw, below = 1.5cm, text width = 4cm,align=center,fill = red!10] at (PODEIM.south) {Intrusive ROM};
  \node (Data) [rectangle,rounded corners,draw, below right = 1cm and 2cm, text width = 4cm,align=center,fill = blue!10] at (NL_PDE.south) {Data ${\bs}(t), \dot{\bs}(t)$ from a FOM solver};
  \node (NonIntru) [rectangle,rounded corners,draw, below  = 1cm, text width = 5cm,align=center,fill = blue!10] at (Data.south) {Apply Operator Inference to learn $\hat \bA, \hat \bB, \hat \bH$ making use of known $\boldf(t,\bs)$};
  \node (ROM2) [rectangle,rounded corners,draw, below = 1.4cm, text width = 4cm,align=center,fill = blue!10] at (NonIntru.south) {Learned ROM};
  \draw[blue,thick,-latex] (NL_PDE.west) to[out=180,in=90] (Discrete.north);
  \draw[blue,thick,-latex] (Discrete.south) -- (PODEIM.north);
  \draw[blue,thick,-latex] (PODEIM.south) -- (ROM1.north);
  \draw[blue,thick,-latex] (NL_PDE.east) to[out=0,in=90] (Data.north);
  \draw[blue,thick,-latex] (Data.south) -- (NonIntru.north);
  \draw[blue,thick,-latex] (NonIntru.south) -- (ROM2.north);
\end{tikzpicture}

%% file: Chafee_DecaySV.tikz
%
%
\definecolor{mycolor1}{rgb}{0.00000,0.44700,0.74100}%
\definecolor{mycolor2}{rgb}{0.85000,0.32500,0.09800}%
\begin{tikzpicture}

\begin{axis}[%
width=0.951\fwidth,
height=\fheight,
at={(0\fwidth,0\fheight)},
scale only axis,
xmin=1,
xmax=100,
xlabel style={font=\color{white!15!black}},
xlabel={$k$},
ymode=log,
ymin=1.61842621189803e-13,
ymax=1023815.27315543,
yminorticks=true,
ylabel style={font=\color{white!15!black}},
ylabel={Singular values},
axis background/.style={fill=white},
legend style={legend cell align=left, align=left, draw=white!15!black, at = { (1,1), anchor = south east }}
]
\addplot [color=mycolor1, line width=2.0pt]
  table[row sep=crcr]{%
1	8059.69573560061\\
2	1149.12155407492\\
3	310.149462065406\\
4	62.7346767987702\\
5	20.9215121134137\\
6	7.31098946500739\\
7	3.28086894653317\\
8	1.34031609255355\\
9	0.498207812567894\\
10	0.196641402169408\\
11	0.11109748422813\\
12	0.0333395790585518\\
13	0.00821849694244576\\
14	0.00352190698573429\\
15	0.00153997736617486\\
16	0.000306205682881312\\
17	7.10745935188658e-05\\
18	5.05475105463467e-05\\
19	7.71167317234559e-06\\
20	1.98030416503835e-06\\
21	7.19491465588576e-07\\
22	3.41372548228592e-07\\
23	4.27726114183263e-08\\
24	1.1702399260409e-08\\
25	3.1993771344963e-09\\
26	3.54386555562809e-10\\
27	1.48701217850809e-10\\
28	3.68378914771427e-11\\
29	1.38094642474593e-11\\
30	3.25100801540158e-12\\
31	1.93651116144266e-12\\
32	1.85142221721897e-12\\
33	1.66133252010365e-12\\
34	1.37793359195135e-12\\
35	1.31214136146493e-12\\
36	1.11876060887826e-12\\
37	8.8095494191727e-13\\
38	8.25850315688784e-13\\
39	7.93897105705295e-13\\
40	7.93897105705295e-13\\
41	7.93897105705295e-13\\
42	7.93897105705295e-13\\
43	7.93897105705295e-13\\
44	7.93897105705295e-13\\
45	7.93897105705295e-13\\
46	7.93897105705295e-13\\
47	7.93897105705295e-13\\
48	7.93897105705295e-13\\
49	7.93897105705295e-13\\
50	7.93897105705295e-13\\
51	7.93897105705295e-13\\
52	7.93897105705295e-13\\
53	7.93897105705295e-13\\
54	7.93897105705295e-13\\
55	7.93897105705295e-13\\
56	7.93897105705295e-13\\
57	7.93897105705295e-13\\
58	7.93897105705295e-13\\
59	7.93897105705295e-13\\
60	7.93897105705295e-13\\
61	7.93897105705295e-13\\
62	7.93897105705295e-13\\
63	7.93897105705295e-13\\
64	7.93897105705295e-13\\
65	7.93897105705295e-13\\
66	7.93897105705295e-13\\
67	7.93897105705295e-13\\
68	7.93897105705295e-13\\
69	7.93897105705295e-13\\
70	7.93897105705295e-13\\
71	7.93897105705295e-13\\
72	7.93897105705295e-13\\
73	7.93897105705295e-13\\
74	7.93897105705295e-13\\
75	7.93897105705295e-13\\
76	7.93897105705295e-13\\
77	7.93897105705295e-13\\
78	7.93897105705295e-13\\
79	7.93897105705295e-13\\
80	7.93897105705295e-13\\
81	7.93897105705295e-13\\
82	7.93897105705295e-13\\
83	7.93897105705295e-13\\
84	7.93897105705295e-13\\
85	7.93897105705295e-13\\
86	7.93897105705295e-13\\
87	7.93897105705295e-13\\
88	7.93897105705295e-13\\
89	7.93897105705295e-13\\
90	7.93897105705295e-13\\
91	7.93897105705295e-13\\
92	7.93897105705295e-13\\
93	7.93897105705295e-13\\
94	7.93897105705295e-13\\
95	7.93897105705295e-13\\
96	7.93897105705295e-13\\
97	7.93897105705295e-13\\
98	7.93897105705295e-13\\
99	7.93897105705295e-13\\
100	7.93897105705295e-13\\
101	7.93897105705295e-13\\
102	7.93897105705295e-13\\
103	7.93897105705295e-13\\
104	7.93897105705295e-13\\
105	7.93897105705295e-13\\
106	7.93897105705295e-13\\
107	7.93897105705295e-13\\
108	7.93897105705295e-13\\
109	7.93897105705295e-13\\
110	7.93897105705295e-13\\
111	7.93897105705295e-13\\
112	7.93897105705295e-13\\
113	7.93897105705295e-13\\
114	7.93897105705295e-13\\
115	7.93897105705295e-13\\
116	7.93897105705295e-13\\
117	7.93897105705295e-13\\
118	7.93897105705295e-13\\
119	7.93897105705295e-13\\
120	7.93897105705295e-13\\
121	7.93897105705295e-13\\
122	7.93897105705295e-13\\
123	7.93897105705295e-13\\
124	7.93897105705295e-13\\
125	7.93897105705295e-13\\
126	7.93897105705295e-13\\
127	7.93897105705295e-13\\
128	7.93897105705295e-13\\
129	7.93897105705295e-13\\
130	7.93897105705295e-13\\
131	7.93897105705295e-13\\
132	7.93897105705295e-13\\
133	7.93897105705295e-13\\
134	7.93897105705295e-13\\
135	7.93897105705295e-13\\
136	7.93897105705295e-13\\
137	7.93897105705295e-13\\
138	7.93897105705295e-13\\
139	7.93897105705295e-13\\
140	7.93897105705295e-13\\
141	7.93897105705295e-13\\
142	7.93897105705295e-13\\
143	7.93897105705295e-13\\
144	7.93897105705295e-13\\
145	7.93897105705295e-13\\
146	7.93897105705295e-13\\
147	7.93897105705295e-13\\
148	7.93897105705295e-13\\
149	7.93897105705295e-13\\
150	7.93897105705295e-13\\
151	7.93897105705295e-13\\
152	7.93897105705295e-13\\
153	7.93897105705295e-13\\
154	7.93897105705295e-13\\
155	7.93897105705295e-13\\
156	7.93897105705295e-13\\
157	7.93897105705295e-13\\
158	7.93897105705295e-13\\
159	7.93897105705295e-13\\
160	7.93897105705295e-13\\
161	7.93897105705295e-13\\
162	7.93897105705295e-13\\
163	7.93897105705295e-13\\
164	7.93897105705295e-13\\
165	7.93897105705295e-13\\
166	7.93897105705295e-13\\
167	7.93897105705295e-13\\
168	7.93897105705295e-13\\
169	7.93897105705295e-13\\
170	7.93897105705295e-13\\
171	7.93897105705295e-13\\
172	7.93897105705295e-13\\
173	7.93897105705295e-13\\
174	7.93897105705295e-13\\
175	7.93897105705295e-13\\
176	7.93897105705295e-13\\
177	7.93897105705295e-13\\
178	7.93897105705295e-13\\
179	7.93897105705295e-13\\
180	7.93897105705295e-13\\
181	7.93897105705295e-13\\
182	7.93897105705295e-13\\
183	7.93897105705295e-13\\
184	7.93897105705295e-13\\
185	7.93897105705295e-13\\
186	7.93897105705295e-13\\
187	7.93897105705295e-13\\
188	7.93897105705295e-13\\
189	7.93897105705295e-13\\
190	7.93897105705295e-13\\
191	7.93897105705295e-13\\
192	7.93897105705295e-13\\
193	7.93897105705295e-13\\
194	7.93897105705295e-13\\
195	7.93897105705295e-13\\
196	7.93897105705295e-13\\
197	7.93897105705295e-13\\
198	7.93897105705295e-13\\
199	7.93897105705295e-13\\
200	7.93897105705295e-13\\
201	7.93897105705295e-13\\
202	7.93897105705295e-13\\
203	7.93897105705295e-13\\
204	7.93897105705295e-13\\
205	7.93897105705295e-13\\
206	7.93897105705295e-13\\
207	7.93897105705295e-13\\
208	7.93897105705295e-13\\
209	7.93897105705295e-13\\
210	7.93897105705295e-13\\
211	7.93897105705295e-13\\
212	7.93897105705295e-13\\
213	7.93897105705295e-13\\
214	7.93897105705295e-13\\
215	7.93897105705295e-13\\
216	7.93897105705295e-13\\
217	7.93897105705295e-13\\
218	7.93897105705295e-13\\
219	7.93897105705295e-13\\
220	7.93897105705295e-13\\
221	7.93897105705295e-13\\
222	7.93897105705295e-13\\
223	7.93897105705295e-13\\
224	7.93897105705295e-13\\
225	7.93897105705295e-13\\
226	7.93897105705295e-13\\
227	7.93897105705295e-13\\
228	7.93897105705295e-13\\
229	7.93897105705295e-13\\
230	7.93897105705295e-13\\
231	7.93897105705295e-13\\
232	7.93897105705295e-13\\
233	7.93897105705295e-13\\
234	7.93897105705295e-13\\
235	7.93897105705295e-13\\
236	7.93897105705295e-13\\
237	7.93897105705295e-13\\
238	7.93897105705295e-13\\
239	7.93897105705295e-13\\
240	7.93897105705295e-13\\
241	7.93897105705295e-13\\
242	7.93897105705295e-13\\
243	7.93897105705295e-13\\
244	7.93897105705295e-13\\
245	7.93897105705295e-13\\
246	7.93897105705295e-13\\
247	7.93897105705295e-13\\
248	7.93897105705295e-13\\
249	7.93897105705295e-13\\
250	7.93897105705295e-13\\
251	7.93897105705295e-13\\
252	7.93897105705295e-13\\
253	7.93897105705295e-13\\
254	7.93897105705295e-13\\
255	7.93897105705295e-13\\
256	7.93897105705295e-13\\
257	7.93897105705295e-13\\
258	7.93897105705295e-13\\
259	7.93897105705295e-13\\
260	7.93897105705295e-13\\
261	7.93897105705295e-13\\
262	7.93897105705295e-13\\
263	7.93897105705295e-13\\
264	7.93897105705295e-13\\
265	7.93897105705295e-13\\
266	7.93897105705295e-13\\
267	7.93897105705295e-13\\
268	7.93897105705295e-13\\
269	7.93897105705295e-13\\
270	7.93897105705295e-13\\
271	7.93897105705295e-13\\
272	7.93897105705295e-13\\
273	7.93897105705295e-13\\
274	7.93897105705295e-13\\
275	7.93897105705295e-13\\
276	7.93897105705295e-13\\
277	7.93897105705295e-13\\
278	7.93897105705295e-13\\
279	7.93897105705295e-13\\
280	7.93897105705295e-13\\
281	7.93897105705295e-13\\
282	7.93897105705295e-13\\
283	7.93897105705295e-13\\
284	7.93897105705295e-13\\
285	7.93897105705295e-13\\
286	7.93897105705295e-13\\
287	7.93897105705295e-13\\
288	7.93897105705295e-13\\
289	7.93897105705295e-13\\
290	7.93897105705295e-13\\
291	7.93897105705295e-13\\
292	7.93897105705295e-13\\
293	7.93897105705295e-13\\
294	7.93897105705295e-13\\
295	7.93897105705295e-13\\
296	7.93897105705295e-13\\
297	7.93897105705295e-13\\
298	7.93897105705295e-13\\
299	7.93897105705295e-13\\
300	7.93897105705295e-13\\
301	7.93897105705295e-13\\
302	7.93897105705295e-13\\
303	7.93897105705295e-13\\
304	7.93897105705295e-13\\
305	7.93897105705295e-13\\
306	7.93897105705295e-13\\
307	7.93897105705295e-13\\
308	7.93897105705295e-13\\
309	7.93897105705295e-13\\
310	7.93897105705295e-13\\
311	7.93897105705295e-13\\
312	7.93897105705295e-13\\
313	7.93897105705295e-13\\
314	7.93897105705295e-13\\
315	7.93897105705295e-13\\
316	7.93897105705295e-13\\
317	7.93897105705295e-13\\
318	7.93897105705295e-13\\
319	7.93897105705295e-13\\
320	7.93897105705295e-13\\
321	7.93897105705295e-13\\
322	7.93897105705295e-13\\
323	7.93897105705295e-13\\
324	7.93897105705295e-13\\
325	7.93897105705295e-13\\
326	7.93897105705295e-13\\
327	7.93897105705295e-13\\
328	7.93897105705295e-13\\
329	7.93897105705295e-13\\
330	7.93897105705295e-13\\
331	7.93897105705295e-13\\
332	7.93897105705295e-13\\
333	7.93897105705295e-13\\
334	7.93897105705295e-13\\
335	7.93897105705295e-13\\
336	7.93897105705295e-13\\
337	7.93897105705295e-13\\
338	7.93897105705295e-13\\
339	7.93897105705295e-13\\
340	7.93897105705295e-13\\
341	7.93897105705295e-13\\
342	7.93897105705295e-13\\
343	7.93897105705295e-13\\
344	7.93897105705295e-13\\
345	7.93897105705295e-13\\
346	7.93897105705295e-13\\
347	7.93897105705295e-13\\
348	7.93897105705295e-13\\
349	7.93897105705295e-13\\
350	7.93897105705295e-13\\
351	7.93897105705295e-13\\
352	7.93897105705295e-13\\
353	7.93897105705295e-13\\
354	7.93897105705295e-13\\
355	7.93897105705295e-13\\
356	7.93897105705295e-13\\
357	7.93897105705295e-13\\
358	7.93897105705295e-13\\
359	7.93897105705295e-13\\
360	7.93897105705295e-13\\
361	7.93897105705295e-13\\
362	7.93897105705295e-13\\
363	7.93897105705295e-13\\
364	7.93897105705295e-13\\
365	7.93897105705295e-13\\
366	7.93897105705295e-13\\
367	7.93897105705295e-13\\
368	7.93897105705295e-13\\
369	7.93897105705295e-13\\
370	7.93897105705295e-13\\
371	7.93897105705295e-13\\
372	7.93897105705295e-13\\
373	7.93897105705295e-13\\
374	7.93897105705295e-13\\
375	7.93897105705295e-13\\
376	7.93897105705295e-13\\
377	7.93897105705295e-13\\
378	7.93897105705295e-13\\
379	7.93897105705295e-13\\
380	7.93897105705295e-13\\
381	7.93897105705295e-13\\
382	7.93897105705295e-13\\
383	7.93897105705295e-13\\
384	7.93897105705295e-13\\
385	7.93897105705295e-13\\
386	7.93897105705295e-13\\
387	7.93897105705295e-13\\
388	7.93897105705295e-13\\
389	7.93897105705295e-13\\
390	7.93897105705295e-13\\
391	7.93897105705295e-13\\
392	7.93897105705295e-13\\
393	7.93897105705295e-13\\
394	7.93897105705295e-13\\
395	7.93897105705295e-13\\
396	7.93897105705295e-13\\
397	7.93897105705295e-13\\
398	7.93897105705295e-13\\
399	7.93897105705295e-13\\
400	7.93897105705295e-13\\
401	7.93897105705295e-13\\
402	7.93897105705295e-13\\
403	7.93897105705295e-13\\
404	7.93897105705295e-13\\
405	7.93897105705295e-13\\
406	7.93897105705295e-13\\
407	7.93897105705295e-13\\
408	7.93897105705295e-13\\
409	7.93897105705295e-13\\
410	7.93897105705295e-13\\
411	7.93897105705295e-13\\
412	7.93897105705295e-13\\
413	7.93897105705295e-13\\
414	7.93897105705295e-13\\
415	7.93897105705295e-13\\
416	7.93897105705295e-13\\
417	7.93897105705295e-13\\
418	7.93897105705295e-13\\
419	7.93897105705295e-13\\
420	7.93897105705295e-13\\
421	7.93897105705295e-13\\
422	7.93897105705295e-13\\
423	7.93897105705295e-13\\
424	7.93897105705295e-13\\
425	7.93897105705295e-13\\
426	7.93897105705295e-13\\
427	7.93897105705295e-13\\
428	7.93897105705295e-13\\
429	7.93897105705295e-13\\
430	7.93897105705295e-13\\
431	7.93897105705295e-13\\
432	7.93897105705295e-13\\
433	7.93897105705295e-13\\
434	7.93897105705295e-13\\
435	7.93897105705295e-13\\
436	7.93897105705295e-13\\
437	7.93897105705295e-13\\
438	7.93897105705295e-13\\
439	7.93897105705295e-13\\
440	7.93897105705295e-13\\
441	7.93897105705295e-13\\
442	7.93897105705295e-13\\
443	7.93897105705295e-13\\
444	7.93897105705295e-13\\
445	7.93897105705295e-13\\
446	7.93897105705295e-13\\
447	7.93897105705295e-13\\
448	7.93897105705295e-13\\
449	7.93897105705295e-13\\
450	7.93897105705295e-13\\
451	7.93897105705295e-13\\
452	7.93897105705295e-13\\
453	7.93897105705295e-13\\
454	7.93897105705295e-13\\
455	7.93897105705295e-13\\
456	7.93897105705295e-13\\
457	7.93897105705295e-13\\
458	7.93897105705295e-13\\
459	7.93897105705295e-13\\
460	7.93897105705295e-13\\
461	7.93897105705295e-13\\
462	7.93897105705295e-13\\
463	7.93897105705295e-13\\
464	7.93897105705295e-13\\
465	7.93897105705295e-13\\
466	7.93897105705295e-13\\
467	7.93897105705295e-13\\
468	7.93897105705295e-13\\
469	7.93897105705295e-13\\
470	7.93897105705295e-13\\
471	7.93897105705295e-13\\
472	7.93897105705295e-13\\
473	7.93897105705295e-13\\
474	7.93897105705295e-13\\
475	7.93897105705295e-13\\
476	7.93897105705295e-13\\
477	7.93897105705295e-13\\
478	7.93897105705295e-13\\
479	7.93897105705295e-13\\
480	7.93897105705295e-13\\
481	7.93897105705295e-13\\
482	7.93897105705295e-13\\
483	7.93897105705295e-13\\
484	7.93897105705295e-13\\
485	7.93897105705295e-13\\
486	7.93897105705295e-13\\
487	7.93897105705295e-13\\
488	7.93897105705295e-13\\
489	7.93897105705295e-13\\
490	7.93897105705295e-13\\
491	7.93897105705295e-13\\
492	7.93897105705295e-13\\
493	7.93897105705295e-13\\
494	7.93897105705295e-13\\
495	7.93897105705295e-13\\
496	7.93897105705295e-13\\
497	7.93897105705295e-13\\
498	5.33348047356892e-13\\
499	4.90050499616517e-13\\
500	1.61842621189803e-13\\
};
\addlegendentry{Snapshot matrix}

\addplot [color=mycolor2, dashed, line width=2.0pt]
  table[row sep=crcr]{%
1	1023815.27315543\\
2	57694.982116919\\
3	8251.42114403893\\
4	2612.61211190125\\
5	901.335848881695\\
6	339.72692613107\\
7	153.399102429529\\
8	71.6580580773291\\
9	33.1879076734367\\
10	14.3361434451517\\
11	6.28627869074491\\
12	3.66586002921167\\
13	1.81613742336202\\
14	0.706498416990828\\
15	0.252140136564421\\
16	0.0877284751262416\\
17	0.0516535878560779\\
18	0.0249969153765958\\
19	0.00746736481292418\\
20	0.00211319680111016\\
21	0.00059487501054693\\
22	0.00047302231156586\\
23	0.000151290813021942\\
24	5.76736568775315e-05\\
25	1.61879081679708e-05\\
26	3.38547227118172e-06\\
27	2.44999078847708e-06\\
28	6.0016193417039e-07\\
29	1.05266784194457e-07\\
30	3.69836417812793e-08\\
31	1.52481918750013e-08\\
32	2.53111071887689e-09\\
33	9.49493329277217e-10\\
34	2.74243820009519e-10\\
35	1.35595403003394e-10\\
36	1.28112032033794e-10\\
37	1.16620574482592e-10\\
38	1.02149301058444e-10\\
39	1.02149301058444e-10\\
40	1.02149301058444e-10\\
41	1.02149301058444e-10\\
42	1.02149301058444e-10\\
43	1.02149301058444e-10\\
44	1.02149301058444e-10\\
45	1.02149301058444e-10\\
46	1.02149301058444e-10\\
47	1.02149301058444e-10\\
48	1.02149301058444e-10\\
49	1.02149301058444e-10\\
50	1.02149301058444e-10\\
51	1.02149301058444e-10\\
52	1.02149301058444e-10\\
53	1.02149301058444e-10\\
54	1.02149301058444e-10\\
55	1.02149301058444e-10\\
56	1.02149301058444e-10\\
57	1.02149301058444e-10\\
58	1.02149301058444e-10\\
59	1.02149301058444e-10\\
60	1.02149301058444e-10\\
61	1.02149301058444e-10\\
62	1.02149301058444e-10\\
63	1.02149301058444e-10\\
64	1.02149301058444e-10\\
65	1.02149301058444e-10\\
66	1.02149301058444e-10\\
67	1.02149301058444e-10\\
68	1.02149301058444e-10\\
69	1.02149301058444e-10\\
70	1.02149301058444e-10\\
71	1.02149301058444e-10\\
72	1.02149301058444e-10\\
73	1.02149301058444e-10\\
74	1.02149301058444e-10\\
75	1.02149301058444e-10\\
76	1.02149301058444e-10\\
77	1.02149301058444e-10\\
78	1.02149301058444e-10\\
79	1.02149301058444e-10\\
80	1.02149301058444e-10\\
81	1.02149301058444e-10\\
82	1.02149301058444e-10\\
83	1.02149301058444e-10\\
84	1.02149301058444e-10\\
85	1.02149301058444e-10\\
86	1.02149301058444e-10\\
87	1.02149301058444e-10\\
88	1.02149301058444e-10\\
89	1.02149301058444e-10\\
90	1.02149301058444e-10\\
91	1.02149301058444e-10\\
92	1.02149301058444e-10\\
93	1.02149301058444e-10\\
94	1.02149301058444e-10\\
95	1.02149301058444e-10\\
96	1.02149301058444e-10\\
97	1.02149301058444e-10\\
98	1.02149301058444e-10\\
99	1.02149301058444e-10\\
100	1.02149301058444e-10\\
101	1.02149301058444e-10\\
102	1.02149301058444e-10\\
103	1.02149301058444e-10\\
104	1.02149301058444e-10\\
105	1.02149301058444e-10\\
106	1.02149301058444e-10\\
107	1.02149301058444e-10\\
108	1.02149301058444e-10\\
109	1.02149301058444e-10\\
110	1.02149301058444e-10\\
111	1.02149301058444e-10\\
112	1.02149301058444e-10\\
113	1.02149301058444e-10\\
114	1.02149301058444e-10\\
115	1.02149301058444e-10\\
116	1.02149301058444e-10\\
117	1.02149301058444e-10\\
118	1.02149301058444e-10\\
119	1.02149301058444e-10\\
120	1.02149301058444e-10\\
121	1.02149301058444e-10\\
122	1.02149301058444e-10\\
123	1.02149301058444e-10\\
124	1.02149301058444e-10\\
125	1.02149301058444e-10\\
126	1.02149301058444e-10\\
127	1.02149301058444e-10\\
128	1.02149301058444e-10\\
129	1.02149301058444e-10\\
130	1.02149301058444e-10\\
131	1.02149301058444e-10\\
132	1.02149301058444e-10\\
133	1.02149301058444e-10\\
134	1.02149301058444e-10\\
135	1.02149301058444e-10\\
136	1.02149301058444e-10\\
137	1.02149301058444e-10\\
138	1.02149301058444e-10\\
139	1.02149301058444e-10\\
140	1.02149301058444e-10\\
141	1.02149301058444e-10\\
142	1.02149301058444e-10\\
143	1.02149301058444e-10\\
144	1.02149301058444e-10\\
145	1.02149301058444e-10\\
146	1.02149301058444e-10\\
147	1.02149301058444e-10\\
148	1.02149301058444e-10\\
149	1.02149301058444e-10\\
150	1.02149301058444e-10\\
151	1.02149301058444e-10\\
152	1.02149301058444e-10\\
153	1.02149301058444e-10\\
154	1.02149301058444e-10\\
155	1.02149301058444e-10\\
156	1.02149301058444e-10\\
157	1.02149301058444e-10\\
158	1.02149301058444e-10\\
159	1.02149301058444e-10\\
160	1.02149301058444e-10\\
161	1.02149301058444e-10\\
162	1.02149301058444e-10\\
163	1.02149301058444e-10\\
164	1.02149301058444e-10\\
165	1.02149301058444e-10\\
166	1.02149301058444e-10\\
167	1.02149301058444e-10\\
168	1.02149301058444e-10\\
169	1.02149301058444e-10\\
170	1.02149301058444e-10\\
171	1.02149301058444e-10\\
172	1.02149301058444e-10\\
173	1.02149301058444e-10\\
174	1.02149301058444e-10\\
175	1.02149301058444e-10\\
176	1.02149301058444e-10\\
177	1.02149301058444e-10\\
178	1.02149301058444e-10\\
179	1.02149301058444e-10\\
180	1.02149301058444e-10\\
181	1.02149301058444e-10\\
182	1.02149301058444e-10\\
183	1.02149301058444e-10\\
184	1.02149301058444e-10\\
185	1.02149301058444e-10\\
186	1.02149301058444e-10\\
187	1.02149301058444e-10\\
188	1.02149301058444e-10\\
189	1.02149301058444e-10\\
190	1.02149301058444e-10\\
191	1.02149301058444e-10\\
192	1.02149301058444e-10\\
193	1.02149301058444e-10\\
194	1.02149301058444e-10\\
195	1.02149301058444e-10\\
196	1.02149301058444e-10\\
197	1.02149301058444e-10\\
198	1.02149301058444e-10\\
199	1.02149301058444e-10\\
200	1.02149301058444e-10\\
201	1.02149301058444e-10\\
202	1.02149301058444e-10\\
203	1.02149301058444e-10\\
204	1.02149301058444e-10\\
205	1.02149301058444e-10\\
206	1.02149301058444e-10\\
207	1.02149301058444e-10\\
208	1.02149301058444e-10\\
209	1.02149301058444e-10\\
210	1.02149301058444e-10\\
211	1.02149301058444e-10\\
212	1.02149301058444e-10\\
213	1.02149301058444e-10\\
214	1.02149301058444e-10\\
215	1.02149301058444e-10\\
216	1.02149301058444e-10\\
217	1.02149301058444e-10\\
218	1.02149301058444e-10\\
219	1.02149301058444e-10\\
220	1.02149301058444e-10\\
221	1.02149301058444e-10\\
222	1.02149301058444e-10\\
223	1.02149301058444e-10\\
224	1.02149301058444e-10\\
225	1.02149301058444e-10\\
226	1.02149301058444e-10\\
227	1.02149301058444e-10\\
228	1.02149301058444e-10\\
229	1.02149301058444e-10\\
230	1.02149301058444e-10\\
231	1.02149301058444e-10\\
232	1.02149301058444e-10\\
233	1.02149301058444e-10\\
234	1.02149301058444e-10\\
235	1.02149301058444e-10\\
236	1.02149301058444e-10\\
237	1.02149301058444e-10\\
238	1.02149301058444e-10\\
239	1.02149301058444e-10\\
240	1.02149301058444e-10\\
241	1.02149301058444e-10\\
242	1.02149301058444e-10\\
243	1.02149301058444e-10\\
244	1.02149301058444e-10\\
245	1.02149301058444e-10\\
246	1.02149301058444e-10\\
247	1.02149301058444e-10\\
248	1.02149301058444e-10\\
249	1.02149301058444e-10\\
250	1.02149301058444e-10\\
251	1.02149301058444e-10\\
252	1.02149301058444e-10\\
253	1.02149301058444e-10\\
254	1.02149301058444e-10\\
255	1.02149301058444e-10\\
256	1.02149301058444e-10\\
257	1.02149301058444e-10\\
258	1.02149301058444e-10\\
259	1.02149301058444e-10\\
260	1.02149301058444e-10\\
261	1.02149301058444e-10\\
262	1.02149301058444e-10\\
263	1.02149301058444e-10\\
264	1.02149301058444e-10\\
265	1.02149301058444e-10\\
266	1.02149301058444e-10\\
267	1.02149301058444e-10\\
268	1.02149301058444e-10\\
269	1.02149301058444e-10\\
270	1.02149301058444e-10\\
271	1.02149301058444e-10\\
272	1.02149301058444e-10\\
273	1.02149301058444e-10\\
274	1.02149301058444e-10\\
275	1.02149301058444e-10\\
276	1.02149301058444e-10\\
277	1.02149301058444e-10\\
278	1.02149301058444e-10\\
279	1.02149301058444e-10\\
280	1.02149301058444e-10\\
281	1.02149301058444e-10\\
282	1.02149301058444e-10\\
283	1.02149301058444e-10\\
284	1.02149301058444e-10\\
285	1.02149301058444e-10\\
286	1.02149301058444e-10\\
287	1.02149301058444e-10\\
288	1.02149301058444e-10\\
289	1.02149301058444e-10\\
290	1.02149301058444e-10\\
291	1.02149301058444e-10\\
292	1.02149301058444e-10\\
293	1.02149301058444e-10\\
294	1.02149301058444e-10\\
295	1.02149301058444e-10\\
296	1.02149301058444e-10\\
297	1.02149301058444e-10\\
298	1.02149301058444e-10\\
299	1.02149301058444e-10\\
300	1.02149301058444e-10\\
301	1.02149301058444e-10\\
302	1.02149301058444e-10\\
303	1.02149301058444e-10\\
304	1.02149301058444e-10\\
305	1.02149301058444e-10\\
306	1.02149301058444e-10\\
307	1.02149301058444e-10\\
308	1.02149301058444e-10\\
309	1.02149301058444e-10\\
310	1.02149301058444e-10\\
311	1.02149301058444e-10\\
312	1.02149301058444e-10\\
313	1.02149301058444e-10\\
314	1.02149301058444e-10\\
315	1.02149301058444e-10\\
316	1.02149301058444e-10\\
317	1.02149301058444e-10\\
318	1.02149301058444e-10\\
319	1.02149301058444e-10\\
320	1.02149301058444e-10\\
321	1.02149301058444e-10\\
322	1.02149301058444e-10\\
323	1.02149301058444e-10\\
324	1.02149301058444e-10\\
325	1.02149301058444e-10\\
326	1.02149301058444e-10\\
327	1.02149301058444e-10\\
328	1.02149301058444e-10\\
329	1.02149301058444e-10\\
330	1.02149301058444e-10\\
331	1.02149301058444e-10\\
332	1.02149301058444e-10\\
333	1.02149301058444e-10\\
334	1.02149301058444e-10\\
335	1.02149301058444e-10\\
336	1.02149301058444e-10\\
337	1.02149301058444e-10\\
338	1.02149301058444e-10\\
339	1.02149301058444e-10\\
340	1.02149301058444e-10\\
341	1.02149301058444e-10\\
342	1.02149301058444e-10\\
343	1.02149301058444e-10\\
344	1.02149301058444e-10\\
345	1.02149301058444e-10\\
346	1.02149301058444e-10\\
347	1.02149301058444e-10\\
348	1.02149301058444e-10\\
349	1.02149301058444e-10\\
350	1.02149301058444e-10\\
351	1.02149301058444e-10\\
352	1.02149301058444e-10\\
353	1.02149301058444e-10\\
354	1.02149301058444e-10\\
355	1.02149301058444e-10\\
356	1.02149301058444e-10\\
357	1.02149301058444e-10\\
358	1.02149301058444e-10\\
359	1.02149301058444e-10\\
360	1.02149301058444e-10\\
361	1.02149301058444e-10\\
362	1.02149301058444e-10\\
363	1.02149301058444e-10\\
364	1.02149301058444e-10\\
365	1.02149301058444e-10\\
366	1.02149301058444e-10\\
367	1.02149301058444e-10\\
368	1.02149301058444e-10\\
369	1.02149301058444e-10\\
370	1.02149301058444e-10\\
371	1.02149301058444e-10\\
372	1.02149301058444e-10\\
373	1.02149301058444e-10\\
374	1.02149301058444e-10\\
375	1.02149301058444e-10\\
376	1.02149301058444e-10\\
377	1.02149301058444e-10\\
378	1.02149301058444e-10\\
379	1.02149301058444e-10\\
380	1.02149301058444e-10\\
381	1.02149301058444e-10\\
382	1.02149301058444e-10\\
383	1.02149301058444e-10\\
384	1.02149301058444e-10\\
385	1.02149301058444e-10\\
386	1.02149301058444e-10\\
387	1.02149301058444e-10\\
388	1.02149301058444e-10\\
389	1.02149301058444e-10\\
390	1.02149301058444e-10\\
391	1.02149301058444e-10\\
392	1.02149301058444e-10\\
393	1.02149301058444e-10\\
394	1.02149301058444e-10\\
395	1.02149301058444e-10\\
396	1.02149301058444e-10\\
397	1.02149301058444e-10\\
398	1.02149301058444e-10\\
399	1.02149301058444e-10\\
400	1.02149301058444e-10\\
401	1.02149301058444e-10\\
402	1.02149301058444e-10\\
403	1.02149301058444e-10\\
404	1.02149301058444e-10\\
405	1.02149301058444e-10\\
406	1.02149301058444e-10\\
407	1.02149301058444e-10\\
408	1.02149301058444e-10\\
409	1.02149301058444e-10\\
410	1.02149301058444e-10\\
411	1.02149301058444e-10\\
412	1.02149301058444e-10\\
413	1.02149301058444e-10\\
414	1.02149301058444e-10\\
415	1.02149301058444e-10\\
416	1.02149301058444e-10\\
417	1.02149301058444e-10\\
418	1.02149301058444e-10\\
419	1.02149301058444e-10\\
420	1.02149301058444e-10\\
421	1.02149301058444e-10\\
422	1.02149301058444e-10\\
423	1.02149301058444e-10\\
424	1.02149301058444e-10\\
425	1.02149301058444e-10\\
426	1.02149301058444e-10\\
427	1.02149301058444e-10\\
428	1.02149301058444e-10\\
429	1.02149301058444e-10\\
430	1.02149301058444e-10\\
431	1.02149301058444e-10\\
432	1.02149301058444e-10\\
433	1.02149301058444e-10\\
434	1.02149301058444e-10\\
435	1.02149301058444e-10\\
436	1.02149301058444e-10\\
437	1.02149301058444e-10\\
438	1.02149301058444e-10\\
439	1.02149301058444e-10\\
440	1.02149301058444e-10\\
441	1.02149301058444e-10\\
442	1.02149301058444e-10\\
443	1.02149301058444e-10\\
444	1.02149301058444e-10\\
445	1.02149301058444e-10\\
446	1.02149301058444e-10\\
447	1.02149301058444e-10\\
448	1.02149301058444e-10\\
449	1.02149301058444e-10\\
450	1.02149301058444e-10\\
451	1.02149301058444e-10\\
452	1.02149301058444e-10\\
453	1.02149301058444e-10\\
454	1.02149301058444e-10\\
455	1.02149301058444e-10\\
456	1.02149301058444e-10\\
457	1.02149301058444e-10\\
458	1.02149301058444e-10\\
459	1.02149301058444e-10\\
460	1.02149301058444e-10\\
461	1.02149301058444e-10\\
462	1.02149301058444e-10\\
463	1.02149301058444e-10\\
464	1.02149301058444e-10\\
465	1.02149301058444e-10\\
466	1.02149301058444e-10\\
467	1.02149301058444e-10\\
468	1.02149301058444e-10\\
469	1.02149301058444e-10\\
470	1.02149301058444e-10\\
471	1.02149301058444e-10\\
472	1.02149301058444e-10\\
473	1.02149301058444e-10\\
474	1.02149301058444e-10\\
475	1.02149301058444e-10\\
476	1.02149301058444e-10\\
477	1.02149301058444e-10\\
478	1.02149301058444e-10\\
479	1.02149301058444e-10\\
480	1.02149301058444e-10\\
481	1.02149301058444e-10\\
482	1.02149301058444e-10\\
483	1.02149301058444e-10\\
484	1.02149301058444e-10\\
485	1.02149301058444e-10\\
486	1.02149301058444e-10\\
487	1.02149301058444e-10\\
488	1.02149301058444e-10\\
489	1.02149301058444e-10\\
490	1.02149301058444e-10\\
491	1.02149301058444e-10\\
492	1.02149301058444e-10\\
493	1.02149301058444e-10\\
494	1.02149301058444e-10\\
495	1.02149301058444e-10\\
496	1.02149301058444e-10\\
497	1.02149301058444e-10\\
498	1.02149301058444e-10\\
499	1.02149301058444e-10\\
500	1.02149301058444e-10\\
};
\addlegendentry{Nonlinear snapshot matrix}

\end{axis}
\end{tikzpicture}%

%% file: Chafee_Error.tikz
%
%
\definecolor{mycolor1}{rgb}{0.00000,0.44700,0.74100}%
\definecolor{mycolor2}{rgb}{0.85000,0.32500,0.09800}%
\begin{tikzpicture}

\begin{axis}[%
width=0.951\fwidth,
height=\fheight,
at={(0\fwidth,0\fheight)},
scale only axis,
xmin=0,
xmax=20,
xlabel style={font=\color{white!15!black}},
xlabel={time [s]},
ymode=log,
ymin=1e-08,
ymax=0.0001,
yminorticks=true,
ylabel style={font=\color{white!15!black}},
ylabel={$|\by(t) - \hat\by(t)|$},
axis background/.style={fill=white},
legend style={legend cell align=left, align=left, draw=white!15!black,at = { (1,1.1), anchor = south east }}
]
\addplot [color=mycolor2, line width=2.0pt,dashdotted]
  table[row sep=crcr]{%
0	0\\
0.05	5.24534522171605e-05\\
0.1	3.9677349170153e-05\\
0.15	1.19604559738473e-05\\
0.2	2.1032355725259e-05\\
0.25	1.61668145461924e-05\\
0.3	1.07885342295333e-05\\
0.35	5.17436860847376e-06\\
0.4	5.65456242718199e-07\\
0.45	5.11852030071758e-06\\
0.5	7.03535583879145e-06\\
0.55	5.72574960266969e-06\\
0.6	1.89306373621356e-06\\
0.65	2.68754725252229e-06\\
0.7	5.90532500099883e-06\\
0.75	6.24606129351335e-06\\
0.8	3.50283125460926e-06\\
0.85	1.08136551202342e-06\\
0.9	5.35712871063687e-06\\
0.95	7.29329300663828e-06\\
1	5.9928651057195e-06\\
1.05	2.14406199972039e-06\\
1.1	2.38601091573365e-06\\
1.15	5.59747969286484e-06\\
1.2	6.38895145566032e-06\\
1.25	4.93939680223932e-06\\
1.3	2.31136602524273e-06\\
1.35	2.99632100686509e-07\\
1.4	2.13003401161593e-06\\
1.45	2.9734291946415e-06\\
1.5	2.98128695952471e-06\\
1.55	2.44772773272395e-06\\
1.6	1.67367661285667e-06\\
1.65	7.86127435037365e-07\\
1.7	2.9360007625101e-07\\
1.75	1.5579671908128e-06\\
1.8	2.64380017411803e-06\\
1.85	2.87464933967918e-06\\
1.9	1.49352475453668e-06\\
1.95	1.57685478141545e-06\\
2	4.90448073442451e-06\\
2.05	6.38779392914124e-06\\
2.1	4.89725800045804e-06\\
2.15	1.004725439957e-06\\
2.2	3.41848308416104e-06\\
2.25	6.23190692716769e-06\\
2.3	6.08754345465812e-06\\
2.35	3.03926143696565e-06\\
2.4	1.52613285897552e-06\\
2.45	5.54922879447695e-06\\
2.5	7.22804186881199e-06\\
2.55	5.81174787117611e-06\\
2.6	1.93070948184548e-06\\
2.65	2.671578954061e-06\\
2.7	5.89884427837362e-06\\
2.75	6.24355466127824e-06\\
2.8	3.50186263209373e-06\\
2.85	1.0817832187815e-06\\
2.9	5.35730182948946e-06\\
2.95	7.29337102245431e-06\\
3	5.99290820013643e-06\\
3.05	2.14409758414469e-06\\
3.1	2.38598514967769e-06\\
3.15	5.59745518113886e-06\\
3.2	6.38893013604758e-06\\
3.25	4.93936447965027e-06\\
3.3	2.31136563066947e-06\\
3.35	2.9965694170464e-07\\
3.4	2.13007787652764e-06\\
3.45	2.9734694335648e-06\\
3.5	2.98130696485543e-06\\
3.55	2.44772671065263e-06\\
3.6	1.67369382042537e-06\\
3.65	7.86169318423013e-07\\
3.7	2.93532192108259e-07\\
3.75	1.55795944289938e-06\\
3.8	2.64400621718952e-06\\
3.85	2.87485076766458e-06\\
3.9	1.49347657640853e-06\\
3.95	1.57699406333478e-06\\
4	4.9045723684582e-06\\
4.05	6.38777085693043e-06\\
4.1	4.89705321826683e-06\\
4.15	1.00463263597028e-06\\
4.2	3.41853628849087e-06\\
4.25	6.23193000492961e-06\\
4.3	6.08771220811377e-06\\
4.35	3.03938891477173e-06\\
4.4	1.52605510694848e-06\\
4.45	5.54918763984169e-06\\
4.5	7.22802774655307e-06\\
4.55	5.81173867053586e-06\\
4.6	1.93069185550065e-06\\
4.65	2.67157613720315e-06\\
4.7	5.8988367062085e-06\\
4.75	6.24354382394721e-06\\
4.8	3.50184918795904e-06\\
4.85	1.08177270163878e-06\\
4.9	5.35729964079579e-06\\
4.95	7.29336230609334e-06\\
5	5.99290745872949e-06\\
5.05	2.14409678789274e-06\\
5.1	2.38598511215216e-06\\
5.15	5.59745775907672e-06\\
5.2	6.38892751525511e-06\\
5.25	4.939369003365e-06\\
5.3	2.311360978835e-06\\
5.35	2.99654720148368e-07\\
5.4	2.13008697436123e-06\\
5.45	2.97349925482138e-06\\
5.5	2.98134422016538e-06\\
5.55	2.44776221314247e-06\\
5.6	1.67371406445405e-06\\
5.65	7.86175845979287e-07\\
5.7	2.93540509344048e-07\\
5.75	1.55793868550358e-06\\
5.8	2.64390604309828e-06\\
5.85	2.87474010440825e-06\\
5.9	1.49346225719604e-06\\
5.95	1.57702990377651e-06\\
6	4.90461718682944e-06\\
6.05	6.38778904571424e-06\\
6.1	4.89708414330714e-06\\
6.15	1.00464591734628e-06\\
6.2	3.41854988894497e-06\\
6.25	6.23193051940696e-06\\
6.3	6.08766841425634e-06\\
6.35	3.03934786738402e-06\\
6.4	1.52609499237677e-06\\
6.45	5.54920273443393e-06\\
6.5	7.22803334141098e-06\\
6.55	5.81175040847981e-06\\
6.6	1.93070519971528e-06\\
6.65	2.67157421007802e-06\\
6.7	5.89883797386115e-06\\
6.75	6.24354614453537e-06\\
6.8	3.50181697927887e-06\\
6.85	1.08182359848108e-06\\
6.9	5.35732702688918e-06\\
6.95	7.29337512050954e-06\\
7	5.99291422398451e-06\\
7.05	2.14409944354621e-06\\
7.1	2.3859830400319e-06\\
7.15	5.59745704875603e-06\\
7.2	6.38892561499738e-06\\
7.25	4.93937006984524e-06\\
7.3	2.3113597462654e-06\\
7.35	2.99655229518692e-07\\
7.4	2.13008616456456e-06\\
7.45	2.97349732059082e-06\\
7.5	2.9813432256276e-06\\
7.55	2.44776126168134e-06\\
7.6	1.67371297865593e-06\\
7.65	7.86175548550538e-07\\
7.7	2.93541703277889e-07\\
7.75	1.55794408240872e-06\\
7.8	2.64392317372852e-06\\
7.85	2.87475758731226e-06\\
7.9	1.4934573262515e-06\\
7.95	1.57700040781528e-06\\
8	4.90457480029072e-06\\
8.05	6.38776387873463e-06\\
8.1	4.89704623674037e-06\\
8.15	1.00460878393882e-06\\
8.2	3.41856304197918e-06\\
8.25	6.23194008619876e-06\\
8.3	6.08775979604737e-06\\
8.35	3.03946299751168e-06\\
8.4	1.52599028502287e-06\\
8.45	5.54915391526301e-06\\
8.5	7.22800042018967e-06\\
8.55	5.81173238534127e-06\\
8.6	1.93069932019618e-06\\
8.65	2.67158069111595e-06\\
8.7	5.89884374924132e-06\\
8.75	6.24355332834448e-06\\
8.8	3.50186355069226e-06\\
8.85	1.08176579760588e-06\\
8.9	5.35729185657807e-06\\
8.95	7.29336147409221e-06\\
9	5.99290600589164e-06\\
9.05	2.14409618060074e-06\\
9.1	2.38598719026761e-06\\
9.15	5.59745771888664e-06\\
9.2	6.38893368876126e-06\\
9.25	4.93937026369018e-06\\
9.3	2.31136437878199e-06\\
9.35	2.99653427848767e-07\\
9.4	2.13008813121363e-06\\
9.45	2.9735006901177e-06\\
9.5	2.98134575582587e-06\\
9.55	2.44776303670591e-06\\
9.6	1.67371918968762e-06\\
9.65	7.86194800483919e-07\\
9.7	2.93511110083244e-07\\
9.75	1.55793768663592e-06\\
9.8	2.64397600280297e-06\\
9.85	2.87482136007711e-06\\
9.9	1.49346324351818e-06\\
9.95	1.5769929657683e-06\\
10	4.90457470347927e-06\\
10.05	6.38777697026249e-06\\
10.1	4.89705148387642e-06\\
10.15	1.00462596663853e-06\\
10.2	3.41854339747094e-06\\
10.25	6.23193467430561e-06\\
10.3	6.08769999055347e-06\\
10.35	3.03938380574742e-06\\
10.4	1.52605808523276e-06\\
10.45	5.54919535322718e-06\\
10.5	7.2280306653294e-06\\
10.55	5.81174166525145e-06\\
10.6	1.93069544551783e-06\\
10.65	2.67157511624205e-06\\
10.7	5.89883732859953e-06\\
10.75	6.2435454253329e-06\\
10.8	3.50185164488259e-06\\
10.85	1.08177680546717e-06\\
10.9	5.35730287087866e-06\\
10.95	7.29336760119104e-06\\
11	5.99290912339789e-06\\
11.05	2.14409719667685e-06\\
11.1	2.38598534441081e-06\\
11.15	5.59745648498478e-06\\
11.2	6.38892968196636e-06\\
11.25	4.9393660299657e-06\\
11.3	2.31136010331312e-06\\
11.35	2.99655870783511e-07\\
11.4	2.13008471638965e-06\\
11.45	2.97348958699928e-06\\
11.5	2.98133213672003e-06\\
11.55	2.44775084179416e-06\\
11.6	1.67370500725461e-06\\
11.65	7.86168258049003e-07\\
11.7	2.93538699125406e-07\\
11.75	1.55795853573615e-06\\
11.8	2.64398635185792e-06\\
11.85	2.87482569616415e-06\\
11.9	1.4934729124505e-06\\
11.95	1.57698832170539e-06\\
12	4.90457214175066e-06\\
12.05	6.38777527806056e-06\\
12.1	4.8970522852354e-06\\
12.15	1.00462646757116e-06\\
12.2	3.41854201169056e-06\\
12.25	6.23193345905548e-06\\
12.3	6.08770123733393e-06\\
12.35	3.0393840495524e-06\\
12.4	1.52605706604803e-06\\
12.45	5.54919383044528e-06\\
12.5	7.22802893404761e-06\\
12.55	5.8117402985669e-06\\
12.6	1.9306942899977e-06\\
12.65	2.67157581079758e-06\\
12.7	5.89883729640306e-06\\
12.75	6.24354521883141e-06\\
12.8	3.50185170372441e-06\\
12.85	1.0817761519899e-06\\
12.9	5.35730274342505e-06\\
12.95	7.29336726368324e-06\\
13	5.99290888714243e-06\\
13.05	2.14409723930942e-06\\
13.1	2.38598543988999e-06\\
13.15	5.59745668105016e-06\\
13.2	6.38892987403494e-06\\
13.25	4.9393674417253e-06\\
13.3	2.31136021011658e-06\\
13.35	2.9965609060767e-07\\
13.4	2.13008380800517e-06\\
13.45	2.97349008615555e-06\\
13.5	2.98133892884245e-06\\
13.55	2.44775895463789e-06\\
13.6	1.67371566428542e-06\\
13.65	7.86187580148479e-07\\
13.7	2.93517478988647e-07\\
13.75	1.5579427984358e-06\\
13.8	2.64398015892287e-06\\
13.85	2.87482455618715e-06\\
13.9	1.49347808675593e-06\\
13.95	1.57702393388526e-06\\
14	4.90461852487023e-06\\
14.05	6.38779869799322e-06\\
14.1	4.89709755746581e-06\\
14.15	1.00465225183477e-06\\
14.2	3.41854583862933e-06\\
14.25	6.23193476534389e-06\\
14.3	6.08774097554665e-06\\
14.35	3.03943368940018e-06\\
14.4	1.5260137640194e-06\\
14.45	5.54916817674389e-06\\
14.5	7.22800970076598e-06\\
14.55	5.81172975055999e-06\\
14.6	1.93069217124808e-06\\
14.65	2.67158405620194e-06\\
14.7	5.89884174417854e-06\\
14.75	6.24354761580292e-06\\
14.8	3.50185232034228e-06\\
14.85	1.08176059931964e-06\\
14.9	5.35729160011655e-06\\
14.95	7.293354852278e-06\\
15	5.99290622993465e-06\\
15.05	2.14409513410452e-06\\
15.1	2.38598533575107e-06\\
15.15	5.59745799999511e-06\\
15.2	6.38892494242427e-06\\
15.25	4.93936999479416e-06\\
15.3	2.31136014461342e-06\\
15.35	2.9965398296028e-07\\
15.4	2.13008608640486e-06\\
15.45	2.97349611111386e-06\\
15.5	2.98134416487628e-06\\
15.55	2.44776315572182e-06\\
15.6	1.67371153425577e-06\\
15.65	7.86177511979957e-07\\
15.7	2.93539243245711e-07\\
15.75	1.55794583500679e-06\\
15.8	2.64394223137288e-06\\
15.85	2.87477878724296e-06\\
15.9	1.49345043953808e-06\\
15.95	1.57700297132024e-06\\
16	4.90457977253556e-06\\
16.05	6.38777483263908e-06\\
16.1	4.89705157669107e-06\\
16.15	1.00435396088727e-06\\
16.2	3.41871695397344e-06\\
16.25	6.23200306582028e-06\\
16.3	6.08763762577347e-06\\
16.35	3.03928041245349e-06\\
16.4	1.52612552883902e-06\\
16.45	5.54922939488556e-06\\
16.5	7.22804611452688e-06\\
16.55	5.81175377134535e-06\\
16.6	1.93070539267204e-06\\
16.65	2.67157202649138e-06\\
16.7	5.89883657142742e-06\\
16.75	6.24354554989992e-06\\
16.8	3.5018112964913e-06\\
16.85	1.08182427349668e-06\\
16.9	5.35732477602302e-06\\
16.95	7.293374371109e-06\\
17	5.99291314462569e-06\\
17.05	2.14409959986561e-06\\
17.1	2.38598420820857e-06\\
17.15	5.59745778816456e-06\\
17.2	6.38892872761865e-06\\
17.25	4.93937130552347e-06\\
17.3	2.31136211215066e-06\\
17.35	2.99654008495409e-07\\
17.4	2.13008616878341e-06\\
17.45	2.97349500733013e-06\\
17.5	2.98133837617343e-06\\
17.55	2.44775598390312e-06\\
17.6	1.67370763937136e-06\\
17.65	7.86168510957808e-07\\
17.7	2.93547441465591e-07\\
17.75	1.55794456524472e-06\\
17.8	2.64391257220886e-06\\
17.85	2.87474482396632e-06\\
17.9	1.49346526345795e-06\\
17.95	1.57702854330921e-06\\
18	4.90461806479381e-06\\
18.05	6.38778677908292e-06\\
18.1	4.89708591788762e-06\\
18.15	1.00464643626452e-06\\
18.2	3.41854725016688e-06\\
18.25	6.23192934101624e-06\\
18.3	6.08763648912714e-06\\
18.35	3.03932318268529e-06\\
18.4	1.5261029424618e-06\\
18.45	5.54921217821303e-06\\
18.5	7.2280391787416e-06\\
18.55	5.81174794533901e-06\\
18.6	1.93070734133549e-06\\
18.65	2.67157593114575e-06\\
18.7	5.89883940227409e-06\\
18.75	6.2435476271272e-06\\
18.8	3.50185175412854e-06\\
18.85	1.08180690938653e-06\\
18.9	5.35732581363746e-06\\
18.95	7.29337706961708e-06\\
19	5.99291435054994e-06\\
19.05	2.14409931897919e-06\\
19.1	2.38597760993109e-06\\
19.15	5.59745194839145e-06\\
19.2	6.38893430915388e-06\\
19.25	4.939380751523e-06\\
19.3	2.31131849171007e-06\\
19.35	2.99652733737332e-07\\
19.4	2.13009252991725e-06\\
19.45	2.97353320344307e-06\\
19.5	2.9813969868453e-06\\
19.55	2.44781222691337e-06\\
19.6	1.67372108283992e-06\\
19.65	7.86186730050709e-07\\
19.7	2.93519689664734e-07\\
19.75	1.55795931278124e-06\\
19.8	2.64401648386592e-06\\
19.85	2.8748855382954e-06\\
19.9	1.49351831923994e-06\\
19.95	1.57695608371533e-06\\
20	4.90456707846754e-06\\
};
\addlegendentry{Intrusive-POD}

\addplot [color=mycolor3, dashed, line width=2.0pt]
  table[row sep=crcr]{%
0	0\\
0.05	3.54333891383618e-05\\
0.1	1.94249857874973e-05\\
0.15	7.64448826739983e-06\\
0.2	1.08925576609575e-05\\
0.25	6.56992150815938e-06\\
0.3	3.21649079348951e-06\\
0.35	1.20491300559245e-06\\
0.4	4.9970403370736e-08\\
0.45	5.46252331590225e-07\\
0.5	7.72107437096992e-07\\
0.55	7.38247200926523e-07\\
0.6	4.85252642734579e-07\\
0.65	5.0148462271693e-08\\
0.7	4.30910302107179e-07\\
0.75	7.00878763693069e-07\\
0.8	5.12132140251254e-07\\
0.85	1.4915028012652e-07\\
0.9	9.35573431615211e-07\\
0.95	1.2959545123703e-06\\
1	8.90425828403707e-07\\
1.05	6.20434772446998e-08\\
1.1	8.97909829822297e-07\\
1.15	1.08974992207855e-06\\
1.2	7.2183599297837e-07\\
1.25	2.97022835216154e-07\\
1.3	1.33019572245985e-07\\
1.35	1.4587559493684e-07\\
1.4	1.47145328588394e-07\\
1.45	8.05071966869519e-08\\
1.5	3.01761458043615e-08\\
1.55	2.09457942990099e-07\\
1.6	5.44970580218518e-07\\
1.65	1.03044335719726e-06\\
1.7	1.43726706258729e-06\\
1.75	1.43740541569404e-06\\
1.8	8.70391752672006e-07\\
1.85	5.07738868726904e-08\\
1.9	6.92731686458004e-07\\
1.95	5.09399814108136e-07\\
2	2.24778172075446e-07\\
2.05	7.23313672290615e-07\\
2.1	5.51811938809266e-07\\
2.15	7.98126884582473e-08\\
2.2	6.51273308172051e-07\\
2.25	7.91654723775181e-07\\
2.3	4.83633542103235e-07\\
2.35	3.74641782041607e-08\\
2.4	5.10478817883353e-07\\
2.45	7.98125907142122e-07\\
2.5	8.8473744885853e-07\\
2.55	7.88288041508878e-07\\
2.6	5.06640687181203e-07\\
2.65	5.87546677977002e-08\\
2.7	4.27751071008231e-07\\
2.75	6.99871857579026e-07\\
2.8	5.118380927982e-07\\
2.85	1.49267754157023e-07\\
2.9	9.35612180619216e-07\\
2.95	1.29597118947444e-06\\
3	8.90440455592056e-07\\
3.05	6.20201436873913e-08\\
3.1	8.97891107465298e-07\\
3.15	1.08972918555494e-06\\
3.2	7.21815765158951e-07\\
3.25	2.96990566139854e-07\\
3.3	1.33021035297887e-07\\
3.35	1.4585099172848e-07\\
3.4	1.47101316239073e-07\\
3.45	8.04631843376313e-08\\
3.5	3.0199570399958e-08\\
3.55	2.09461343159134e-07\\
3.6	5.44990452988614e-07\\
3.65	1.03046953736641e-06\\
3.7	1.43731438073669e-06\\
3.75	1.43739857538794e-06\\
3.8	8.70185563273118e-07\\
3.85	5.09684321414738e-08\\
3.9	6.92691026760173e-07\\
3.95	5.09267908954669e-07\\
4	2.24867997333789e-07\\
4.05	7.23293489546251e-07\\
4.1	5.51608887233712e-07\\
4.15	7.99043835542079e-08\\
4.2	6.51325325451424e-07\\
4.25	7.91677032818683e-07\\
4.3	4.83801338768686e-07\\
4.35	3.73362722960735e-08\\
4.4	5.10401303888131e-07\\
4.45	7.98085053377307e-07\\
4.5	8.84723178273816e-07\\
4.55	7.88278877283943e-07\\
4.6	5.066230022166e-07\\
4.65	5.87575121890893e-08\\
4.7	4.27743596098651e-07\\
4.75	6.99861135045055e-07\\
4.8	5.1182453697507e-07\\
4.85	1.49257138426506e-07\\
4.9	9.35610019903166e-07\\
4.95	1.29596254772046e-06\\
5	8.90439760370398e-07\\
5.05	6.20210045543246e-08\\
5.1	8.9789227541992e-07\\
5.15	1.08973268453383e-06\\
5.2	7.21813712800667e-07\\
5.25	2.96996105264569e-07\\
5.3	1.33015746861531e-07\\
5.35	1.45852921740186e-07\\
5.4	1.47095646774176e-07\\
5.45	8.04450173141902e-08\\
5.5	3.02244544947428e-08\\
5.55	2.09483009383504e-07\\
5.6	5.45000198082235e-07\\
5.65	1.03048315436283e-06\\
5.7	1.43731672275216e-06\\
5.75	1.43742573821548e-06\\
5.8	8.70276637532363e-07\\
5.85	5.08731299309062e-08\\
5.9	6.92671704882741e-07\\
5.95	5.09230814405015e-07\\
6	2.24910959634173e-07\\
6.05	7.23317399753398e-07\\
6.1	5.51652085123422e-07\\
6.15	7.98824724146385e-08\\
6.2	6.51334322476771e-07\\
6.25	7.91675484501653e-07\\
6.3	4.83757433444865e-07\\
6.35	3.73769546424541e-08\\
6.4	5.10439822187791e-07\\
6.45	7.98100470156271e-07\\
6.5	8.84728012184866e-07\\
6.55	7.88289236997031e-07\\
6.6	5.06635511321463e-07\\
6.65	5.87586865830048e-08\\
6.7	4.27744772046879e-07\\
6.75	6.99863288211588e-07\\
6.8	5.11792634272368e-07\\
6.85	1.49308193586606e-07\\
6.9	9.35637284982249e-07\\
6.95	1.29597549913818e-06\\
7	8.90446678392109e-07\\
7.05	6.20180260479941e-08\\
7.1	8.97890249929034e-07\\
7.15	1.08973206769392e-06\\
7.2	7.21812278170475e-07\\
7.25	2.9699766423974e-07\\
7.3	1.33015298109385e-07\\
7.35	1.45853043198585e-07\\
7.4	1.47095015945453e-07\\
7.45	8.04426194545016e-08\\
7.5	3.02280627195728e-08\\
7.55	2.09486609392684e-07\\
7.6	5.45002545759843e-07\\
7.65	1.03048530430971e-06\\
7.7	1.43731786661494e-06\\
7.75	1.43742328750918e-06\\
7.8	8.70263556329576e-07\\
7.85	5.08873418958444e-08\\
7.9	6.92666734192215e-07\\
7.95	5.09271082416163e-07\\
8	2.2486221307183e-07\\
8.05	7.23286541992607e-07\\
8.1	5.51605987775261e-07\\
8.15	7.99247581451112e-08\\
8.2	6.51350189784239e-07\\
8.25	7.91686893597543e-07\\
8.3	4.83849452281859e-07\\
8.35	3.72646711266356e-08\\
8.4	5.1033436632153e-07\\
8.45	7.98050250105931e-07\\
8.5	8.84696283121045e-07\\
8.55	7.88268817553117e-07\\
8.6	5.06625343676959e-07\\
8.65	5.87474395796761e-08\\
8.7	4.27755082466064e-07\\
8.75	6.99873549558916e-07\\
8.8	5.11837878747201e-07\\
8.85	1.49250235059739e-07\\
8.9	9.3560117964131e-07\\
8.95	1.29596030973289e-06\\
9	8.90438298206675e-07\\
9.05	6.20220927949333e-08\\
9.1	8.97893020823659e-07\\
9.15	1.08973099921528e-06\\
9.2	7.21815942128501e-07\\
9.25	2.96992447967881e-07\\
9.3	1.33014205649928e-07\\
9.35	1.45851158261934e-07\\
9.4	1.47097130476226e-07\\
9.45	8.04504540763418e-08\\
9.5	3.02175657829196e-08\\
9.55	2.09477230894706e-07\\
9.6	5.44997836859906e-07\\
9.65	1.03048123412108e-06\\
9.7	1.43732249158202e-06\\
9.75	1.43740971425554e-06\\
9.8	8.70201028790873e-07\\
9.85	5.09547046778636e-08\\
9.9	6.92680187208694e-07\\
9.95	5.09270364545955e-07\\
10	2.24866980147453e-07\\
10.05	7.23304239613753e-07\\
10.1	5.51619094846245e-07\\
10.15	7.99021457886795e-08\\
10.2	6.51327078049491e-07\\
10.25	7.91679250822241e-07\\
10.3	4.83788703542487e-07\\
10.35	3.7343392378375e-08\\
10.4	5.10403917131086e-07\\
10.45	7.98092482767743e-07\\
10.5	8.84725175565038e-07\\
10.55	7.88280679620001e-07\\
10.6	5.06625804641558e-07\\
10.65	5.87575770261139e-08\\
10.7	4.27744437425659e-07\\
10.75	6.9986279305212e-07\\
10.8	5.11826986571151e-07\\
10.85	1.4926154823236e-07\\
10.9	9.35613270636182e-07\\
10.95	1.29596751263783e-06\\
11	8.904413304478e-07\\
11.05	6.20207181167842e-08\\
11.1	8.97892211693119e-07\\
11.15	1.08973105716892e-06\\
11.2	7.21815532900294e-07\\
11.25	2.96992969328613e-07\\
11.3	1.33014413705723e-07\\
11.35	1.45851474231407e-07\\
11.4	1.47096964608906e-07\\
11.45	8.04496083084416e-08\\
11.5	3.021897621025e-08\\
11.55	2.09478818069542e-07\\
11.6	5.44998352447479e-07\\
11.65	1.03048648991688e-06\\
11.7	1.43732817281528e-06\\
11.75	1.43741149138954e-06\\
11.8	8.70203057390384e-07\\
11.85	5.09517732449893e-08\\
11.9	6.92679885450076e-07\\
11.95	5.09261321113286e-07\\
12	2.24872296783474e-07\\
12.05	7.23299055538362e-07\\
12.1	5.51609732779568e-07\\
12.15	7.99087087610673e-08\\
12.2	6.51330225975855e-07\\
12.25	7.91679964917691e-07\\
12.3	4.83789887484321e-07\\
12.35	3.7341463476892e-08\\
12.4	5.10404239983941e-07\\
12.45	7.98092312459531e-07\\
12.5	8.84724203231713e-07\\
12.55	7.8827939353765e-07\\
12.6	5.06624789453625e-07\\
12.65	5.87576605148854e-08\\
12.7	4.27744015762954e-07\\
12.75	6.99862266806406e-07\\
12.8	5.11826657945136e-07\\
12.85	1.4926115188274e-07\\
12.9	9.35613553965098e-07\\
12.95	1.29596748155159e-06\\
13	8.90441256506946e-07\\
13.05	6.20205160561937e-08\\
13.1	8.97892397766498e-07\\
13.15	1.08973137091795e-06\\
13.2	7.21815895721178e-07\\
13.25	2.96993754700381e-07\\
13.3	1.33015332082209e-07\\
13.35	1.4585176177917e-07\\
13.4	1.47095978952905e-07\\
13.45	8.04470849935512e-08\\
13.5	3.02214095970754e-08\\
13.55	2.09480487178837e-07\\
13.6	5.44998950857689e-07\\
13.65	1.0304830760921e-06\\
13.7	1.43732448831813e-06\\
13.75	1.4374078587398e-06\\
13.8	8.70194782121025e-07\\
13.85	5.09595934339302e-08\\
13.9	6.92697020410193e-07\\
13.95	5.09249497460118e-07\\
14	2.24904584289476e-07\\
14.05	7.2331370049028e-07\\
14.1	5.516481040857e-07\\
14.15	7.98877002328169e-08\\
14.2	6.51335894996663e-07\\
14.25	7.91682170264707e-07\\
14.3	4.83829967423688e-07\\
14.35	3.72914081836484e-08\\
14.4	5.1035911274866e-07\\
14.45	7.98066331242353e-07\\
14.5	8.84705873671621e-07\\
14.55	7.88269229223815e-07\\
14.6	5.06623106799609e-07\\
14.65	5.874978947773e-08\\
14.7	4.27748318543308e-07\\
14.75	6.99864100894843e-07\\
14.8	5.11827308979917e-07\\
14.85	1.49245199976278e-07\\
14.9	9.35601971674416e-07\\
14.95	1.29595517117664e-06\\
15	8.90438623279977e-07\\
15.05	6.2022393665373e-08\\
15.1	8.97891652140714e-07\\
15.15	1.0897323130532e-06\\
15.2	7.21811942216988e-07\\
15.25	2.9699833303809e-07\\
15.3	1.33016984760204e-07\\
15.35	1.45854810007506e-07\\
15.4	1.47092752422751e-07\\
15.45	8.04346937943734e-08\\
15.5	3.02373972527192e-08\\
15.55	2.09495664593717e-07\\
15.6	5.45007408092602e-07\\
15.65	1.03048947042161e-06\\
15.7	1.43732122148688e-06\\
15.75	1.43742022207238e-06\\
15.8	8.70242237938079e-07\\
15.85	5.0908278925732e-08\\
15.9	6.92659165135723e-07\\
15.95	5.09260666081701e-07\\
16	2.2487199280441e-07\\
16.05	7.23299702354296e-07\\
16.1	5.51613940302786e-07\\
16.15	8.01780097869909e-08\\
16.2	6.51503647031149e-07\\
16.25	7.91749744655235e-07\\
16.3	4.83727577105242e-07\\
16.35	3.74461741614596e-08\\
16.4	5.10471696912873e-07\\
16.45	7.9812687192593e-07\\
16.5	8.84742059392707e-07\\
16.55	7.88292668918444e-07\\
16.6	5.06634767916125e-07\\
16.65	5.87598740775519e-08\\
16.7	4.27744851760892e-07\\
16.75	6.99863921038713e-07\\
16.8	5.11785651857721e-07\\
16.85	1.49309403729703e-07\\
16.9	9.35634932197615e-07\\
16.95	1.29597396236747e-06\\
17	8.90444843415494e-07\\
17.05	6.20185121036343e-08\\
17.1	8.97890589213191e-07\\
17.15	1.08973169088422e-06\\
17.2	7.21812726256488e-07\\
17.25	2.96995692705693e-07\\
17.3	1.33015225722843e-07\\
17.35	1.45852732336138e-07\\
17.4	1.47095940761233e-07\\
17.45	8.04458759606774e-08\\
17.5	3.0223870961521e-08\\
17.55	2.09483597801707e-07\\
17.6	5.44999345652997e-07\\
17.65	1.03048683086637e-06\\
17.7	1.4373211155716e-06\\
17.75	1.43742750857712e-06\\
17.8	8.70275969733214e-07\\
17.85	5.08724464776122e-08\\
17.9	6.92670885760194e-07\\
17.95	5.09229226564045e-07\\
18	2.24913919932845e-07\\
18.05	7.23315839223915e-07\\
18.1	5.51652658664636e-07\\
18.15	7.98818495795217e-08\\
18.2	6.51331757861584e-07\\
18.25	7.91672749578254e-07\\
18.3	4.83725427047332e-07\\
18.35	3.74012920634215e-08\\
18.4	5.10448660007157e-07\\
18.45	7.98109193844709e-07\\
18.5	8.84735279260696e-07\\
18.55	7.88287851882785e-07\\
18.6	5.06638016872785e-07\\
18.65	5.87572204224784e-08\\
18.7	4.27746944087204e-07\\
18.75	6.99865255304744e-07\\
18.8	5.1182681870543e-07\\
18.85	1.49291356832393e-07\\
18.9	9.35635847021388e-07\\
18.95	1.29597712339447e-06\\
19	8.90446396395461e-07\\
19.05	6.20184317234873e-08\\
19.1	8.97884506967372e-07\\
19.15	1.08972657097972e-06\\
19.2	7.21820015536778e-07\\
19.25	2.97006755411999e-07\\
19.3	1.32973385635893e-07\\
19.35	1.45855013622409e-07\\
19.4	1.47086970159194e-07\\
19.45	8.04004864907171e-08\\
19.5	3.02889153758201e-08\\
19.55	2.09544034790454e-07\\
19.6	5.4501718249611e-07\\
19.65	1.03050505984026e-06\\
19.7	1.43734718704991e-06\\
19.75	1.43741313707313e-06\\
19.8	8.70171516842433e-07\\
19.85	5.10156683564134e-08\\
19.9	6.92724856588001e-07\\
19.95	5.09300791096123e-07\\
20	2.24863053288615e-07\\
};
\addlegendentry{Learned ROM}

\end{axis}
\end{tikzpicture}%

%% file: Chafee_DiffFullState.tikz
%
%
\definecolor{mycolor1}{rgb}{0.00000,0.44700,0.74100}%
\definecolor{mycolor2}{rgb}{0.85000,0.32500,0.09800}%
\definecolor{mycolor3}{rgb}{0.92900,0.69400,0.12500}%
\begin{tikzpicture}

\begin{axis}[%
width=0.951\fwidth,
height=\fheight,
at={(0\fwidth,0\fheight)},
scale only axis,
xmin=2,
xmax=12,
xlabel={Order of ROM},
ymode=log,
ymin=1e-06,
ymax=1,
yminorticks=true,
ylabel={Error b$/$w FOM and ROMs},
axis background/.style={fill=white},
legend style={legend cell align=left, align=left, draw=white!15!black,opacity = 0.8}
]
\addplot [color=mycolor2, line width=2.0pt,dashdotted]
  table[row sep=crcr]{%
2	0.102072678589675\\
3	0.0145292316544612\\
4	0.00443735776634642\\
5	0.00125838764567452\\
6	0.000478253472671275\\
7	0.000183921567935423\\
8	7.17515523014688e-05\\
9	5.03597261119824e-05\\
10	1.60248407031151e-05\\
11	4.6515176296464e-06\\
12	2.6155980340248e-06\\
};
\addlegendentry{$\|\bS-\bV\hat\bS^{\text{POD}}\|_2/\|\bS\|_2$}

\addplot [color=mycolor3, dashed, line width=2.0pt]
  table[row sep=crcr]{%
2	0.0453308427594091\\
3	0.00761852689744839\\
4	0.00226650294625967\\
5	0.000696320983792447\\
6	0.000304051106959241\\
7	0.000123690612626223\\
8	4.93874952823831e-05\\
9	2.64365774377674e-05\\
10	1.60226481490839e-05\\
11	7.37607508213918e-06\\
12	7.61896581350641e-06\\
};
\addlegendentry{$\|\bS-\bV\hat\bS^{\text{NI}}\|_2/\|\bS\|_2$}

\end{axis}
\end{tikzpicture}%

%% file: decay_singularValues.tikz
%
%
\definecolor{mycolor1}{rgb}{0.00000,0.44700,0.74100}%
\definecolor{mycolor2}{rgb}{0.85000,0.32500,0.09800}%
\definecolor{mycolor3}{rgb}{0.92900,0.69400,0.12500}%
\definecolor{mycolor4}{rgb}{0.49400,0.18400,0.55600}%
\begin{tikzpicture}

\begin{axis}[%
width=0.951\fwidth,
height=\fheight,
at={(0\fwidth,0\fheight)},
scale only axis,
xmin=1,
xmax=100,
xlabel style={font=\color{white!15!black}},
xlabel={$k$},
ymode=log,
ymin=5.70677860945322e-15,
ymax=2174.32045116354,
yminorticks=true,
ylabel style={font=\color{white!15!black}},
ylabel={Singular values},
axis background/.style={fill=white},
grid = major,
legend style={legend cell align=left, align=left, draw=white!15!black, at = {(1,1), anchor = south east }}
]
\addplot [color=mycolor1, line width=2.0pt]
  table[row sep=crcr]{%
1	2174.32045116354\\
2	108.935872943434\\
3	13.9774273336337\\
4	2.69660097277839\\
5	1.02844126620881\\
6	0.237477024885003\\
7	0.034651655451967\\
8	0.0307598692937067\\
9	0.00609256085420922\\
10	0.00291120630255683\\
11	0.000679755076180014\\
12	0.000174485570313061\\
13	0.000137763040718758\\
14	0.000108168071522644\\
15	2.99168699733686e-05\\
16	1.92940392518332e-05\\
17	1.03787108079988e-05\\
18	7.49026365519655e-06\\
19	3.70828047639027e-06\\
20	3.096977390223e-06\\
21	2.05840219636725e-06\\
22	1.20095777882593e-06\\
23	4.56070444595736e-07\\
24	3.24003104205034e-07\\
25	2.23146654139249e-07\\
26	1.04994435273152e-07\\
27	5.23808416654637e-08\\
28	3.34863362198161e-08\\
29	1.50432238473375e-08\\
30	9.42187913110115e-09\\
31	5.69194775003818e-09\\
32	3.22428947707088e-09\\
33	1.65815414068253e-09\\
34	6.99725552697215e-10\\
35	6.39068918441052e-10\\
36	3.48717931743479e-10\\
37	2.91978737077954e-10\\
38	1.04419871966636e-10\\
39	8.89919334806616e-11\\
40	2.62879378798492e-11\\
41	1.52684187317545e-11\\
42	9.77509575450388e-12\\
43	3.02218170725811e-12\\
44	2.2951320616308e-12\\
45	1.11052477635054e-12\\
46	8.88786709108713e-13\\
47	7.20909289104601e-13\\
48	3.69520787885674e-13\\
49	2.77423676199409e-13\\
50	2.17258194012705e-13\\
51	2.17258194012705e-13\\
52	2.17258194012705e-13\\
53	2.17258194012705e-13\\
54	2.17258194012705e-13\\
55	2.17258194012705e-13\\
56	2.17258194012705e-13\\
57	2.17258194012705e-13\\
58	2.17258194012705e-13\\
59	2.17258194012705e-13\\
60	2.17258194012705e-13\\
61	2.17258194012705e-13\\
62	2.17258194012705e-13\\
63	2.17258194012705e-13\\
64	2.17258194012705e-13\\
65	2.17258194012705e-13\\
66	2.17258194012705e-13\\
67	2.17258194012705e-13\\
68	2.17258194012705e-13\\
69	2.17258194012705e-13\\
70	2.17258194012705e-13\\
71	2.17258194012705e-13\\
72	2.17258194012705e-13\\
73	2.17258194012705e-13\\
74	2.17258194012705e-13\\
75	2.17258194012705e-13\\
76	2.17258194012705e-13\\
77	2.17258194012705e-13\\
78	2.17258194012705e-13\\
79	2.17258194012705e-13\\
80	2.17258194012705e-13\\
81	2.17258194012705e-13\\
82	2.17258194012705e-13\\
83	2.17258194012705e-13\\
84	2.17258194012705e-13\\
85	2.17258194012705e-13\\
86	2.17258194012705e-13\\
87	2.17258194012705e-13\\
88	2.17258194012705e-13\\
89	2.17258194012705e-13\\
90	2.17258194012705e-13\\
91	2.17258194012705e-13\\
92	2.17258194012705e-13\\
93	2.17258194012705e-13\\
94	2.17258194012705e-13\\
95	2.17258194012705e-13\\
96	2.17258194012705e-13\\
97	2.17258194012705e-13\\
98	2.17258194012705e-13\\
99	2.17258194012705e-13\\
100	2.17258194012705e-13\\
101	2.17258194012705e-13\\
102	2.17258194012705e-13\\
103	2.17258194012705e-13\\
104	2.17258194012705e-13\\
105	2.17258194012705e-13\\
106	2.17258194012705e-13\\
107	2.17258194012705e-13\\
108	2.17258194012705e-13\\
109	2.17258194012705e-13\\
110	2.17258194012705e-13\\
111	2.17258194012705e-13\\
112	2.17258194012705e-13\\
113	2.17258194012705e-13\\
114	2.17258194012705e-13\\
115	2.17258194012705e-13\\
116	2.17258194012705e-13\\
117	2.17258194012705e-13\\
118	2.17258194012705e-13\\
119	2.17258194012705e-13\\
120	2.17258194012705e-13\\
121	2.17258194012705e-13\\
122	2.17258194012705e-13\\
123	2.17258194012705e-13\\
124	2.17258194012705e-13\\
125	2.17258194012705e-13\\
126	2.17258194012705e-13\\
127	2.17258194012705e-13\\
128	2.17258194012705e-13\\
129	2.17258194012705e-13\\
130	2.17258194012705e-13\\
131	2.17258194012705e-13\\
132	2.17258194012705e-13\\
133	2.17258194012705e-13\\
134	2.17258194012705e-13\\
135	2.17258194012705e-13\\
136	2.17258194012705e-13\\
137	2.17258194012705e-13\\
138	2.17258194012705e-13\\
139	2.17258194012705e-13\\
140	2.17258194012705e-13\\
141	2.17258194012705e-13\\
142	2.17258194012705e-13\\
143	2.17258194012705e-13\\
144	2.17258194012705e-13\\
145	2.17258194012705e-13\\
146	2.17258194012705e-13\\
147	2.17258194012705e-13\\
148	2.17258194012705e-13\\
149	2.17258194012705e-13\\
150	2.17258194012705e-13\\
151	2.17258194012705e-13\\
152	2.17258194012705e-13\\
153	2.17258194012705e-13\\
154	2.17258194012705e-13\\
155	2.17258194012705e-13\\
156	2.17258194012705e-13\\
157	2.17258194012705e-13\\
158	2.17258194012705e-13\\
159	2.17258194012705e-13\\
160	2.17258194012705e-13\\
161	2.17258194012705e-13\\
162	2.17258194012705e-13\\
163	2.17258194012705e-13\\
164	2.17258194012705e-13\\
165	2.17258194012705e-13\\
166	2.17258194012705e-13\\
167	2.17258194012705e-13\\
168	2.17258194012705e-13\\
169	2.17258194012705e-13\\
170	2.17258194012705e-13\\
171	2.17258194012705e-13\\
172	2.17258194012705e-13\\
173	2.17258194012705e-13\\
174	2.17258194012705e-13\\
175	2.17258194012705e-13\\
176	2.17258194012705e-13\\
177	2.17258194012705e-13\\
178	2.17258194012705e-13\\
179	2.17258194012705e-13\\
180	2.17258194012705e-13\\
181	2.17258194012705e-13\\
182	2.17258194012705e-13\\
183	2.17258194012705e-13\\
184	2.17258194012705e-13\\
185	2.17258194012705e-13\\
186	2.17258194012705e-13\\
187	2.17258194012705e-13\\
188	2.17258194012705e-13\\
189	2.17258194012705e-13\\
190	2.17258194012705e-13\\
191	2.17258194012705e-13\\
192	2.17258194012705e-13\\
193	2.17258194012705e-13\\
194	2.17258194012705e-13\\
195	2.17258194012705e-13\\
196	2.17258194012705e-13\\
197	1.2773604195785e-13\\
198	1.16383371853545e-13\\
};
\addlegendentry{Snapshot matrix}

\addplot [color=mycolor2, dashed, line width=2.0pt]
  table[row sep=crcr]{%
1	1557.15490228217\\
2	146.204336009763\\
3	80.8211899071104\\
4	11.7009235637465\\
5	4.49833481225393\\
6	0.714259231406132\\
7	0.205024169014485\\
8	0.0351758481889373\\
9	0.00991065677861757\\
10	0.00194727901772263\\
11	0.000639302959907262\\
12	0.0001600992486503\\
13	6.40377641218541e-05\\
14	2.64178087487014e-05\\
15	9.69320314561997e-06\\
16	4.50678154958125e-06\\
17	1.87722863270979e-06\\
18	1.08345439354584e-06\\
19	7.88885017261276e-07\\
20	4.27783163761563e-07\\
21	1.60134823583664e-07\\
22	1.08927571302158e-07\\
23	4.39543506307744e-08\\
24	2.62908470223653e-08\\
25	7.28036982974865e-09\\
26	5.06632297006585e-09\\
27	1.39754921801776e-09\\
28	9.80676201373733e-10\\
29	2.27075464544385e-10\\
30	1.67835486562915e-10\\
31	5.19862619893593e-11\\
32	2.73167329408813e-11\\
33	1.82577639107895e-11\\
34	5.94179249616524e-12\\
35	3.27220493350765e-12\\
36	1.38303138033728e-12\\
37	1.14004964637243e-12\\
38	9.94758356379108e-13\\
39	9.00365169540395e-13\\
40	8.77736825066599e-13\\
41	8.23590444519417e-13\\
42	7.81620461727985e-13\\
43	7.57579525467774e-13\\
44	7.2394748785214e-13\\
45	7.08961080769107e-13\\
46	7.0055796808167e-13\\
47	6.88032662184285e-13\\
48	6.7715445594399e-13\\
49	6.7582055083286e-13\\
50	6.72176389671619e-13\\
51	6.59834247965479e-13\\
52	6.51121887882371e-13\\
53	6.44320723064408e-13\\
54	6.41584290229045e-13\\
55	6.30168587097945e-13\\
56	6.25588853870376e-13\\
57	6.23725381000592e-13\\
58	6.08652774301146e-13\\
59	6.0779564118211e-13\\
60	6.01299452335177e-13\\
61	5.78750275322339e-13\\
62	5.713240697351e-13\\
63	5.60723950579601e-13\\
64	5.44044876074337e-13\\
65	5.3103011019709e-13\\
66	5.18953719123414e-13\\
67	5.084919508512e-13\\
68	4.94212910646204e-13\\
69	4.84437233923026e-13\\
70	4.67837266885727e-13\\
71	4.41861760111156e-13\\
72	4.29697189586494e-13\\
73	4.20831195275238e-13\\
74	4.16601242237751e-13\\
75	4.01711778114369e-13\\
76	3.77366681224417e-13\\
77	3.74018320808762e-13\\
78	3.71853232855463e-13\\
79	3.46009041674648e-13\\
80	3.26406451009189e-13\\
81	3.14025839556803e-13\\
82	3.12979844059504e-13\\
83	2.96400493657544e-13\\
84	2.87423178103509e-13\\
85	2.82338240579228e-13\\
86	2.81058184812994e-13\\
87	2.78214136695572e-13\\
88	2.7140324269395e-13\\
89	2.67971338845858e-13\\
90	2.5371189466033e-13\\
91	2.46040345958391e-13\\
92	2.33920066117206e-13\\
93	2.22334673105644e-13\\
94	2.07301277326696e-13\\
95	2.04618109274764e-13\\
96	1.98024694947874e-13\\
97	1.97170290904485e-13\\
98	1.883297279286e-13\\
99	1.74203725032407e-13\\
100	1.53355271210834e-13\\
101	1.53355271210834e-13\\
102	1.53355271210834e-13\\
103	1.53355271210834e-13\\
104	1.53355271210834e-13\\
105	1.53355271210834e-13\\
106	1.53355271210834e-13\\
107	1.53355271210834e-13\\
108	1.53355271210834e-13\\
109	1.53355271210834e-13\\
110	1.53355271210834e-13\\
111	1.53355271210834e-13\\
112	1.53355271210834e-13\\
113	1.53355271210834e-13\\
114	1.53355271210834e-13\\
115	1.53355271210834e-13\\
116	1.53355271210834e-13\\
117	1.53355271210834e-13\\
118	1.53355271210834e-13\\
119	1.53355271210834e-13\\
120	1.53355271210834e-13\\
121	1.53355271210834e-13\\
122	1.53355271210834e-13\\
123	1.53355271210834e-13\\
124	1.53355271210834e-13\\
125	1.53355271210834e-13\\
126	1.53355271210834e-13\\
127	1.53355271210834e-13\\
128	1.53355271210834e-13\\
129	1.53355271210834e-13\\
130	1.53355271210834e-13\\
131	1.53355271210834e-13\\
132	1.53355271210834e-13\\
133	1.53355271210834e-13\\
134	1.53355271210834e-13\\
135	1.53355271210834e-13\\
136	1.53355271210834e-13\\
137	1.53355271210834e-13\\
138	1.53355271210834e-13\\
139	1.53355271210834e-13\\
140	1.53355271210834e-13\\
141	1.53355271210834e-13\\
142	1.53355271210834e-13\\
143	1.53355271210834e-13\\
144	1.53355271210834e-13\\
145	1.53355271210834e-13\\
146	1.53355271210834e-13\\
147	1.53355271210834e-13\\
148	1.53355271210834e-13\\
149	1.53355271210834e-13\\
150	1.53355271210834e-13\\
151	1.53355271210834e-13\\
152	1.53355271210834e-13\\
153	1.53355271210834e-13\\
154	1.53355271210834e-13\\
155	1.53355271210834e-13\\
156	1.53355271210834e-13\\
157	1.53355271210834e-13\\
158	1.53355271210834e-13\\
159	1.53355271210834e-13\\
160	1.53355271210834e-13\\
161	1.53355271210834e-13\\
162	1.53355271210834e-13\\
163	1.53355271210834e-13\\
164	1.53355271210834e-13\\
165	1.53355271210834e-13\\
166	1.53355271210834e-13\\
167	1.53355271210834e-13\\
168	1.53355271210834e-13\\
169	1.53355271210834e-13\\
170	1.53355271210834e-13\\
171	1.53355271210834e-13\\
172	1.53355271210834e-13\\
173	1.53355271210834e-13\\
174	1.53355271210834e-13\\
175	1.53355271210834e-13\\
176	1.53355271210834e-13\\
177	1.53355271210834e-13\\
178	1.53355271210834e-13\\
179	1.53355271210834e-13\\
180	1.53355271210834e-13\\
181	1.53355271210834e-13\\
182	1.53355271210834e-13\\
183	1.53355271210834e-13\\
184	1.53355271210834e-13\\
185	1.53355271210834e-13\\
186	1.53355271210834e-13\\
187	1.53355271210834e-13\\
188	1.53355271210834e-13\\
189	1.53355271210834e-13\\
190	1.53355271210834e-13\\
191	1.53355271210834e-13\\
192	9.49192549125257e-14\\
193	8.36760517178906e-14\\
194	3.5103211514031e-14\\
195	2.2627382119127e-14\\
196	1.62086346736163e-14\\
197	1.10309537498471e-14\\
198	5.70677860945322e-15\\
};
\addlegendentry{Nonlinear snapshot matrix}

\end{axis}
\end{tikzpicture}%

%% file: state_error_convergence.tikz
%
%
\definecolor{mycolor1}{rgb}{0.00000,0.44700,0.74100}%
\definecolor{mycolor2}{rgb}{0.85000,0.32500,0.09800}%
\definecolor{mycolor3}{rgb}{0.92900,0.69400,0.12500}%
\begin{tikzpicture}

\begin{axis}[%
width=0.973\fwidth,
height=\fheight,
at={(0\fwidth,0\fheight)},
scale only axis,
xmin=4,
xmax=20,
xlabel style={font=\color{white!15!black}},
xlabel={reduced model dimension},
ymode=log,
ymin=1.4203677216366e-07,
ymax=0.0491884373834765,
yminorticks=true,
ylabel style={font=\color{white!15!black}},
ylabel={Error b/w FOM and ROM},
axis background/.style={fill=white},
legend style={legend cell align=left, align=left, draw=white!15!black}
]
\addplot [color=mycolor2,  line width=2.0pt,dashdotted]
  table[row sep=crcr]{%
4	0.0491884373834765\\
6	0.00483437602869119\\
8	0.000787590194736884\\
10	6.4146946927848e-05\\
12	3.10069060137348e-05\\
14	1.90219938301552e-06\\
16	1.4160993524883e-06\\
18	2.51694444132713e-07\\
20	2.22579164884324e-07\\
};
\addlegendentry{$\|\bS-\bV\hat\bS^{\text{POD}}\|_2/\|\bS\|_2$}

\addplot [color=mycolor3, dashed, line width=2.0pt]
  table[row sep=crcr]{%
4	0.0415090415699991\\
6	0.00567360888756136\\
8	0.000430473900016045\\
10	3.9907976906436e-05\\
12	9.46690926672403e-07\\
14	7.10848405376726e-07\\
16	1.4203677216366e-07\\
18	2.2875232079919e-07\\
20	2.20542048997309e-07\\
};
\addlegendentry{$\|\bS-\bV\hat\bS^{\text{NI}}\|_2/\|\bS\|_2$}

\end{axis}
\end{tikzpicture}%

%% file: BC_DecaySV.tikz
%
%
\definecolor{mycolor1}{rgb}{0.00000,0.44700,0.74100}%
\definecolor{mycolor2}{rgb}{0.85000,0.32500,0.09800}%
\definecolor{mycolor3}{rgb}{0.92900,0.69400,0.12500}%
\definecolor{mycolor4}{rgb}{0.49400,0.18400,0.55600}%
\definecolor{mycolor5}{rgb}{0.46600,0.67400,0.18800}%
\begin{tikzpicture}

\begin{axis}[%
width=0.951\fwidth,
height=\fheight,
at={(0\fwidth,0\fheight)},
scale only axis,
xmin=1,
xmax=350,
xlabel style={font=\color{white!15!black}},
xlabel={$k$},
ymode=log,
ymin=2.64302792544328e-14,
ymax=7955.70092479283,
yminorticks=true,
ylabel style={font=\color{white!15!black}},
ylabel={Singular values},
axis background/.style={fill=white},
legend style={legend cell align=left, align=left, draw=white!15!black,  at = { (1.88,0.9), anchor = south west }}
]
\addplot [color=mycolor1, line width=1.2pt,mark phase = 1,mark repeat = 20, mark = +]
  table[row sep=crcr]{%
1	1010.48076857751\\
2	880.341972972118\\
3	700.409062500222\\
4	509.273322325765\\
5	340.610643314711\\
6	214.024272618234\\
7	133.841026364291\\
8	92.142181181417\\
9	71.3377288066036\\
10	55.2417716006298\\
11	40.6321347146246\\
12	29.3547065919775\\
13	21.3959291443384\\
14	15.6950710908012\\
15	11.4748869005457\\
16	8.35402616957108\\
17	6.07697505398949\\
18	4.42543243791565\\
19	3.22882436871284\\
20	2.36002912658698\\
21	1.72480564688916\\
22	1.25705955051848\\
23	0.912090909970408\\
24	0.658475807768415\\
25	0.472894448031454\\
26	0.337798474313051\\
27	0.240015463294706\\
28	0.169705410973083\\
29	0.119506400761611\\
30	0.0839140162951742\\
31	0.0588341070440921\\
32	0.0412502910049666\\
33	0.0289641595466224\\
34	0.0203936065578281\\
35	0.0144117469127676\\
36	0.010223914851111\\
37	0.00727666906629096\\
38	0.00519021326515193\\
39	0.0037058375786593\\
40	0.00264636588592331\\
41	0.00188889307549029\\
42	0.00134714823006137\\
43	0.000959949670291407\\
44	0.000683469807370151\\
45	0.000486136088122822\\
46	0.000345320741858961\\
47	0.000244965342182961\\
48	0.000173618300103608\\
49	0.000122941869568655\\
50	8.68574635259498e-05\\
51	6.11042855797319e-05\\
52	4.27833266812698e-05\\
53	2.98872212976586e-05\\
54	2.22348370691893e-05\\
55	1.99490729985652e-05\\
56	1.40976127683956e-05\\
57	9.74447410316311e-06\\
58	6.76839718519579e-06\\
59	5.05038728430787e-06\\
60	4.29869639152568e-06\\
61	3.12030106713464e-06\\
62	2.20690052470815e-06\\
63	1.5735227498299e-06\\
64	1.13515837188128e-06\\
65	8.28825460567453e-07\\
66	6.11442111241505e-07\\
67	4.53885696144782e-07\\
68	3.37165469865874e-07\\
69	2.49864902326537e-07\\
70	1.85295000877314e-07\\
71	1.38803073833967e-07\\
72	1.07978142667989e-07\\
73	9.17736772219047e-08\\
74	7.23348928507098e-08\\
75	5.37257800520368e-08\\
76	3.95910912107459e-08\\
77	2.97126302438571e-08\\
78	2.29500423321568e-08\\
79	1.73190389401767e-08\\
80	1.25380040805739e-08\\
81	8.84530364356521e-09\\
82	6.140729224161e-09\\
83	4.26403564605736e-09\\
84	3.71297407588422e-09\\
85	3.58602691014338e-09\\
86	2.91957250589815e-09\\
87	2.06811120616657e-09\\
88	1.51192893928475e-09\\
89	1.46046453973018e-09\\
90	1.2945944916938e-09\\
91	1.00206153938606e-09\\
92	7.80418444289404e-10\\
93	7.01862296443175e-10\\
94	5.55139965208342e-10\\
95	4.60487905269217e-10\\
96	4.28267458983244e-10\\
97	3.29809992104615e-10\\
98	3.033734341436e-10\\
99	2.2920951525023e-10\\
100	1.71430602483303e-10\\
101	1.48221672287855e-10\\
102	1.34148731257576e-10\\
103	1.07652956089479e-10\\
104	7.8671923450456e-11\\
105	6.27581897430841e-11\\
106	5.53175719633688e-11\\
107	5.00854869473979e-11\\
108	4.07219919955481e-11\\
109	3.73937575958853e-11\\
110	3.3575713941373e-11\\
111	2.41992515661811e-11\\
112	1.81550387788657e-11\\
113	1.55546080396792e-11\\
114	1.357355485303e-11\\
115	1.18136204154349e-11\\
116	1.04310296506423e-11\\
117	8.24916918812578e-12\\
118	6.8738327041231e-12\\
119	5.26877453962094e-12\\
120	4.9676879606741e-12\\
121	2.83846008372827e-12\\
122	2.49628559360106e-12\\
123	2.14087667732958e-12\\
124	1.57420723527031e-12\\
125	1.2303717902476e-12\\
126	1.21337373811439e-12\\
127	9.91386215298001e-13\\
128	8.75474072934398e-13\\
129	8.65232695060415e-13\\
130	7.66745331689939e-13\\
131	6.25573343619117e-13\\
132	5.33751990383002e-13\\
133	4.79893730124061e-13\\
134	4.35613473413456e-13\\
135	3.52304129710016e-13\\
136	3.22889425186768e-13\\
137	2.9434331704172e-13\\
138	2.72500153331752e-13\\
139	2.46178895652301e-13\\
140	2.31635464876591e-13\\
141	2.26517605717981e-13\\
142	2.21226534096267e-13\\
143	2.06040447548171e-13\\
144	2.01450099888459e-13\\
145	1.97204121249254e-13\\
146	1.91926843925251e-13\\
147	1.85900479065337e-13\\
148	1.76928360573934e-13\\
149	1.72829948074765e-13\\
150	1.67038448945858e-13\\
151	1.56143442065219e-13\\
152	1.50433071430032e-13\\
153	1.4742473929331e-13\\
154	1.42138874488932e-13\\
155	1.38251188957299e-13\\
156	1.33832825451606e-13\\
157	1.24336819884677e-13\\
158	1.12950428177911e-13\\
159	1.07943092746695e-13\\
160	1.05211410788159e-13\\
161	9.50833876795416e-14\\
162	9.39978198568356e-14\\
163	9.22561076730829e-14\\
164	9.05486235527575e-14\\
165	8.84451477873469e-14\\
166	8.68263585921566e-14\\
167	8.60451235704143e-14\\
168	8.5883791969696e-14\\
169	8.50340042570561e-14\\
170	8.29830707365168e-14\\
171	8.25150855396928e-14\\
172	8.25149493815734e-14\\
173	8.23136279334979e-14\\
174	7.93592277769506e-14\\
175	7.73629135886322e-14\\
176	7.55323168159416e-14\\
177	7.50946499736079e-14\\
178	7.32918634431538e-14\\
179	7.16740181223553e-14\\
180	7.1199809660249e-14\\
181	6.70322258531267e-14\\
182	6.66650035431742e-14\\
183	6.65177291992187e-14\\
184	6.51254352918756e-14\\
185	6.43192641438552e-14\\
186	6.43192641438552e-14\\
187	6.43192641438552e-14\\
188	6.43192641438552e-14\\
189	6.43192641438552e-14\\
190	6.43192641438552e-14\\
191	6.43192641438552e-14\\
192	6.43192641438552e-14\\
193	6.43192641438552e-14\\
194	6.39114595507417e-14\\
195	5.91807401338287e-14\\
196	4.96071918330552e-14\\
197	4.8668251943828e-14\\
198	4.0122557398616e-14\\
199	3.65550459614579e-14\\
200	2.64302792544328e-14\\
};
\addlegendentry{$\text{Snapshot matrix}~\bC_\text{1}$}

\addplot [color=mycolor2,  line width=1.2pt,mark phase = 5,mark repeat = 20, mark = o]
  table[row sep=crcr]{%
1	3410.22958279789\\
2	2943.52393342683\\
3	2314.49585369594\\
4	1652.96490986726\\
5	1076.92434758092\\
6	653.09202198362\\
7	394.208194305581\\
8	270.515760762778\\
9	213.396244498417\\
10	164.175379608348\\
11	117.483986339782\\
12	82.4239469565709\\
13	58.7195563726446\\
14	42.3560960485585\\
15	30.3157323538658\\
16	21.4438695092745\\
17	15.1174100030805\\
18	10.6735151430099\\
19	7.52667174292033\\
20	5.29398970313695\\
21	3.71683819100089\\
22	2.61139808368655\\
23	1.83408445321971\\
24	1.28714902680403\\
25	0.902837474393588\\
26	0.634607625516477\\
27	0.447156472249657\\
28	0.315863408038277\\
29	0.22344075894183\\
30	0.158313063744845\\
31	0.112241063562759\\
32	0.0796411782463228\\
33	0.0565594785334511\\
34	0.0402475134146849\\
35	0.0287130521784512\\
36	0.020566879754641\\
37	0.014830214304509\\
38	0.0108449089638799\\
39	0.00811878316544681\\
40	0.00613855703843665\\
41	0.00453923696211626\\
42	0.00327594952976494\\
43	0.00233210980149535\\
44	0.0016478774833939\\
45	0.00116048540909922\\
46	0.000818023214034402\\
47	0.000582430947326075\\
48	0.00042562082749507\\
49	0.000317206050440083\\
50	0.000230005720549078\\
51	0.000161257800716584\\
52	0.000113242426589429\\
53	0.000102077308790048\\
54	7.42017366819435e-05\\
55	5.02172252602723e-05\\
56	3.37675768272631e-05\\
57	2.26690063392376e-05\\
58	1.5276770049469e-05\\
59	1.04219538152718e-05\\
60	7.30213647201897e-06\\
61	5.54845505167958e-06\\
62	4.62540567890513e-06\\
63	3.52861163623186e-06\\
64	2.68501242094048e-06\\
65	2.09275587054729e-06\\
66	1.61323237215072e-06\\
67	1.21886235389122e-06\\
68	9.11743049655022e-07\\
69	6.76857358076346e-07\\
70	4.98644187596862e-07\\
71	3.67749213402758e-07\\
72	3.0204053315493e-07\\
73	2.5035220101974e-07\\
74	1.81610997454078e-07\\
75	1.29176832454401e-07\\
76	9.13585722774093e-08\\
77	6.44353741406325e-08\\
78	4.52854606358409e-08\\
79	3.16525629797301e-08\\
80	2.2034836572647e-08\\
81	1.54063449168738e-08\\
82	1.09974847296128e-08\\
83	1.02841195270639e-08\\
84	8.05773972416489e-09\\
85	6.1428389680697e-09\\
86	4.58195768872915e-09\\
87	3.26832430114968e-09\\
88	2.28199091443559e-09\\
89	1.60147279447335e-09\\
90	1.42388205306685e-09\\
91	1.33250463241109e-09\\
92	1.08168273765649e-09\\
93	7.56211066838254e-10\\
94	5.53445878320623e-10\\
95	5.09447705919827e-10\\
96	4.83652112114042e-10\\
97	3.92361674235405e-10\\
98	3.35541641458535e-10\\
99	2.50478847359178e-10\\
100	1.85092418439712e-10\\
101	1.55375397510737e-10\\
102	1.29926753309494e-10\\
103	1.09816612129584e-10\\
104	7.94203210976035e-11\\
105	5.85870764997905e-11\\
106	5.41140418764759e-11\\
107	4.50300796528912e-11\\
108	3.39721815891549e-11\\
109	3.31569831891043e-11\\
110	2.48170476482954e-11\\
111	1.62708888228208e-11\\
112	1.49054083159252e-11\\
113	1.25549638154809e-11\\
114	1.03587850486965e-11\\
115	8.91665938804699e-12\\
116	7.89929817267588e-12\\
117	7.32241351854255e-12\\
118	5.31019703272283e-12\\
119	4.40079672031799e-12\\
120	3.95114026331882e-12\\
121	3.38136198686854e-12\\
122	2.55961328089253e-12\\
123	1.98893904843502e-12\\
124	1.80879484991561e-12\\
125	1.55291158796275e-12\\
126	1.27381155783119e-12\\
127	1.16645637912954e-12\\
128	1.02026783403808e-12\\
129	9.51208715675089e-13\\
130	8.55746959347959e-13\\
131	8.19043017337041e-13\\
132	7.99946464511661e-13\\
133	7.52159050466734e-13\\
134	7.06619582060643e-13\\
135	6.99443020757862e-13\\
136	6.79040400942531e-13\\
137	6.64187402887003e-13\\
138	6.37905540784232e-13\\
139	6.18049105432388e-13\\
140	5.9483419166007e-13\\
141	5.85412150946458e-13\\
142	5.69130162365726e-13\\
143	5.65654976751398e-13\\
144	5.43535183520006e-13\\
145	5.14044794797562e-13\\
146	5.10498085912817e-13\\
147	5.05823992704926e-13\\
148	5.05529464003634e-13\\
149	4.71533284443301e-13\\
150	4.51370977403935e-13\\
151	4.45341502000006e-13\\
152	4.20549954915978e-13\\
153	4.04122777885119e-13\\
154	3.99207554995656e-13\\
155	3.83330323291712e-13\\
156	3.7445188512405e-13\\
157	3.42517426557876e-13\\
158	3.2230858636765e-13\\
159	3.20040743147517e-13\\
160	3.19044220750833e-13\\
161	2.96642861232391e-13\\
162	2.94525987459686e-13\\
163	2.92855231041713e-13\\
164	2.69667801701527e-13\\
165	2.64385977973589e-13\\
166	2.63654417342984e-13\\
167	2.59153478900858e-13\\
168	2.58401557978516e-13\\
169	2.57956538879896e-13\\
170	2.56278806775233e-13\\
171	2.53273100950688e-13\\
172	2.513614039512e-13\\
173	2.43952739388183e-13\\
174	2.37459237684874e-13\\
175	2.37038028022152e-13\\
176	2.35773776744538e-13\\
177	2.32970327774122e-13\\
178	2.31376457986122e-13\\
179	2.26875493601155e-13\\
180	2.24030545209815e-13\\
181	2.23562329038896e-13\\
182	2.1846646701543e-13\\
183	2.10978641748741e-13\\
184	2.09712898952172e-13\\
185	2.08185827728705e-13\\
186	2.08185827728705e-13\\
187	2.08185827728705e-13\\
188	2.08185827728705e-13\\
189	2.08185827728705e-13\\
190	2.08185827728705e-13\\
191	2.08185827728705e-13\\
192	2.08185827728705e-13\\
193	2.08185827728705e-13\\
194	1.87546040934803e-13\\
195	1.83719344249213e-13\\
196	1.50428418554323e-13\\
197	1.31161863511168e-13\\
198	1.2299319056592e-13\\
199	1.03554751914993e-13\\
200	9.57201888098986e-14\\
};
\addlegendentry{$\text{Snapshot matrix}~\bQ_\text{1}$}

\addplot [color=mycolor3, line width=1.2pt,mark phase = 9,mark repeat = 20, mark = diamond]
  table[row sep=crcr]{%
1	1147.11733183916\\
2	977.782068963694\\
3	752.457534224375\\
4	527.785894526225\\
5	344.977398036624\\
6	219.788519547038\\
7	146.248093223427\\
8	107.366751820653\\
9	83.4391362317714\\
10	62.7339560887447\\
11	45.3615150860219\\
12	32.8750036457024\\
13	24.0823965031518\\
14	17.5628906354564\\
15	12.6981161687897\\
16	9.16249605935009\\
17	6.62910595350801\\
18	4.81406731444785\\
19	3.50699347858292\\
20	2.55401683100882\\
21	1.85144331120914\\
22	1.33367152044598\\
23	0.954735255454111\\
24	0.679452837081307\\
25	0.480998721749856\\
26	0.339191906027034\\
27	0.238773960926258\\
28	0.168174942519169\\
29	0.118740658107569\\
30	0.0841287015747261\\
31	0.0597984737651656\\
32	0.0425894690952231\\
33	0.030358974384541\\
34	0.0216609996978152\\
35	0.0154903655773068\\
36	0.0111193747267108\\
37	0.00801458565181423\\
38	0.00579404361361313\\
39	0.00419421827560919\\
40	0.00303698804336716\\
41	0.00220058374647484\\
42	0.0015986535654259\\
43	0.00116726341532467\\
44	0.000857411841051178\\
45	0.000631605657594535\\
46	0.000463786671443094\\
47	0.000338109386073618\\
48	0.00024450153509581\\
49	0.000175576369424996\\
50	0.000126411506373512\\
51	0.000102036717094294\\
52	8.40409329709041e-05\\
53	5.98908161381243e-05\\
54	4.18301997510019e-05\\
55	2.90215778711507e-05\\
56	2.00734984570891e-05\\
57	1.38923051164722e-05\\
58	9.66302256919551e-06\\
59	6.78869112657053e-06\\
60	4.84508153271266e-06\\
61	3.55954361683615e-06\\
62	2.86649308135239e-06\\
63	2.39585601161343e-06\\
64	1.84718172942927e-06\\
65	1.41433613810976e-06\\
66	1.08476666752172e-06\\
67	8.29019485146967e-07\\
68	6.32348930835031e-07\\
69	4.8788180643937e-07\\
70	3.82132300247354e-07\\
71	2.93540851837941e-07\\
72	2.18744491704581e-07\\
73	1.60541620018028e-07\\
74	1.17373606882709e-07\\
75	8.61837121670197e-08\\
76	6.48862489616258e-08\\
77	5.53281484782127e-08\\
78	4.46667224358365e-08\\
79	3.33991699424293e-08\\
80	2.47761311391271e-08\\
81	1.83533752318032e-08\\
82	1.36010003183758e-08\\
83	1.00891288779339e-08\\
84	7.48060719077098e-09\\
85	5.52973522143407e-09\\
86	4.07889964715442e-09\\
87	3.45244613279429e-09\\
88	2.94501481507693e-09\\
89	2.15772038749622e-09\\
90	1.5739481691237e-09\\
91	1.14673074308588e-09\\
92	8.34480340599962e-10\\
93	6.24048079274247e-10\\
94	6.1737995609514e-10\\
95	6.00746221944932e-10\\
96	4.41576112515796e-10\\
97	3.88735775330153e-10\\
98	3.31676814827476e-10\\
99	3.09272726049551e-10\\
100	3.02523722289367e-10\\
101	2.24655850647951e-10\\
102	1.62778294955043e-10\\
103	1.22418583284362e-10\\
104	1.13593901328198e-10\\
105	1.07424180280653e-10\\
106	8.25743783375353e-11\\
107	6.10292143482927e-11\\
108	5.04852801791818e-11\\
109	4.67262065712709e-11\\
110	4.09707278547629e-11\\
111	3.26476856989324e-11\\
112	2.83095093819501e-11\\
113	2.45106525049308e-11\\
114	2.16398618136148e-11\\
115	2.01534671423862e-11\\
116	1.5551447053241e-11\\
117	1.36493133725214e-11\\
118	1.21536054093664e-11\\
119	9.82957492054672e-12\\
120	7.61124683624837e-12\\
121	6.77162594469456e-12\\
122	5.38364656471559e-12\\
123	3.7040478453126e-12\\
124	3.49449028259285e-12\\
125	2.70734861666873e-12\\
126	2.06850218827166e-12\\
127	2.05056209742249e-12\\
128	1.6374826934058e-12\\
129	1.35307076979186e-12\\
130	1.05434264453123e-12\\
131	8.48597295849799e-13\\
132	7.02337157535624e-13\\
133	6.08468992451774e-13\\
134	4.84503069325889e-13\\
135	4.05675735820922e-13\\
136	3.37834588643452e-13\\
137	2.87336141401095e-13\\
138	2.73999549526066e-13\\
139	2.57440498082381e-13\\
140	2.47565801644023e-13\\
141	2.44263108695849e-13\\
142	2.42465426932163e-13\\
143	2.26991382713126e-13\\
144	2.221499789418e-13\\
145	2.04503960304738e-13\\
146	2.00964033788233e-13\\
147	1.95486925400462e-13\\
148	1.93976532868786e-13\\
149	1.85146723251621e-13\\
150	1.81442741632023e-13\\
151	1.77714942622457e-13\\
152	1.70150684628308e-13\\
153	1.59899689794589e-13\\
154	1.59529220510475e-13\\
155	1.48405937900766e-13\\
156	1.45821564124633e-13\\
157	1.42266156206534e-13\\
158	1.34033044329107e-13\\
159	1.26593845387802e-13\\
160	1.18940890243519e-13\\
161	1.14340066327822e-13\\
162	1.01550731413919e-13\\
163	9.78312072484344e-14\\
164	9.53704081556547e-14\\
165	9.35162803794221e-14\\
166	9.15315625542111e-14\\
167	9.03533781968011e-14\\
168	8.88671465710912e-14\\
169	8.88081474832394e-14\\
170	8.67251899379881e-14\\
171	8.6482333877594e-14\\
172	8.63761045865459e-14\\
173	8.34416253978492e-14\\
174	8.32573226358586e-14\\
175	8.24296903522748e-14\\
176	8.23717060939328e-14\\
177	7.98514898693083e-14\\
178	7.83930676274672e-14\\
179	7.71017846228819e-14\\
180	7.69233401562917e-14\\
181	7.67110948257279e-14\\
182	7.47039800506795e-14\\
183	7.46030955407549e-14\\
184	7.35987811500329e-14\\
185	7.35760361560659e-14\\
186	7.35254030480045e-14\\
187	7.29474994238949e-14\\
188	7.29474994238949e-14\\
189	7.29474994238949e-14\\
190	7.29474994238949e-14\\
191	7.29474994238949e-14\\
192	7.29474994238949e-14\\
193	6.79091097765712e-14\\
194	6.52544250678463e-14\\
195	6.33665250590267e-14\\
196	5.78719411905338e-14\\
197	4.68470148636153e-14\\
198	4.31979215368887e-14\\
199	4.03616844819385e-14\\
200	3.26997421285959e-14\\
};
\addlegendentry{$\text{Snapshot matrix}~\bC_\text{2}$}

\addplot [color=mycolor4, line width=1.2pt,mark phase = 13,mark repeat = 20, mark = triangle]
  table[row sep=crcr]{%
1	2766.55421198873\\
2	2333.55129201766\\
3	1770.73022208513\\
4	1219.3674252386\\
5	782.42298737938\\
6	494.725962865688\\
7	334.221388523062\\
8	251.622189081203\\
9	195.93238039469\\
10	144.29037218874\\
11	102.193980718103\\
12	73.0482700088207\\
13	52.8214581862468\\
14	37.7399236245942\\
15	26.4918066105869\\
16	18.4878620532421\\
17	12.9029394045111\\
18	8.97112211459025\\
19	6.1999176376211\\
20	4.27143331354387\\
21	2.9420519065428\\
22	2.02585837815993\\
23	1.39568887415954\\
24	0.964912167908439\\
25	0.671530831644762\\
26	0.471204023394693\\
27	0.333792954939055\\
28	0.238753275527018\\
29	0.172037567141828\\
30	0.124272999625176\\
31	0.0895564788747075\\
32	0.0642505047286576\\
33	0.045972483535118\\
34	0.0329372848831923\\
35	0.0237230431904692\\
36	0.0172008181745157\\
37	0.0125222594929536\\
38	0.00911361392390054\\
39	0.00664605815240224\\
40	0.00495744667545814\\
41	0.00386585239887954\\
42	0.00297031941161232\\
43	0.00224060108990338\\
44	0.00179092380982659\\
45	0.00140743008411885\\
46	0.00103300762275691\\
47	0.000741612909215914\\
48	0.000526798808638951\\
49	0.00037125897119504\\
50	0.000260385169455778\\
51	0.000182668168356598\\
52	0.000129688221732604\\
53	9.72381979369376e-05\\
54	7.66036405490295e-05\\
55	5.56270176844175e-05\\
56	3.85539745344185e-05\\
57	2.63818163658778e-05\\
58	1.80500745885424e-05\\
59	1.24361010364516e-05\\
60	8.68301232599685e-06\\
61	6.20177876097268e-06\\
62	4.57621492511695e-06\\
63	3.49929227360643e-06\\
64	2.76483940315814e-06\\
65	2.26620509570854e-06\\
66	1.88704739999443e-06\\
67	1.50936415776288e-06\\
68	1.16584367027667e-06\\
69	8.86714785145727e-07\\
70	6.79386494859528e-07\\
71	5.40760918844061e-07\\
72	4.27731256234715e-07\\
73	3.22087494173077e-07\\
74	2.37069237153124e-07\\
75	1.73285227835253e-07\\
76	1.26623475560858e-07\\
77	9.27728260670951e-08\\
78	6.81895057411278e-08\\
79	5.02961097071626e-08\\
80	3.74924830763784e-08\\
81	2.98591723701887e-08\\
82	2.53513448644152e-08\\
83	1.94676513961488e-08\\
84	1.46008363513247e-08\\
85	1.09409090256591e-08\\
86	8.23554834607244e-09\\
87	6.24298216654061e-09\\
88	4.7341171879107e-09\\
89	3.58784242793808e-09\\
90	2.95726218512281e-09\\
91	2.48818636611442e-09\\
92	1.85903269845751e-09\\
93	1.35929441564584e-09\\
94	9.86349220431914e-10\\
95	7.12061187810782e-10\\
96	5.12050259808957e-10\\
97	3.67446522087835e-10\\
98	2.64003739540783e-10\\
99	1.92769343307041e-10\\
100	1.73079037043367e-10\\
101	1.68600191081022e-10\\
102	1.46207695245822e-10\\
103	1.22204157938578e-10\\
104	1.00017278095941e-10\\
105	9.17646276321387e-11\\
106	8.48175973051201e-11\\
107	7.80689862490525e-11\\
108	6.18313283583614e-11\\
109	4.45716372652532e-11\\
110	3.29955228637068e-11\\
111	2.98990660708319e-11\\
112	2.77440456280035e-11\\
113	2.16140969069346e-11\\
114	1.60937899326567e-11\\
115	1.57077420094384e-11\\
116	1.28259497075654e-11\\
117	1.05297341574905e-11\\
118	9.10078893968522e-12\\
119	8.08816625952896e-12\\
120	6.97720438401795e-12\\
121	6.44916414617207e-12\\
122	5.09827431910672e-12\\
123	3.75814202602729e-12\\
124	2.98865587429317e-12\\
125	2.54182761630541e-12\\
126	2.11446268652142e-12\\
127	1.62960083274892e-12\\
128	1.26462757255451e-12\\
129	1.10491936747433e-12\\
130	1.0386142313916e-12\\
131	8.60935131245768e-13\\
132	7.21658821246872e-13\\
133	6.77163932660451e-13\\
134	6.01245971191316e-13\\
135	5.93223944649815e-13\\
136	5.59355227674227e-13\\
137	5.52927779964392e-13\\
138	5.30083427596183e-13\\
139	5.10960753397935e-13\\
140	4.94751958078732e-13\\
141	4.80428292357824e-13\\
142	4.69006587624458e-13\\
143	4.43436388535065e-13\\
144	4.39284853009143e-13\\
145	4.12032753563979e-13\\
146	4.06172003921789e-13\\
147	3.93486493435877e-13\\
148	3.907682720924e-13\\
149	3.70569099448519e-13\\
150	3.56331176666893e-13\\
151	3.52329255620214e-13\\
152	3.26904949866619e-13\\
153	3.2681481959116e-13\\
154	3.06109789799605e-13\\
155	2.94091957698369e-13\\
156	2.7318726684599e-13\\
157	2.60965879112891e-13\\
158	2.46938778106114e-13\\
159	2.2566805953051e-13\\
160	2.22317607013498e-13\\
161	2.17421589985925e-13\\
162	2.15825625591694e-13\\
163	2.13709354614912e-13\\
164	2.12327711382128e-13\\
165	2.09527216251017e-13\\
166	2.06242993379159e-13\\
167	1.99843097558438e-13\\
168	1.99186366821975e-13\\
169	1.97572547012056e-13\\
170	1.97116919155635e-13\\
171	1.9703545682165e-13\\
172	1.92552418065031e-13\\
173	1.84814409559593e-13\\
174	1.84669739996058e-13\\
175	1.76988016338704e-13\\
176	1.75676815238952e-13\\
177	1.71437018570213e-13\\
178	1.71437018570213e-13\\
179	1.71437018570213e-13\\
180	1.71437018570213e-13\\
181	1.71437018570213e-13\\
182	1.71437018570213e-13\\
183	1.71437018570213e-13\\
184	1.71437018570213e-13\\
185	1.71437018570213e-13\\
186	1.71437018570213e-13\\
187	1.71437018570213e-13\\
188	1.71437018570213e-13\\
189	1.71437018570213e-13\\
190	1.71437018570213e-13\\
191	1.71437018570213e-13\\
192	1.64539055615938e-13\\
193	1.57420359463359e-13\\
194	1.54360071358505e-13\\
195	1.32364551417592e-13\\
196	1.28254495542938e-13\\
197	1.14316462875149e-13\\
198	9.38595465586555e-14\\
199	8.83567163801913e-14\\
200	7.87015196181102e-14\\
};
\addlegendentry{$\text{Snapshot matrix}~\bQ_\text{2}$}

\addplot [color=mycolor5, line width=1.2pt,mark phase = 17,mark repeat = 20, mark = *]
  table[row sep=crcr]{%
1	7955.70092479283\\
2	7346.64876347815\\
3	5856.40784626988\\
4	5248.32670168142\\
5	4397.95504315782\\
6	3827.85190920668\\
7	3472.98801980046\\
8	3082.29756929005\\
9	2649.19636214644\\
10	2320.06216506491\\
11	2022.35311346338\\
12	1668.95195802809\\
13	1336.3076245505\\
14	1058.83489623056\\
15	832.638481815618\\
16	654.068129965128\\
17	512.440772047622\\
18	401.492689308572\\
19	312.468017339983\\
20	241.266353858136\\
21	184.54352606017\\
22	140.274224819394\\
23	106.997432947485\\
24	84.0942492284287\\
25	68.7472556145272\\
26	56.1510079180298\\
27	47.9206673714823\\
28	38.6377105443057\\
29	29.5425263829581\\
30	22.6548849638251\\
31	17.8941176807711\\
32	14.4502049880174\\
33	11.5489531187101\\
34	9.19072405939252\\
35	7.58683362149791\\
36	6.31210751477896\\
37	5.08456025302541\\
38	4.00055209729222\\
39	3.11489182944974\\
40	2.41865597990249\\
41	1.86673283081613\\
42	1.43338524652162\\
43	1.10640692895707\\
44	0.881795277047138\\
45	0.742011085245546\\
46	0.697493684222576\\
47	0.648737975622157\\
48	0.578723001278032\\
49	0.470797052592923\\
50	0.378219721976691\\
51	0.303840831564952\\
52	0.251292582480619\\
53	0.212961049065844\\
54	0.182558703592605\\
55	0.155233824711675\\
56	0.127856421140829\\
57	0.101752158217543\\
58	0.0804285026868728\\
59	0.0636666236538685\\
60	0.0500750620702554\\
61	0.0390311051103702\\
62	0.0302692314001168\\
63	0.0234290231452247\\
64	0.0183792602667076\\
65	0.0151477928368793\\
66	0.0132095681011759\\
67	0.0113941491606181\\
68	0.0104531938077358\\
69	0.00875208753215275\\
70	0.00684999907669988\\
71	0.00530911331499672\\
72	0.00408302495040176\\
73	0.00322061286338039\\
74	0.00311540952615962\\
75	0.00239182733851869\\
76	0.00183386296156448\\
77	0.001380600258885\\
78	0.0010336081226009\\
79	0.000807832505140973\\
80	0.000688044154917027\\
81	0.000588743119583125\\
82	0.000489077033009447\\
83	0.000463407871089118\\
84	0.000398890058847689\\
85	0.00032139687445518\\
86	0.000253991454239083\\
87	0.000202879494579794\\
88	0.000166497529226264\\
89	0.000138798324404716\\
90	0.000116164174875866\\
91	9.70923753925267e-05\\
92	8.23690843123189e-05\\
93	7.14770807886175e-05\\
94	6.09996797607192e-05\\
95	5.14847124117126e-05\\
96	4.27971240966943e-05\\
97	3.44088607624922e-05\\
98	2.71604019219224e-05\\
99	2.16643516336587e-05\\
100	1.75555356157358e-05\\
101	1.42156572323033e-05\\
102	1.16473949709224e-05\\
103	1.01396276774108e-05\\
104	9.5500563312565e-06\\
105	7.63285441227683e-06\\
106	5.99644593143666e-06\\
107	4.73787753675869e-06\\
108	3.80106153241456e-06\\
109	3.08861934974108e-06\\
110	2.51037005236153e-06\\
111	2.03950077153487e-06\\
112	1.66867898579892e-06\\
113	1.36239939125649e-06\\
114	1.09814813621083e-06\\
115	8.7912372236792e-07\\
116	7.06944018091509e-07\\
117	5.77078096171623e-07\\
118	5.4144371740512e-07\\
119	5.38144098419512e-07\\
120	4.96187300097741e-07\\
121	4.20883189818439e-07\\
122	3.3961142138767e-07\\
123	2.79137288195789e-07\\
124	2.43419951061548e-07\\
125	2.27897244098721e-07\\
126	2.2468847004673e-07\\
127	2.02021754111655e-07\\
128	1.91120722710611e-07\\
129	1.52336061524262e-07\\
130	1.2333615483533e-07\\
131	1.10845609611301e-07\\
132	9.5320196493266e-08\\
133	9.31193985791531e-08\\
134	8.80734569369712e-08\\
135	7.11056678533526e-08\\
136	6.75278218329694e-08\\
137	6.71047306086554e-08\\
138	5.83445116859799e-08\\
139	4.88735655776281e-08\\
140	4.27798818106332e-08\\
141	3.53808875673696e-08\\
142	3.49504238038281e-08\\
143	3.46264377501688e-08\\
144	2.83306066047209e-08\\
145	2.77844943649371e-08\\
146	2.65594278188193e-08\\
147	2.11044136441717e-08\\
148	1.67495876039131e-08\\
149	1.42908533137698e-08\\
150	1.38313574095023e-08\\
151	1.30472989677697e-08\\
152	1.13088061325934e-08\\
153	9.7891200047759e-09\\
154	8.43696507281487e-09\\
155	7.67661221640328e-09\\
156	7.59505928761252e-09\\
157	6.81463416626353e-09\\
158	5.84810909896992e-09\\
159	5.81754268983951e-09\\
160	5.52986310012539e-09\\
161	4.96140854418176e-09\\
162	4.90837785022363e-09\\
163	4.36361199226121e-09\\
164	3.53816248204572e-09\\
165	2.85757772118265e-09\\
166	2.56684209918349e-09\\
167	2.5229234986292e-09\\
168	2.46718947944387e-09\\
169	2.33456405284544e-09\\
170	2.11938065526856e-09\\
171	1.77575024202609e-09\\
172	1.71419160021158e-09\\
173	1.62387582827365e-09\\
174	1.53585995380948e-09\\
175	1.50617823119852e-09\\
176	1.32263510678284e-09\\
177	1.18652425496956e-09\\
178	1.18138806864202e-09\\
179	1.10266556399349e-09\\
180	1.0416373746431e-09\\
181	9.20679240419452e-10\\
182	8.43554235674031e-10\\
183	7.49244861240706e-10\\
184	6.6416948035254e-10\\
185	5.8384810787537e-10\\
186	5.77515800654783e-10\\
187	5.60240754777278e-10\\
188	5.34212271563516e-10\\
189	4.90424694642197e-10\\
190	4.64446932986403e-10\\
191	4.29189273742829e-10\\
192	3.96451902067765e-10\\
193	3.58037557787697e-10\\
194	3.51743060572776e-10\\
195	3.19617192317502e-10\\
196	3.1082182751159e-10\\
197	2.58149123500422e-10\\
198	2.41064044232518e-10\\
199	1.99745494293533e-10\\
200	1.79314876680244e-10\\
201	1.7025995519123e-10\\
202	1.55450099866577e-10\\
203	1.39749600653424e-10\\
204	1.23423864745391e-10\\
205	1.16545509243502e-10\\
206	1.0960585551203e-10\\
207	1.07714512577174e-10\\
208	9.86691278567969e-11\\
209	9.41263921551441e-11\\
210	7.49954955361927e-11\\
211	7.29545624642089e-11\\
212	6.31487403481423e-11\\
213	6.16557137769839e-11\\
214	5.15391714966088e-11\\
215	4.79824248369731e-11\\
216	4.12552698620884e-11\\
217	3.67791335439618e-11\\
218	3.53968158647016e-11\\
219	3.10511850338223e-11\\
220	2.99539210154878e-11\\
221	2.84373586094177e-11\\
222	2.66036824335543e-11\\
223	2.39766380888423e-11\\
224	2.13782851192818e-11\\
225	1.91768338437509e-11\\
226	1.86326953005829e-11\\
227	1.56188660788825e-11\\
228	1.29607663613542e-11\\
229	1.25848187778736e-11\\
230	1.0921336903147e-11\\
231	9.53533471128915e-12\\
232	9.29434033358478e-12\\
233	9.04550045212684e-12\\
234	8.84269588014537e-12\\
235	8.44935609160224e-12\\
236	7.62427118363171e-12\\
237	7.06608187498932e-12\\
238	6.50812991330875e-12\\
239	5.88619512507933e-12\\
240	5.19079629125187e-12\\
241	5.18103398046656e-12\\
242	4.8308666708765e-12\\
243	4.7421360466488e-12\\
244	4.42623771340821e-12\\
245	4.1906907041148e-12\\
246	4.13041824835858e-12\\
247	4.02826977712679e-12\\
248	4.01164597887098e-12\\
249	3.94294657754297e-12\\
250	3.82703313341747e-12\\
251	3.79079598813632e-12\\
252	3.75311525562652e-12\\
253	3.71179193100284e-12\\
254	3.69519572194391e-12\\
255	3.62767561411344e-12\\
256	3.60017182349928e-12\\
257	3.57509583841773e-12\\
258	3.54066248731661e-12\\
259	3.52086044721032e-12\\
260	3.47825519177204e-12\\
261	3.45767920378162e-12\\
262	3.43437182881882e-12\\
263	3.39662808245624e-12\\
264	3.36800186175132e-12\\
265	3.34320734132528e-12\\
266	3.33429365089283e-12\\
267	3.33221189090217e-12\\
268	3.30282936360587e-12\\
269	3.28363296680347e-12\\
270	3.24527482828635e-12\\
271	3.23609084692144e-12\\
272	3.22511248297176e-12\\
273	3.21547167439791e-12\\
274	3.19633004564816e-12\\
275	3.18157911736408e-12\\
276	3.16520845695093e-12\\
277	3.15527346020064e-12\\
278	3.13928838809195e-12\\
279	3.12052286369055e-12\\
280	3.10546373949365e-12\\
281	3.08210893534133e-12\\
282	3.07420586670948e-12\\
283	3.06628202676591e-12\\
284	3.04261751710706e-12\\
285	3.03535777094093e-12\\
286	3.02806714734603e-12\\
287	3.0157211362489e-12\\
288	3.00719644719953e-12\\
289	2.97739577161773e-12\\
290	2.969140615321e-12\\
291	2.96207089590944e-12\\
292	2.94899529341993e-12\\
293	2.9280622703083e-12\\
294	2.91882971562899e-12\\
295	2.91632037508717e-12\\
296	2.90056325017579e-12\\
297	2.88783494782751e-12\\
298	2.88186727373103e-12\\
299	2.86941304397958e-12\\
300	2.86064271324891e-12\\
301	2.85523221324873e-12\\
302	2.85379909449507e-12\\
303	2.84081617636903e-12\\
304	2.83252604482098e-12\\
305	2.81879350450533e-12\\
306	2.81029719068908e-12\\
307	2.79384198072475e-12\\
308	2.78699422779874e-12\\
309	2.77650723694875e-12\\
310	2.76457198912866e-12\\
311	2.75433655258577e-12\\
312	2.75341553621889e-12\\
313	2.72905410522954e-12\\
314	2.72517636330818e-12\\
315	2.71310454928005e-12\\
316	2.70435573172545e-12\\
317	2.6959199029824e-12\\
318	2.68520010168426e-12\\
319	2.66783220550339e-12\\
320	2.6588230272511e-12\\
321	2.638068699754e-12\\
322	2.63280053218657e-12\\
323	2.62218519896519e-12\\
324	2.60985368332896e-12\\
325	2.59300770003496e-12\\
326	2.57967245474339e-12\\
327	2.55144222527027e-12\\
328	2.54357189022243e-12\\
329	2.54055678999286e-12\\
330	2.53730139776324e-12\\
331	2.52325718974017e-12\\
332	2.52010492990167e-12\\
333	2.49415048598841e-12\\
334	2.4764881872319e-12\\
335	2.46949607107826e-12\\
336	2.46038249000856e-12\\
337	2.44593245150673e-12\\
338	2.42268447435496e-12\\
339	2.40611345145157e-12\\
340	2.40261821880823e-12\\
341	2.39408396997105e-12\\
342	2.3729266951494e-12\\
343	2.3684831867766e-12\\
344	2.35596927701494e-12\\
345	2.33465777765644e-12\\
346	2.32709083726298e-12\\
347	2.31477181811244e-12\\
348	2.30075745954014e-12\\
349	2.29442340484012e-12\\
350	2.28830891057148e-12\\
351	2.27509138981938e-12\\
352	2.26972884059176e-12\\
353	2.2621668126544e-12\\
354	2.25414003746817e-12\\
355	2.22716168621618e-12\\
356	2.22017194168306e-12\\
357	2.20271019765558e-12\\
358	2.19620674434751e-12\\
359	2.19351600688585e-12\\
360	2.18417563219295e-12\\
361	2.16510465012332e-12\\
362	2.15215589431148e-12\\
363	2.14077574170663e-12\\
364	2.1265821460009e-12\\
365	2.12637677599598e-12\\
366	2.11386877906296e-12\\
367	2.10315500104239e-12\\
368	2.08490263761285e-12\\
369	2.0824113982671e-12\\
370	2.06868755341915e-12\\
371	2.06143244547793e-12\\
372	2.05422379604467e-12\\
373	2.04353443462871e-12\\
374	2.03231215909265e-12\\
375	2.02506964838034e-12\\
376	2.02112361891918e-12\\
377	2.00422463126872e-12\\
378	2.00087062341808e-12\\
379	1.99914997010901e-12\\
380	1.99302398304954e-12\\
381	1.98526751157508e-12\\
382	1.97641795117954e-12\\
383	1.96734953185202e-12\\
384	1.96596195726422e-12\\
385	1.96415337572692e-12\\
386	1.95396397336649e-12\\
387	1.94905447872865e-12\\
388	1.94380684587682e-12\\
389	1.93893717120055e-12\\
390	1.93555361942126e-12\\
391	1.93172007153489e-12\\
392	1.92736183213932e-12\\
393	1.91160993236533e-12\\
394	1.90429299009007e-12\\
395	1.89023836320091e-12\\
396	1.88735592786126e-12\\
397	1.88197184378595e-12\\
398	1.8710514201098e-12\\
399	1.86405116039702e-12\\
400	1.85531204044224e-12\\
401	1.84530054195931e-12\\
402	1.83312985587403e-12\\
403	1.83243447349945e-12\\
404	1.80439434391181e-12\\
405	1.7968140607166e-12\\
406	1.77932707491826e-12\\
407	1.72521965017943e-12\\
408	1.68296884676368e-12\\
409	1.66122245024487e-12\\
410	1.58803687341542e-12\\
411	1.46536398383806e-12\\
412	1.39761300045747e-12\\
413	1.35522117657821e-12\\
414	1.27217901244073e-12\\
415	1.2583173685636e-12\\
416	1.22574429209765e-12\\
417	1.1947249570582e-12\\
418	1.16636905511461e-12\\
419	1.14420592270943e-12\\
420	1.11351731110116e-12\\
421	1.08330628008247e-12\\
422	1.07609029991263e-12\\
423	1.05563434305995e-12\\
424	1.046080486022e-12\\
425	1.03334410729632e-12\\
426	1.02262232839261e-12\\
427	9.9401842958913e-13\\
428	9.90517061690264e-13\\
429	9.76369276661723e-13\\
430	9.5962527358067e-13\\
431	9.50729938435468e-13\\
432	9.36700624161091e-13\\
433	9.26021978381551e-13\\
434	9.15093557075546e-13\\
435	9.00056819707932e-13\\
436	8.91158842787663e-13\\
437	8.7296203327468e-13\\
438	8.55377174715147e-13\\
439	8.39454729468723e-13\\
440	8.27412597710766e-13\\
441	8.00625289464013e-13\\
442	7.90135979885923e-13\\
443	7.72217191096789e-13\\
444	7.62484842035636e-13\\
445	7.55898878944028e-13\\
446	7.54951445737886e-13\\
447	7.37947959329245e-13\\
448	7.3781042946388e-13\\
449	7.21346089007187e-13\\
450	7.11069437753252e-13\\
451	6.91273818147357e-13\\
452	6.86645891307074e-13\\
453	6.78204862253205e-13\\
454	6.55848524688145e-13\\
455	6.40939093052632e-13\\
456	6.18249968666683e-13\\
457	6.10037427883093e-13\\
458	5.82916218984186e-13\\
459	5.67577396786435e-13\\
460	5.55701529896831e-13\\
461	5.35499845435663e-13\\
462	5.35499845435663e-13\\
463	5.35499845435663e-13\\
464	5.35499845435663e-13\\
465	5.35499845435663e-13\\
466	5.35499845435663e-13\\
467	5.35499845435663e-13\\
468	5.35499845435663e-13\\
469	5.35499845435663e-13\\
470	5.35499845435663e-13\\
471	5.35499845435663e-13\\
472	5.35499845435663e-13\\
473	5.35499845435663e-13\\
474	5.35499845435663e-13\\
475	5.35499845435663e-13\\
476	5.35499845435663e-13\\
477	5.35499845435663e-13\\
478	5.35499845435663e-13\\
479	5.35499845435663e-13\\
480	5.35499845435663e-13\\
481	5.35499845435663e-13\\
482	5.35499845435663e-13\\
483	5.35499845435663e-13\\
484	5.35499845435663e-13\\
485	5.35499845435663e-13\\
486	5.35499845435663e-13\\
487	5.35499845435663e-13\\
488	5.35499845435663e-13\\
489	5.35499845435663e-13\\
490	5.35499845435663e-13\\
491	5.35499845435663e-13\\
492	5.35499845435663e-13\\
493	5.35499845435663e-13\\
494	5.35499845435663e-13\\
495	5.35499845435663e-13\\
496	5.35499845435663e-13\\
497	5.35499845435663e-13\\
498	5.35499845435663e-13\\
499	5.35499845435663e-13\\
500	5.35499845435663e-13\\
501	5.35499845435663e-13\\
502	5.35499845435663e-13\\
503	5.35499845435663e-13\\
504	5.35499845435663e-13\\
505	5.35499845435663e-13\\
506	5.35499845435663e-13\\
507	5.35499845435663e-13\\
508	5.35499845435663e-13\\
509	5.35499845435663e-13\\
510	5.35499845435663e-13\\
511	5.35499845435663e-13\\
512	5.35499845435663e-13\\
513	5.35499845435663e-13\\
514	5.35499845435663e-13\\
515	5.35499845435663e-13\\
516	5.35499845435663e-13\\
517	5.35499845435663e-13\\
518	5.35499845435663e-13\\
519	5.35499845435663e-13\\
520	5.35499845435663e-13\\
521	5.35499845435663e-13\\
522	5.35499845435663e-13\\
523	5.35499845435663e-13\\
524	5.35499845435663e-13\\
525	5.35499845435663e-13\\
526	5.35499845435663e-13\\
527	5.35499845435663e-13\\
528	5.35499845435663e-13\\
529	5.35499845435663e-13\\
530	5.35499845435663e-13\\
531	5.35499845435663e-13\\
532	5.35499845435663e-13\\
533	5.35499845435663e-13\\
534	5.35499845435663e-13\\
535	5.35499845435663e-13\\
536	5.35499845435663e-13\\
537	5.35499845435663e-13\\
538	5.35499845435663e-13\\
539	5.35499845435663e-13\\
540	5.35499845435663e-13\\
541	5.35499845435663e-13\\
542	5.35499845435663e-13\\
543	5.35499845435663e-13\\
544	5.35499845435663e-13\\
545	5.35499845435663e-13\\
546	5.35499845435663e-13\\
547	5.35499845435663e-13\\
548	5.35499845435663e-13\\
549	5.35499845435663e-13\\
550	5.35499845435663e-13\\
551	5.35499845435663e-13\\
552	5.35499845435663e-13\\
553	5.35499845435663e-13\\
554	5.35499845435663e-13\\
555	5.35499845435663e-13\\
556	5.35499845435663e-13\\
557	5.35499845435663e-13\\
558	5.35499845435663e-13\\
559	5.35499845435663e-13\\
560	5.35499845435663e-13\\
561	5.35499845435663e-13\\
562	5.35499845435663e-13\\
563	5.35499845435663e-13\\
564	5.35499845435663e-13\\
565	5.35499845435663e-13\\
566	5.35499845435663e-13\\
567	5.35499845435663e-13\\
568	5.35499845435663e-13\\
569	5.35499845435663e-13\\
570	5.35499845435663e-13\\
571	5.35499845435663e-13\\
572	5.35499845435663e-13\\
573	5.35499845435663e-13\\
574	5.35499845435663e-13\\
575	5.35499845435663e-13\\
576	5.35499845435663e-13\\
577	5.35499845435663e-13\\
578	5.35499845435663e-13\\
579	5.35499845435663e-13\\
580	5.35499845435663e-13\\
581	5.35499845435663e-13\\
582	5.35499845435663e-13\\
583	5.35499845435663e-13\\
584	5.35499845435663e-13\\
585	5.35499845435663e-13\\
586	5.35499845435663e-13\\
587	5.35499845435663e-13\\
588	5.35499845435663e-13\\
589	5.35499845435663e-13\\
590	5.35499845435663e-13\\
591	5.35499845435663e-13\\
592	5.35499845435663e-13\\
593	5.35499845435663e-13\\
594	5.35499845435663e-13\\
595	5.35499845435663e-13\\
596	5.35499845435663e-13\\
597	5.35499845435663e-13\\
598	5.35499845435663e-13\\
599	5.35499845435663e-13\\
600	5.35499845435663e-13\\
601	5.35499845435663e-13\\
602	5.35499845435663e-13\\
603	5.35499845435663e-13\\
604	5.35499845435663e-13\\
605	5.35499845435663e-13\\
606	5.35499845435663e-13\\
607	5.35499845435663e-13\\
608	5.35499845435663e-13\\
609	5.35499845435663e-13\\
610	5.35499845435663e-13\\
611	5.35499845435663e-13\\
612	5.35499845435663e-13\\
613	5.35499845435663e-13\\
614	5.35499845435663e-13\\
615	5.35499845435663e-13\\
616	5.35499845435663e-13\\
617	5.35499845435663e-13\\
618	5.35499845435663e-13\\
619	5.35499845435663e-13\\
620	5.35499845435663e-13\\
621	5.35499845435663e-13\\
622	5.35499845435663e-13\\
623	5.35499845435663e-13\\
624	5.35499845435663e-13\\
625	5.35499845435663e-13\\
626	5.35499845435663e-13\\
627	5.35499845435663e-13\\
628	5.35499845435663e-13\\
629	5.35499845435663e-13\\
630	5.35499845435663e-13\\
631	5.35499845435663e-13\\
632	5.35499845435663e-13\\
633	5.35499845435663e-13\\
634	5.35499845435663e-13\\
635	5.35499845435663e-13\\
636	5.35499845435663e-13\\
637	5.35499845435663e-13\\
638	5.35499845435663e-13\\
639	5.35499845435663e-13\\
640	5.35499845435663e-13\\
641	5.35499845435663e-13\\
642	5.35499845435663e-13\\
643	5.35499845435663e-13\\
644	5.35499845435663e-13\\
645	5.35499845435663e-13\\
646	5.35499845435663e-13\\
647	5.35499845435663e-13\\
648	5.35499845435663e-13\\
649	5.35499845435663e-13\\
650	5.35499845435663e-13\\
651	5.35499845435663e-13\\
652	5.35499845435663e-13\\
653	5.35499845435663e-13\\
654	5.35499845435663e-13\\
655	5.35499845435663e-13\\
656	5.35499845435663e-13\\
657	5.35499845435663e-13\\
658	5.35499845435663e-13\\
659	5.35499845435663e-13\\
660	5.35499845435663e-13\\
661	5.35499845435663e-13\\
662	5.35499845435663e-13\\
663	5.35499845435663e-13\\
664	5.35499845435663e-13\\
665	5.35499845435663e-13\\
666	5.35499845435663e-13\\
667	5.35499845435663e-13\\
668	5.35499845435663e-13\\
669	5.35499845435663e-13\\
670	5.35499845435663e-13\\
671	5.35499845435663e-13\\
672	5.35499845435663e-13\\
673	5.35499845435663e-13\\
674	5.35499845435663e-13\\
675	5.35499845435663e-13\\
676	5.35499845435663e-13\\
677	5.35499845435663e-13\\
678	5.35499845435663e-13\\
679	5.35499845435663e-13\\
680	5.35499845435663e-13\\
681	5.35499845435663e-13\\
682	5.35499845435663e-13\\
683	5.35499845435663e-13\\
684	5.35499845435663e-13\\
685	5.35499845435663e-13\\
686	5.35499845435663e-13\\
687	5.35499845435663e-13\\
688	5.35499845435663e-13\\
689	5.35499845435663e-13\\
690	5.35499845435663e-13\\
691	5.35499845435663e-13\\
692	5.35499845435663e-13\\
693	5.35499845435663e-13\\
694	5.35499845435663e-13\\
695	5.35499845435663e-13\\
696	5.35499845435663e-13\\
697	5.35499845435663e-13\\
698	5.35499845435663e-13\\
699	5.35499845435663e-13\\
700	5.35499845435663e-13\\
701	5.35499845435663e-13\\
702	5.35499845435663e-13\\
703	5.35499845435663e-13\\
704	5.35499845435663e-13\\
705	5.35499845435663e-13\\
706	5.35499845435663e-13\\
707	5.35499845435663e-13\\
708	5.35499845435663e-13\\
709	5.35499845435663e-13\\
710	5.35499845435663e-13\\
711	5.35499845435663e-13\\
712	5.35499845435663e-13\\
713	5.35499845435663e-13\\
714	5.35499845435663e-13\\
715	5.35499845435663e-13\\
716	5.35499845435663e-13\\
717	5.35499845435663e-13\\
718	5.35499845435663e-13\\
719	5.35499845435663e-13\\
720	5.35499845435663e-13\\
721	5.35499845435663e-13\\
722	5.35499845435663e-13\\
723	5.35499845435663e-13\\
724	5.35499845435663e-13\\
725	5.35499845435663e-13\\
726	5.35499845435663e-13\\
727	5.35499845435663e-13\\
728	5.35499845435663e-13\\
729	5.35499845435663e-13\\
730	5.35499845435663e-13\\
731	5.35499845435663e-13\\
732	5.35499845435663e-13\\
733	5.35499845435663e-13\\
734	5.35499845435663e-13\\
735	5.35499845435663e-13\\
736	5.35499845435663e-13\\
737	5.35499845435663e-13\\
738	5.35499845435663e-13\\
739	5.35499845435663e-13\\
740	5.35499845435663e-13\\
741	5.35499845435663e-13\\
742	5.35499845435663e-13\\
743	5.35499845435663e-13\\
744	5.35499845435663e-13\\
745	5.35499845435663e-13\\
746	5.35499845435663e-13\\
747	5.35499845435663e-13\\
748	5.35499845435663e-13\\
749	5.35499845435663e-13\\
750	5.35499845435663e-13\\
751	5.35499845435663e-13\\
752	5.35499845435663e-13\\
753	5.35499845435663e-13\\
754	5.35499845435663e-13\\
755	5.35499845435663e-13\\
756	5.35499845435663e-13\\
757	5.35499845435663e-13\\
758	5.35499845435663e-13\\
759	5.35499845435663e-13\\
760	5.35499845435663e-13\\
761	5.35499845435663e-13\\
762	5.35499845435663e-13\\
763	5.35499845435663e-13\\
764	5.35499845435663e-13\\
765	5.35499845435663e-13\\
766	5.32330172859356e-13\\
767	5.09038066290989e-13\\
768	5.0728888490655e-13\\
769	4.88020333861398e-13\\
770	4.6420769431061e-13\\
771	4.54176470598869e-13\\
772	4.48091595834672e-13\\
773	4.34627980962358e-13\\
774	4.16853346800328e-13\\
775	3.99520023271186e-13\\
776	3.97562241001104e-13\\
777	3.80243096788e-13\\
778	3.61711479746553e-13\\
779	3.49716660235448e-13\\
780	3.34314782782774e-13\\
781	3.28519923478619e-13\\
782	3.27541079712363e-13\\
783	3.12940728510873e-13\\
784	2.99952835026842e-13\\
785	2.88174144050861e-13\\
786	2.69253056428314e-13\\
787	2.54703231535452e-13\\
788	2.49739887984499e-13\\
789	2.44491789890197e-13\\
790	2.26389919415349e-13\\
791	2.07813609266552e-13\\
792	1.87229938764973e-13\\
793	1.65361231535851e-13\\
794	1.52939694748201e-13\\
795	1.3659324661305e-13\\
796	1.20354256484709e-13\\
797	1.05962880015669e-13\\
798	9.54761139940972e-14\\
799	8.52130026559976e-14\\
800	7.56743551504872e-14\\
};
\addlegendentry{Nonlinear snapshot matrix}

\end{axis}
\end{tikzpicture}%

%% file: BC_Response.tikz
%
%
\definecolor{mycolor1}{rgb}{0.00000,0.44700,0.74100}%
\definecolor{mycolor2}{rgb}{0.85000,0.32500,0.09800}%
\definecolor{mycolor3}{rgb}{0.92900,0.69400,0.12500}%
\begin{tikzpicture}

\begin{axis}[%
width=0.951\fwidth,
height=\fheight,
at={(0\fwidth,0\fheight)},
scale only axis,
xmin=0,
xmax=10,
xlabel={time [s]},
ymin=-0.1,
ymax=0.8,
ylabel={$\by(t)$},
axis background/.style={fill=white},
legend style={legend cell align=left, align=right, draw=black, at = {(0.85,1.44), anchor = north east }}
]
\addplot [color=mycolor1, line width=2.0pt]
  table[row sep=crcr]{%
0	0\\
0.05	1.85784765291348e-222\\
0.1	1.61817633782714e-166\\
0.15	5.05134997400489e-135\\
0.2	3.50243719309119e-114\\
0.25	4.38096678973921e-99\\
0.3	1.79866758330622e-87\\
0.35	2.92383470205854e-78\\
0.4	8.20049857027535e-71\\
0.45	1.04888437937072e-64\\
0.5	1.21608744326719e-59\\
0.55	2.13234355636596e-55\\
0.6	8.38639490899117e-52\\
0.65	1.00789636543387e-48\\
0.7	4.72117877495551e-46\\
0.75	1.04255387108948e-43\\
0.8	1.25776767912207e-41\\
0.85	9.2876447567245e-40\\
0.9	4.5816186324037e-38\\
0.95	1.61562686404812e-36\\
1	4.29346174172241e-35\\
1.05	8.96400428714724e-34\\
1.1	1.52002322101245e-32\\
1.15	2.15030075984187e-31\\
1.2	2.59396166815166e-30\\
1.25	2.71708258871397e-29\\
1.3	2.50887044738655e-28\\
1.35	2.06833401938251e-27\\
1.4	1.53898639241164e-26\\
1.45	1.04317210233673e-25\\
1.5	6.49340077105997e-25\\
1.55	3.73783244699274e-24\\
1.6	2.00197828100791e-23\\
1.65	1.00308145554392e-22\\
1.7	4.7242151320415e-22\\
1.75	2.10035559110472e-21\\
1.8	8.8488174852843e-21\\
1.85	3.54486707390355e-20\\
1.9	1.35452196611907e-19\\
1.95	4.95070949344873e-19\\
2	1.7352369206792e-18\\
2.05	5.846270574934e-18\\
2.1	1.89741520651489e-17\\
2.15	5.94386926877762e-17\\
2.2	1.80050205306952e-16\\
2.25	5.28288247575319e-16\\
2.3	1.50378052493189e-15\\
2.35	4.15882557029485e-15\\
2.4	1.11898564690586e-14\\
2.45	2.93294062870281e-14\\
2.5	7.49774641067104e-14\\
2.55	1.87154595541272e-13\\
2.6	4.56645158253677e-13\\
2.65	1.09020368793965e-12\\
2.7	2.54921544733661e-12\\
2.75	5.84352899878603e-12\\
2.8	1.31430198327286e-11\\
2.85	2.90289588809135e-11\\
2.9	6.30135668429727e-11\\
2.95	1.345361666361e-10\\
3	2.82728466216643e-10\\
3.05	5.85243165542497e-10\\
3.1	1.19409360147865e-09\\
3.15	2.40304852346871e-09\\
3.2	4.77291227972731e-09\\
3.25	9.3618670637676e-09\\
3.3	1.81445534106876e-08\\
3.35	3.47668096438224e-08\\
3.4	6.58913731267153e-08\\
3.45	1.23573336098021e-07\\
3.5	2.29410541489388e-07\\
3.55	4.21717547081534e-07\\
3.6	7.67783286769935e-07\\
3.65	1.38455051384767e-06\\
3.7	2.47301095923634e-06\\
3.75	4.3744991730569e-06\\
3.8	7.66114962621387e-06\\
3.85	1.3278228750993e-05\\
3.9	2.2762795013453e-05\\
3.95	3.85706385615017e-05\\
4	6.45504589085095e-05\\
4.05	0.000106608614998577\\
4.1	0.000173606506925266\\
4.15	0.000278522150407578\\
4.2	0.000439884396510702\\
4.25	0.000683450513155301\\
4.3	0.00104404636190557\\
4.35	0.00156742845554868\\
4.4	0.00231196914697449\\
4.45	0.00334992443882539\\
4.5	0.00476803324317896\\
4.55	0.00666722735840318\\
4.6	0.00916130318425063\\
4.65	0.01237450778118\\
4.7	0.0164381028659826\\
4.75	0.0214860680796281\\
4.8	0.0276501721427118\\
4.85	0.035054670800844\\
4.9	0.0438108884719918\\
4.95	0.0540119186699918\\
5	0.0657276509449832\\
5.05	0.0790003103614523\\
5.1	0.0938406841818226\\
5.15	0.110225207366206\\
5.2	0.128094076689372\\
5.25	0.14735055328938\\
5.3	0.167861586828717\\
5.35	0.189459845752896\\
5.4	0.211947166974586\\
5.45	0.235099349374076\\
5.5	0.258672118045179\\
5.55	0.28240799239098\\
5.6	0.306043713793486\\
5.65	0.32931783888345\\
5.7	0.351978089848266\\
5.75	0.37378807613002\\
5.8	0.394533059235207\\
5.85	0.414024516485736\\
5.9	0.43210335946466\\
5.95	0.448641766460388\\
6	0.4635436840597\\
6.05	0.476744132371608\\
6.1	0.488207505417945\\
6.15	0.49792509140214\\
6.2	0.505912047490112\\
6.25	0.512204054607479\\
6.3	0.516853853962664\\
6.35	0.519927834047948\\
6.4	0.521502799617378\\
6.45	0.52166301668953\\
6.5	0.520497592998549\\
6.55	0.518098223379544\\
6.6	0.514557305358121\\
6.65	0.509966411809175\\
6.7	0.504415094655041\\
6.75	0.497989985451453\\
6.8	0.490774154533797\\
6.85	0.482846689279034\\
6.9	0.474282453132933\\
6.95	0.465151989668781\\
7	0.455521539452058\\
7.05	0.445453141494698\\
7.1	0.435004795188101\\
7.15	0.424230662554921\\
7.2	0.413181294385096\\
7.25	0.401903867164729\\
7.3	0.390442420645377\\
7.35	0.378838088415269\\
7.4	0.367129315966427\\
7.45	0.355352062514112\\
7.5	0.343539984284557\\
7.55	0.331724598107655\\
7.6	0.319935425058525\\
7.65	0.308200114607405\\
7.7	0.296544550240607\\
7.75	0.28499293790745\\
7.8	0.273567878904835\\
7.85	0.262290428986969\\
7.9	0.251180145590808\\
7.95	0.240255125105445\\
8	0.229532032105886\\
8.05	0.219026122428307\\
8.1	0.208751261907434\\
8.15	0.198719942472296\\
8.2	0.188943297200577\\
8.25	0.179431115784883\\
8.3	0.170191861712554\\
8.35	0.161232692298875\\
8.4	0.152559482520812\\
8.45	0.14417685342215\\
8.5	0.136088205643249\\
8.55	0.128295758441404\\
8.6	0.120800594350918\\
8.65	0.113602709435887\\
8.7	0.106701068895459\\
8.75	0.100093667592375\\
8.8	0.0937775949264583\\
8.85	0.0877491033273487\\
8.9	0.0820036795391659\\
8.95	0.0765361177857444\\
9	0.071340593862265\\
9.05	0.0664107391838862\\
9.1	0.0617397138465002\\
9.15	0.0573202778063038\\
9.2	0.0531448593659302\\
9.25	0.0492056202565958\\
9.3	0.0454945167262792\\
9.35	0.0420033561778911\\
9.4	0.0387238490396786\\
9.45	0.0356476556923976\\
9.5	0.0327664284095219\\
9.55	0.0300718483846633\\
9.6	0.0275556580460534\\
9.65	0.0252096889288728\\
9.7	0.0230258854698611\\
9.75	0.0209963251220656\\
9.8	0.0191132352354343\\
9.85	0.0173690071451716\\
9.9	0.0157562079142017\\
9.95	0.0142675901453984\\
10	0.0128961002427203\\
};
\addlegendentry{Original}

\addplot [color=mycolor1, line width=2.0pt,forget plot]
  table[row sep=crcr]{%
0	0\\
0.05	1.50742390148927e-221\\
0.1	9.71706091029466e-165\\
0.15	2.46519936905339e-132\\
0.2	1.19192572655354e-110\\
0.25	8.96967198710715e-95\\
0.3	2.07340892199423e-82\\
0.35	1.77797596979111e-72\\
0.4	2.420716507198e-64\\
0.45	1.35070710913936e-57\\
0.5	5.978811295468e-52\\
0.55	3.41657925037673e-47\\
0.6	3.66491782699657e-43\\
0.65	9.95194787896262e-40\\
0.7	8.7423187687339e-37\\
0.75	3.04581194581475e-34\\
0.8	4.98426394218486e-32\\
0.85	4.40172124350387e-30\\
0.9	2.34693725765492e-28\\
0.95	8.26208524449267e-27\\
1	2.06114876290139e-25\\
1.05	3.85305377828109e-24\\
1.1	5.64117715230847e-23\\
1.15	6.70039270222344e-22\\
1.2	6.64197326795669e-21\\
1.25	5.62276411590631e-20\\
1.3	4.142425958892e-19\\
1.35	2.69776970595563e-18\\
1.4	1.57356639677402e-17\\
1.45	8.31182871787721e-17\\
1.5	4.01355481776857e-16\\
1.55	1.78605397709362e-15\\
1.6	7.37616077188451e-15\\
1.65	2.84434480996992e-14\\
1.7	1.02960141607664e-13\\
1.75	3.51508305236547e-13\\
1.8	1.13656458834115e-12\\
1.85	3.49349135049218e-12\\
1.9	1.02418401979365e-11\\
1.95	2.87242749746994e-11\\
2	7.72763604154103e-11\\
2.05	1.99910040561146e-10\\
2.1	4.98403041653787e-10\\
2.15	1.19996408213264e-09\\
2.2	2.79515699834758e-09\\
2.25	6.3101191508568e-09\\
2.3	1.38275976880061e-08\\
2.35	2.9455468796693e-08\\
2.4	6.10774180273788e-08\\
2.45	1.23433919762255e-07\\
2.5	2.43407181428377e-07\\
2.55	4.68866448912608e-07\\
2.6	8.83126584912315e-07\\
2.65	1.62805112323508e-06\\
2.7	2.94016901955165e-06\\
2.75	5.20595663193135e-06\\
2.8	9.04474864807136e-06\\
2.85	1.54306620761982e-05\\
2.9	2.58685062876799e-05\\
2.95	4.26429377023866e-05\\
3	6.91650798180273e-05\\
3.05	0.000110446378692428\\
3.1	0.000173735417742766\\
3.15	0.000269359464030714\\
3.2	0.000411818178842042\\
3.25	0.000621181476951676\\
3.3	0.000924845902497388\\
3.35	0.00135970256578687\\
3.4	0.00197476244058431\\
3.45	0.00283426852261893\\
3.5	0.00402129469439727\\
3.55	0.00564178249416369\\
3.6	0.00782889253066713\\
3.65	0.0107474398310073\\
3.7	0.0145980363962476\\
3.75	0.0196203794575556\\
3.8	0.0260949111573959\\
3.85	0.0343418631513125\\
3.9	0.0447165394470941\\
3.95	0.0575996563609656\\
4	0.0733817339244261\\
4.05	0.0924409889689352\\
4.1	0.115114942759608\\
4.15	0.141666982555488\\
4.2	0.17225028662878\\
4.25	0.20687264824076\\
4.3	0.245366580591985\\
4.35	0.287369396751333\\
4.4	0.332317507389939\\
4.45	0.379457840249176\\
4.5	0.42787713308838\\
4.55	0.476547219597593\\
4.6	0.524381869353432\\
4.65	0.570298879134886\\
4.7	0.613280415018522\\
4.75	0.652425219971203\\
4.8	0.686988007383295\\
4.85	0.716403664946711\\
4.9	0.740296220959584\\
4.95	0.758474406946452\\
5	0.770916824657214\\
5.05	0.7777501620162\\
5.1	0.779223733248216\\
5.15	0.775683060390397\\
5.2	0.767544484529427\\
5.25	0.755272068810748\\
5.3	0.739357440193424\\
5.35	0.720302760907069\\
5.4	0.69860672631247\\
5.45	0.674753328336028\\
5.5	0.649203066039529\\
5.55	0.62238628964513\\
5.6	0.59469839957549\\
5.65	0.566496664793921\\
5.7	0.538098460831317\\
5.75	0.509780751209798\\
5.8	0.481780646257421\\
5.85	0.454296873872649\\
5.9	0.42749199270079\\
5.95	0.401495174438465\\
6	0.37640538246389\\
6.05	0.352294780939106\\
6.1	0.329212222506347\\
6.15	0.307186682808769\\
6.2	0.286230534462277\\
6.25	0.266342579532463\\
6.3	0.247510785828472\\
6.35	0.2297146965884\\
6.4	0.212927504022132\\
6.45	0.197117794025587\\
6.5	0.182250981938999\\
6.55	0.16829046769895\\
6.6	0.155198543621936\\
6.65	0.142937089970744\\
6.7	0.131468093088423\\
6.75	0.120754018880238\\
6.8	0.110758071365763\\
6.85	0.101444362379747\\
6.9	0.0927780146432771\\
6.95	0.0847252166307355\\
7	0.0772532440995928\\
7.05	0.0703304599408632\\
7.1	0.0639263012084962\\
7.15	0.0580112598119113\\
7.2	0.052556861391358\\
7.25	0.0475356453130498\\
7.3	0.0429211474880156\\
7.35	0.0386878867792513\\
7.4	0.0348113550750683\\
7.45	0.0312680106390145\\
7.5	0.0280352740473002\\
7.55	0.0250915258513307\\
7.6	0.0224161050535307\\
7.65	0.0199893074953599\\
7.7	0.0177923833314811\\
7.75	0.0158075328735649\\
7.8	0.014017900217906\\
7.85	0.0124075642116913\\
7.9	0.010961526453962\\
7.95	0.00966569616193427\\
8	0.00850687185649011\\
8.05	0.00747271992992165\\
8.1	0.0065517502513995\\
8.15	0.0057332890417029\\
8.2	0.00500744930765245\\
8.25	0.00436509916885358\\
8.3	0.00379782843749573\\
8.35	0.00329791382676078\\
8.4	0.00285828316556101\\
8.45	0.00247247898974762\\
8.5	0.00213462186478701\\
8.55	0.00183937377307258\\
8.6	0.00158190187132932\\
8.65	0.00135784289308038\\
8.7	0.00116326843885765\\
8.75	0.000994651362821847\\
8.8	0.000848833430033017\\
8.85	0.00072299438562517\\
8.9	0.000614622545515686\\
8.95	0.000521486987766547\\
9	0.00044161139577946\\
9.05	0.000373249579557945\\
9.1	0.000314862678630095\\
9.15	0.000265098030064808\\
9.2	0.000222769668074457\\
9.25	0.000186840407622137\\
9.3	0.00015640545254576\\
9.35	0.000130677459351924\\
9.4	0.00010897298111635\\
9.45	9.07002112338791e-05\\
9.5	7.53479431176852e-05\\
9.55	6.24756615857147e-05\\
9.6	5.17046806721091e-05\\
9.65	4.27102449943823e-05\\
9.7	3.5214512767582e-05\\
9.75	2.89803431049025e-05\\
9.8	2.38058122399012e-05\\
9.85	1.95193886864948e-05\\
9.9	1.59757011674219e-05\\
9.95	1.30518374304211e-05\\
10	1.06441170333341e-05\\
};

\addplot [color=mycolor2, dashdotted, line width=2.0pt]
  table[row sep=crcr]{%
0	0\\
0.05	-3.77901106705675e-06\\
0.1	1.18932068567279e-05\\
0.15	-5.88610941254414e-07\\
0.2	-1.32676228220114e-05\\
0.25	2.3322067519385e-07\\
0.3	1.38950365993629e-05\\
0.35	4.54854136389348e-06\\
0.4	-1.19791408955083e-05\\
0.45	-1.1318864355408e-05\\
0.5	4.64787428491803e-06\\
0.55	1.47007169459575e-05\\
0.6	7.06523249124506e-06\\
0.65	-8.57997502807783e-06\\
0.7	-1.48078793310371e-05\\
0.75	-5.6369387351851e-06\\
0.8	9.05339739632727e-06\\
0.85	1.49737241242007e-05\\
0.9	7.07396436536504e-06\\
0.95	-7.06863315278501e-06\\
1	-1.51223386646121e-05\\
1.05	-1.06742922093014e-05\\
1.1	2.26486507862735e-06\\
1.15	1.33920548583992e-05\\
1.2	1.4430990147165e-05\\
1.25	4.96216041268509e-06\\
1.3	-7.98450893760833e-06\\
1.35	-1.5403945464887e-05\\
1.4	-1.2497368421764e-05\\
1.45	-1.4458642016749e-06\\
1.5	1.05310187495797e-05\\
1.55	1.60249270916977e-05\\
1.6	1.19090388936918e-05\\
1.65	8.49178503525839e-07\\
1.7	-1.06515407694656e-05\\
1.75	-1.61657529390195e-05\\
1.8	-1.28427581289202e-05\\
1.85	-2.67451298136094e-06\\
1.9	8.88014203567522e-06\\
1.95	1.59142932469867e-05\\
2	1.50447266576388e-05\\
2.05	6.88832588986286e-06\\
2.1	-4.50231985359384e-06\\
2.15	-1.37607546254517e-05\\
2.2	-1.67455019602968e-05\\
2.25	-1.23091878301678e-05\\
2.3	-2.59885525467238e-06\\
2.35	8.08809151814836e-06\\
2.4	1.52562068497214e-05\\
2.45	1.60768657789245e-05\\
2.5	1.04150870462961e-05\\
2.55	6.96099219111437e-07\\
2.6	-9.21365814088484e-06\\
2.65	-1.55764882995546e-05\\
2.7	-1.61644660176528e-05\\
2.75	-1.09614973425688e-05\\
2.8	-2.00584673137334e-06\\
2.85	7.47642735775125e-06\\
2.9	1.42663951302225e-05\\
2.95	1.62304899299725e-05\\
3	1.29344786581321e-05\\
3.05	5.64564488662357e-06\\
3.1	-3.20530226707774e-06\\
3.15	-1.08682428189102e-05\\
3.2	-1.51259271570975e-05\\
3.25	-1.4896407605667e-05\\
3.3	-1.04398260558752e-05\\
3.35	-3.16156756419056e-06\\
3.4	4.87911044742429e-06\\
3.45	1.15804274592117e-05\\
3.5	1.53505189879965e-05\\
3.55	1.54684602485098e-05\\
3.6	1.21956860833342e-05\\
3.65	6.65609992550885e-06\\
3.7	5.61267818539044e-07\\
3.75	-4.11366325032436e-06\\
3.8	-5.39694162657061e-06\\
3.85	-1.44292413371487e-06\\
3.9	9.54950348232496e-06\\
3.95	2.96585008741512e-05\\
4	6.18208072687213e-05\\
4.05	0.000110685608346753\\
4.1	0.000183748028554631\\
4.15	0.000292758973701523\\
4.2	0.000455385166566474\\
4.25	0.000697055486809172\\
4.3	0.00105288065535996\\
4.35	0.00156948225054499\\
4.4	0.00230652697623596\\
4.45	0.00333774581333141\\
4.5	0.00475123355496719\\
4.55	0.00664887298505406\\
4.6	0.00914480179901889\\
4.65	0.0123629249927139\\
4.7	0.016433554552012\\
4.75	0.0214893187566465\\
4.8	0.0276605192573321\\
4.85	0.035070127428533\\
4.9	0.0438286106150783\\
4.95	0.0540287741884387\\
5	0.065740805159369\\
5.05	0.0790077108212799\\
5.1	0.0938413591189517\\
5.15	0.110219339134967\\
5.2	0.128082861098699\\
5.25	0.147335897162341\\
5.3	0.167845721457179\\
5.35	0.189444939848312\\
5.4	0.211935010533992\\
5.45	0.235091154739183\\
5.5	0.258668453639761\\
5.55	0.282408835699331\\
5.6	0.30604858910039\\
5.65	0.329325995355191\\
5.7	0.351988676869412\\
5.75	0.373800282978221\\
5.8	0.394546201190928\\
5.85	0.414038064996152\\
5.9	0.432116926453822\\
5.95	0.448655060239197\\
6	0.463556456133593\\
6.05	0.476756131612981\\
6.1	0.488218450413326\\
6.15	0.497934665043957\\
6.2	0.505919912015162\\
6.25	0.512209881067084\\
6.3	0.516857358118813\\
6.35	0.519928810732786\\
6.4	0.521501149168818\\
6.45	0.521658759497112\\
6.5	0.520490870795201\\
6.55	0.518089288228736\\
6.6	0.51454649893086\\
6.65	0.509954138480411\\
6.7	0.504401792126491\\
6.75	0.497976096278388\\
6.8	0.490760101177978\\
6.85	0.482832854419169\\
6.9	0.474269166039567\\
6.95	0.465139518653138\\
7	0.455510089801238\\
7.05	0.445442857899242\\
7.1	0.434995767462593\\
7.15	0.424222933451466\\
7.2	0.413174868403056\\
7.25	0.401898719448813\\
7.3	0.390438505282427\\
7.35	0.378835345679535\\
7.4	0.36712767827301\\
7.45	0.355351459018932\\
7.5	0.343540344181776\\
7.55	0.331725852774411\\
7.6	0.319937509253983\\
7.65	0.308202966938732\\
7.7	0.29654811311106\\
7.75	0.284997157135896\\
7.8	0.273572703178027\\
7.85	0.26229580927531\\
7.9	0.25118603461818\\
7.95	0.240261476932605\\
8	0.229538801861073\\
8.05	0.219033266199144\\
8.1	0.20875873677581\\
8.15	0.198727706676355\\
8.2	0.188951310390459\\
8.25	0.179439339336993\\
8.3	0.170200259066496\\
8.35	0.16124122927863\\
8.4	0.15256812761225\\
8.45	0.144185577977181\\
8.5	0.136096983997379\\
8.55	0.12830456793232\\
8.6	0.120809415236274\\
8.65	0.113611524714442\\
8.7	0.106709864038663\\
8.75	0.100102430205737\\
8.8	0.0937863143580145\\
8.85	0.0877577702476221\\
8.9	0.0820122855139046\\
8.95	0.0765446548643337\\
9	0.0713490542015089\\
9.05	0.066419114726732\\
9.1	0.0617479960706683\\
9.15	0.0573284575530705\\
9.2	0.0531529267534691\\
9.25	0.0492135646773949\\
9.3	0.0455023269238424\\
9.35	0.0420110203936473\\
9.4	0.0387313552171929\\
9.45	0.0356549917223326\\
9.5	0.0327735823961501\\
9.55	0.0300788089207915\\
9.6	0.0275624144746414\\
9.65	0.0252162315823595\\
9.7	0.0230322058715827\\
9.75	0.0210024161460902\\
9.8	0.0191190912173505\\
9.85	0.0173746239464401\\
9.9	0.0157615829407121\\
9.95	0.0142727223251476\\
10	0.0129009899694357\\
};
\addlegendentry{Intrusive-POD}

\addplot [color=mycolor2, dashdotted, line width=2.0pt,forget plot]
  table[row sep=crcr]{%
0	0\\
0.05	1.31058676643264e-05\\
0.1	-6.00794426646507e-06\\
0.15	-8.08888207951717e-06\\
0.2	7.4750652473873e-06\\
0.25	1.04484136456034e-05\\
0.3	-5.49272188868399e-06\\
0.35	-1.34316748388332e-05\\
0.4	-6.66636458418869e-07\\
0.45	1.35778928846226e-05\\
0.5	9.44935398270595e-06\\
0.55	-7.02511304799159e-06\\
0.6	-1.49192556797116e-05\\
0.65	-5.28304046960134e-06\\
0.7	1.02404159920397e-05\\
0.75	1.46279634289392e-05\\
0.8	3.84559987595801e-06\\
0.85	-1.06488959433509e-05\\
0.9	-1.47315480208976e-05\\
0.95	-5.06754544900677e-06\\
1	9.03991426304554e-06\\
1.05	1.50887699657575e-05\\
1.1	8.29471015996926e-06\\
1.15	-5.22537538066996e-06\\
1.2	-1.43852514933443e-05\\
1.25	-1.22261383777195e-05\\
1.3	-8.27376078257671e-07\\
1.35	1.10252597542691e-05\\
1.4	1.47579449506048e-05\\
1.45	8.03633039979507e-06\\
1.5	-4.13599929870461e-06\\
1.55	-1.33731684303756e-05\\
1.6	-1.36947362615807e-05\\
1.65	-5.19589600370805e-06\\
1.7	6.46832603458117e-06\\
1.75	1.39861729845002e-05\\
1.8	1.29357332515521e-05\\
1.85	4.20638064145858e-06\\
1.9	-6.88423011612381e-06\\
1.95	-1.39246835048717e-05\\
2	-1.30755787938997e-05\\
2.05	-5.01855432658669e-06\\
2.1	5.69361556917159e-06\\
2.15	1.32906567535107e-05\\
2.2	1.38712925503767e-05\\
2.25	7.30229736000372e-06\\
2.3	-2.92038733798907e-06\\
2.35	-1.16027849123857e-05\\
2.4	-1.44828492904343e-05\\
2.45	-1.02373025246215e-05\\
2.5	-9.61471369862481e-07\\
2.55	8.99065424340319e-06\\
2.6	1.51596237466442e-05\\
2.65	1.5089713214426e-05\\
2.7	9.46616707018392e-06\\
2.75	1.86087507820443e-06\\
2.8	-2.68501020326641e-06\\
2.85	4.8696119516621e-07\\
2.9	1.42520145951639e-05\\
2.95	3.94214486343037e-05\\
3	7.57804983693863e-05\\
3.05	0.000124134936685215\\
3.1	0.000188718706834755\\
3.15	0.000279272405522012\\
3.2	0.00041235454513626\\
3.25	0.000611878605398802\\
3.3	0.000909301327543339\\
3.35	0.00134415322330769\\
3.4	0.00196560389366937\\
3.45	0.00283549748276868\\
3.5	0.00403287500322373\\
3.55	0.00565956741667776\\
3.6	0.00784613402383473\\
3.65	0.0107573113925514\\
3.7	0.0145962175425684\\
3.75	0.0196067343509299\\
3.8	0.026073640355752\\
3.85	0.0343200798422003\\
3.9	0.044701802142411\\
3.95	0.0575973620730186\\
4	0.0733933044915863\\
4.05	0.0924634654317011\\
4.1	0.11514206761289\\
4.15	0.141691308421961\\
4.2	0.172265511961978\\
4.25	0.206875368377041\\
4.3	0.245356939187261\\
4.35	0.287350581992047\\
4.4	0.33229445394437\\
4.45	0.379435706487212\\
4.5	0.427860073264289\\
4.55	0.476537736138642\\
4.6	0.524380752399349\\
4.65	0.570305545305373\\
4.7	0.613293399147354\\
4.75	0.652442628367533\\
4.8	0.687007843475716\\
4.85	0.716424026033891\\
4.9	0.740315408617921\\
4.95	0.75849099486757\\
5	0.77092970842118\\
5.05	0.777758594425586\\
5.1	0.779227338518899\\
5.15	0.7756818208028\\
5.2	0.767538699404154\\
5.25	0.755262292172772\\
5.3	0.739344406887444\\
5.35	0.720287310873182\\
5.4	0.698589734446668\\
5.45	0.674735644562299\\
5.5	0.649185468009433\\
5.55	0.62236944910411\\
5.6	0.594682862637061\\
5.65	0.566482845629778\\
5.7	0.538086647368512\\
5.75	0.509771120365413\\
5.8	0.481773285284534\\
5.85	0.454291803905556\\
5.9	0.42748919076847\\
5.95	0.401494591046878\\
6	0.376406953128477\\
6.05	0.352298431474731\\
6.1	0.329217869111532\\
6.15	0.307194228781362\\
6.2	0.286239865658079\\
6.25	0.266353560485444\\
6.3	0.247523257927385\\
6.35	0.229728478999538\\
6.4	0.212942397378029\\
6.45	0.197133586359602\\
6.5	0.182267456014798\\
6.55	0.168307408773673\\
6.6	0.155215746754757\\
6.65	0.142954366181429\\
6.7	0.131485273935565\\
6.75	0.120770959286512\\
6.8	0.110774650748877\\
6.85	0.10146048430984\\
6.9	0.0927936053495225\\
6.95	0.0847402227145535\\
7	0.0772676297968957\\
7.05	0.0703442042287834\\
7.1	0.0639393949868351\\
7.15	0.0580237033179854\\
7.2	0.0525686619412619\\
7.25	0.0475468154100726\\
7.3	0.0429317032959088\\
7.35	0.0386978469302314\\
7.4	0.0348207397698727\\
7.45	0.031276840988612\\
7.5	0.0280435716037414\\
7.55	0.0250993122857484\\
7.6	0.0224234019412609\\
7.65	0.0199961361775404\\
7.7	0.0177987648292763\\
7.75	0.0158134878370919\\
7.8	0.0140234488967696\\
7.85	0.0124127264382755\\
7.9	0.0109663216332033\\
7.95	0.00967014326353951\\
8	0.00851098940727251\\
8.05	0.00747652600514231\\
8.1	0.00655526246484388\\
8.15	0.00573652453498701\\
8.2	0.0050104247393722\\
8.25	0.00436783070478072\\
8.3	0.0038003317432808\\
8.35	0.00330020406384888\\
8.4	0.00286037499096815\\
8.45	0.00247438656004813\\
8.5	0.00213635884398754\\
8.55	0.00184095334297527\\
8.6	0.00158333674257687\\
8.65	0.00135914531468754\\
8.7	0.0011644502032643\\
8.75	0.000995723803156206\\
8.8	0.000849807406561493\\
8.85	0.000723880258817154\\
8.9	0.000615430133368328\\
8.95	0.000522225506298947\\
9	0.000442289382798454\\
9.05	0.000373874803295545\\
9.1	0.000315442034090519\\
9.15	0.000265637427702324\\
9.2	0.000223273920931573\\
9.25	0.000187313123996357\\
9.3	0.000156848942194542\\
9.35	0.000131092661817076\\
9.4	0.000109359424718069\\
9.45	9.10560102794056e-05\\
9.5	7.5669840529931e-05\\
9.55	6.27591217776694e-05\\
9.6	5.19440360871254e-05\\
9.65	4.28988967751354e-05\\
9.7	3.53451841195157e-05\\
9.75	2.90453804453811e-05\\
9.8	2.37975273988402e-05\\
9.85	1.94304325378308e-05\\
9.9	1.57994570493327e-05\\
9.95	1.27828213898717e-05\\
10	1.02783708065893e-05\\
};

\addplot [color=mycolor3, dotted, line width=2.0pt]
  table[row sep=crcr]{%
0	0\\
0.05	-1.53862406026049e-05\\
0.1	6.5164906666106e-06\\
0.15	-9.15990151024731e-06\\
0.2	-1.92550271002236e-05\\
0.25	1.47804030496668e-06\\
0.3	1.73584319821163e-05\\
0.35	4.63310571081982e-06\\
0.4	-1.40221762888726e-05\\
0.45	-1.10580518624135e-05\\
0.5	8.22460002823893e-06\\
0.55	1.83863502106324e-05\\
0.6	7.67661391357312e-06\\
0.65	-1.08335823753296e-05\\
0.7	-1.70749542371643e-05\\
0.75	-5.58692661129383e-06\\
0.8	1.11604698543954e-05\\
0.85	1.69313317159689e-05\\
0.9	6.95534848817572e-06\\
0.95	-9.28997759623869e-06\\
1	-1.77568219712023e-05\\
1.05	-1.18262771884889e-05\\
1.1	3.24554153349535e-06\\
1.15	1.5576094449226e-05\\
1.2	1.61444585111303e-05\\
1.25	5.02028770866073e-06\\
1.3	-9.5386652562115e-06\\
1.35	-1.74603496083193e-05\\
1.4	-1.36970366710851e-05\\
1.45	-1.04958986445054e-06\\
1.5	1.22542185744013e-05\\
1.55	1.80474517144349e-05\\
1.6	1.30985357641296e-05\\
1.65	6.05722015078139e-07\\
1.7	-1.21187300524243e-05\\
1.75	-1.80200943236387e-05\\
1.8	-1.41058710878219e-05\\
1.85	-2.73067963419036e-06\\
1.9	1.00121090974425e-05\\
1.95	1.76435394186017e-05\\
2	1.65337941538233e-05\\
2.05	7.45864823996824e-06\\
2.1	-5.07783778776687e-06\\
2.15	-1.51935657797059e-05\\
2.2	-1.83958314223767e-05\\
2.25	-1.34822215609767e-05\\
2.3	-2.83462129225853e-06\\
2.35	8.8476763681255e-06\\
2.4	1.66749034338002e-05\\
2.45	1.7583094993241e-05\\
2.5	1.14345985136461e-05\\
2.55	8.63600554250066e-07\\
2.6	-9.93823540631665e-06\\
2.65	-1.69174134991295e-05\\
2.7	-1.76477296456578e-05\\
2.75	-1.20883122963098e-05\\
2.8	-2.41832147257822e-06\\
2.85	7.89018390234661e-06\\
2.9	1.53541788658672e-05\\
2.95	1.76408003431186e-05\\
3	1.42354648416802e-05\\
3.05	6.45717210575643e-06\\
3.1	-3.10592264012655e-06\\
3.15	-1.14937328987856e-05\\
3.2	-1.6288450417457e-05\\
3.25	-1.62698521758612e-05\\
3.3	-1.16532068996638e-05\\
3.35	-3.89657158522998e-06\\
3.4	4.81113004719395e-06\\
3.45	1.21987558330252e-05\\
3.5	1.65086477460832e-05\\
3.55	1.68961990896574e-05\\
3.6	1.35668905249997e-05\\
3.65	7.66477805205519e-06\\
3.7	9.88765806746513e-07\\
3.75	-4.35411198951292e-06\\
3.8	-6.24603272460412e-06\\
3.85	-2.71232675255418e-06\\
3.9	8.134315448862e-06\\
3.95	2.83997068740496e-05\\
4	6.09863840885794e-05\\
4.05	0.000110456401411196\\
4.1	0.000184183622095637\\
4.15	0.000293788231890179\\
4.2	0.000456822478459514\\
4.25	0.000698638942327978\\
4.3	0.00105432444171334\\
4.35	0.00157053316092282\\
4.4	0.00230701446045824\\
4.45	0.00333761650981507\\
4.5	0.00475056438636452\\
4.55	0.00664786102873677\\
4.6	0.00914373343712263\\
4.65	0.0123621304815073\\
4.7	0.0164333564694897\\
4.75	0.0214899834265125\\
4.8	0.0276622191959775\\
4.85	0.0350729204411339\\
4.9	0.0438324379383113\\
4.95	0.0540334767025673\\
5	0.0657461537894134\\
5.05	0.0790134452743516\\
5.1	0.0938472282959358\\
5.15	0.110225136910274\\
5.2	0.128088452482865\\
5.25	0.147341231818139\\
5.3	0.167850833741541\\
5.35	0.189449936624778\\
5.4	0.211940049640872\\
5.45	0.235096418013348\\
5.5	0.258674118712927\\
5.55	0.282415050453607\\
5.6	0.30605545182968\\
5.65	0.329333542528758\\
5.7	0.351996879147066\\
5.75	0.373809048980007\\
5.8	0.394555387621233\\
5.85	0.414047491157935\\
5.9	0.432126390964817\\
5.95	0.448664357868306\\
6	0.463565393013031\\
6.05	0.476764537567691\\
6.1	0.488226187699163\\
6.15	0.497941633292496\\
6.2	0.505926049616462\\
6.25	0.512215163528535\\
6.3	0.516861794120207\\
6.35	0.51993243665998\\
6.4	0.521504022874516\\
6.45	0.521660953935388\\
6.5	0.520492468034989\\
6.55	0.518090374196028\\
6.6	0.514547159083251\\
6.65	0.50995445448367\\
6.7	0.504401839510169\\
6.75	0.497975942969391\\
6.8	0.490759806772416\\
6.85	0.482832470036435\\
6.9	0.474268734600855\\
6.95	0.465139075444433\\
7	0.455509663201012\\
7.05	0.445442470179174\\
7.1	0.434995435594139\\
7.15	0.424222669874094\\
7.2	0.413174681730506\\
7.25	0.40189861510169\\
7.3	0.390438486047536\\
7.35	0.378835412200316\\
7.4	0.367127829478877\\
7.45	0.355351692501056\\
7.5	0.343540656521943\\
7.55	0.331726239832033\\
7.6	0.319937966412291\\
7.65	0.308203489314784\\
7.7	0.296548695726847\\
7.75	0.284997795053505\\
7.8	0.273573391599851\\
7.85	0.262296543609878\\
7.9	0.25118681051676\\
7.95	0.240262290300387\\
8	0.229539648848849\\
8.05	0.219034143180678\\
8.1	0.20875964031841\\
8.15	0.198728633509469\\
8.2	0.18895225737784\\
8.25	0.179440303455814\\
8.3	0.170201237396249\\
8.35	0.161242219000095\\
8.4	0.15256912601644\\
8.45	0.144186582481952\\
8.5	0.136097992169526\\
8.55	0.128305577511032\\
8.6	0.120810424155521\\
8.65	0.113612531119934\\
8.7	0.106710866297419\\
8.75	0.100103426906073\\
8.8	0.0937873042990904\\
8.85	0.0877587524185348\\
8.9	0.0820132590635487\\
8.95	0.0765456190638087\\
9	0.0713500084017588\\
9.05	0.0664200583145162\\
9.1	0.0617489284257378\\
9.15	0.057329378009904\\
9.2	0.0531538345702262\\
9.25	0.0492144590139982\\
9.3	0.0455032068310021\\
9.35	0.0420118848129548\\
9.4	0.0387322029926847\\
9.45	0.0356558216224641\\
9.5	0.0327743931440516\\
9.55	0.0300795992316287\\
9.6	0.0275631830968528\\
9.65	0.0252169773403407\\
9.7	0.0230329277074379\\
9.75	0.0210031131577225\\
9.8	0.0191197626909438\\
9.85	0.0173752693816477\\
9.9	0.0157622020674558\\
9.95	0.0142733151114722\\
10	0.0129015566205264\\
};
\addlegendentry{Learned ROM}

\addplot [color=mycolor3, dotted, line width=2.0pt]
  table[row sep=crcr]{%
0	0\\
0.05	-2.44064281875666e-07\\
0.1	-1.69932941348951e-05\\
0.15	-1.38036814194895e-05\\
0.2	1.37745069120043e-05\\
0.25	1.65917879830626e-05\\
0.3	-8.28713204986897e-06\\
0.35	-1.78068026932718e-05\\
0.4	2.09169723190951e-06\\
0.45	2.05070766238917e-05\\
0.5	1.17290959129289e-05\\
0.55	-1.19326421832967e-05\\
0.6	-2.1054901005341e-05\\
0.65	-6.0468018805533e-06\\
0.7	1.49885556646806e-05\\
0.75	1.93945131742281e-05\\
0.8	3.74390828530251e-06\\
0.85	-1.53102896184111e-05\\
0.9	-1.94069585618186e-05\\
0.95	-5.56477897888764e-06\\
1	1.29655567582477e-05\\
1.05	1.99695731789232e-05\\
1.1	1.01716458827661e-05\\
1.15	-7.63076554674862e-06\\
1.2	-1.89765994788833e-05\\
1.25	-1.54494787061404e-05\\
1.3	-4.06358246995202e-07\\
1.35	1.46014965114315e-05\\
1.4	1.88452720724346e-05\\
1.45	9.83302288836236e-06\\
1.5	-5.71950885895208e-06\\
1.55	-1.7164197628258e-05\\
1.6	-1.72002823691208e-05\\
1.65	-6.25985137616174e-06\\
1.7	8.38678417179394e-06\\
1.75	1.76237594076773e-05\\
1.8	1.61037061538442e-05\\
1.85	5.11087161736377e-06\\
1.9	-8.66782671700211e-06\\
1.95	-1.73183531962337e-05\\
2	-1.61867785383853e-05\\
2.05	-6.19960499755432e-06\\
2.1	7.00280245360115e-06\\
2.15	1.63449853817829e-05\\
2.2	1.70768314766503e-05\\
2.25	9.0580607049179e-06\\
2.3	-3.4394545015325e-06\\
2.35	-1.40913475141903e-05\\
2.4	-1.77100450128022e-05\\
2.45	-1.2669703634184e-05\\
2.5	-1.48529594397977e-06\\
2.55	1.05823830885801e-05\\
2.6	1.81150525635482e-05\\
2.65	1.80731372556437e-05\\
2.7	1.11674107649323e-05\\
2.75	1.56145730967859e-06\\
2.8	-4.82685469707739e-06\\
2.85	-2.56107445013141e-06\\
2.9	1.15931762706033e-05\\
2.95	3.82577577355636e-05\\
3	7.65806033366936e-05\\
3.05	0.000126558535168435\\
3.1	0.00019177616350686\\
3.15	0.000281737064128635\\
3.2	0.000413255011109328\\
3.25	0.000610876667903813\\
3.3	0.000906815003613188\\
3.35	0.00134118803124995\\
3.4	0.0019633668325766\\
3.45	0.00283494565247441\\
3.5	0.00403436819914487\\
3.55	0.00566274888518547\\
3.6	0.00785007602135388\\
3.65	0.0107608661114702\\
3.7	0.0145984438638994\\
3.75	0.0196072497243512\\
3.8	0.0260727791649409\\
3.85	0.0343188104546872\\
3.9	0.0447014378152063\\
3.95	0.057599167223131\\
4	0.0733981265685875\\
4.05	0.0924715069229269\\
4.1	0.115152860980039\\
4.15	0.14170388862555\\
4.2	0.172278720138172\\
4.25	0.206888185414171\\
4.3	0.245368748271077\\
4.35	0.28736129097439\\
4.4	0.332304457835646\\
4.45	0.379445719442226\\
4.5	0.427870892696998\\
4.55	0.476550011223978\\
4.6	0.524394816375194\\
4.65	0.570321343777552\\
4.7	0.6133105159338\\
4.75	0.652460385821867\\
4.8	0.687025441128521\\
4.85	0.716440680474864\\
4.9	0.740330465857077\\
4.95	0.758503999316753\\
5	0.770940425647516\\
5.05	0.777766994842777\\
5.1	0.779233554824919\\
5.15	0.775686093359599\\
5.2	0.767541321934163\\
5.25	0.755263566504431\\
5.3	0.73934461131571\\
5.35	0.720286682971307\\
5.4	0.698588466425033\\
5.45	0.674733887703807\\
5.5	0.649183342226566\\
5.55	0.622367054014586\\
5.6	0.594680287421314\\
5.65	0.566480176013161\\
5.7	0.538083969112377\\
5.75	0.509768519491089\\
5.8	0.481770845809124\\
5.85	0.454289604127134\\
5.9	0.427487299146479\\
5.95	0.401493062550079\\
6	0.376405826650791\\
6.05	0.352297728619838\\
6.1	0.329217594378852\\
6.15	0.307194370978592\\
6.2	0.286240400238375\\
6.25	0.266354452452677\\
6.3	0.247524464960749\\
6.35	0.229729954499303\\
6.4	0.212944093208126\\
6.45	0.197135455143774\\
6.5	0.182269452918511\\
6.55	0.168309492766609\\
6.6	0.155217881393068\\
6.65	0.142956519979689\\
6.7	0.131487420409938\\
6.75	0.12077307675349\\
6.8	0.110776721958174\\
6.85	0.101462495977939\\
6.9	0.0927955476477358\\
6.95	0.0847420887518852\\
7	0.0772694151258083\\
7.05	0.0703459063915675\\
7.1	0.0639410131122575\\
7.15	0.0580252377717427\\
7.2	0.0525701140296608\\
7.25	0.0475481871336468\\
7.3	0.0429329971481793\\
7.35	0.0386990657360292\\
7.4	0.0348218865577425\\
7.45	0.0312779188919061\\
7.5	0.0280445837854784\\
7.55	0.0251002618829159\\
7.6	0.0224242920247877\\
7.65	0.0199969697248043\\
7.7	0.0177995447064162\\
7.75	0.0158142167889546\\
7.8	0.0140241295428277\\
7.85	0.012413361272816\\
7.9	0.0109669130287622\\
7.95	0.00967069347645484\\
8	0.00851150058446425\\
8.05	0.0074770001910989\\
8.1	0.00655570160823478\\
8.15	0.0057369304938076\\
8.2	0.00501079928404001\\
8.25	0.00436817551847385\\
8.3	0.00380064841892714\\
8.35	0.00330049409763187\\
8.4	0.0028606397722631\\
8.45	0.00247462735831125\\
8.5	0.00213657679345808\\
8.55	0.00184114942657273\\
8.6	0.00158351177685494\\
8.65	0.0013592999380742\\
8.7	0.00116458487014914\\
8.75	0.000995838786516458\\
8.8	0.000849902811765357\\
8.85	0.00072395605090826\\
8.9	0.000615486179684718\\
8.95	0.000522261635416341\\
9	0.000442305459817054\\
9.05	0.000373870820367493\\
9.1	0.000315418213808702\\
9.15	0.000265594335552147\\
9.2	0.00022321258192528\\
9.25	0.000187235137451551\\
9.3	0.000156756587750988\\
9.35	0.000130988989382758\\
9.4	0.000109248320766752\\
9.45	9.09422333251461e-05\\
9.5	7.5559019477795e-05\\
9.55	6.26577118328002e-05\\
9.6	5.18592287588721e-05\\
9.65	4.28384823191751e-05\\
9.7	3.53173669814691e-05\\
9.75	2.90585506239759e-05\\
9.8	2.38599929191183e-05\\
9.85	1.95501204452073e-05\\
9.9	1.59835922977999e-05\\
9.95	1.303759452227e-05\\
10	1.06086065402151e-05\\
};

\end{axis}
\end{tikzpicture}%

%% file: BC_Error.tikz
%
%
\definecolor{mycolor1}{rgb}{0.00000,0.44700,0.74100}%
\definecolor{mycolor2}{rgb}{0.85000,0.32500,0.09800}%
\definecolor{mycolor3}{rgb}{0.92900,0.69400,0.12500}%
\begin{tikzpicture}

\begin{axis}[%
width=0.951\fwidth,
height=\fheight,
at={(0\fwidth,0\fheight)},
scale only axis,
xmin=0,
xmax=10,
xlabel style={font=\color{white!15!black}},
xlabel={time [s]},
ymode=log,
ymin=1e-06,
ymax=0.001,
yminorticks=true,
ylabel style={font=\color{white!15!black}},
ylabel={\texttt{mean}($|\by(t) - \hat\by(t)|$)},
axis background/.style={fill=white},
legend style={legend cell align=left, align=left, draw=white!15!black, at = {(0.85,1.30), anchor = north east}}
]
\addplot [color=mycolor2, line width=2.0pt,dashdotted]
  table[row sep=crcr]{%
0	0\\
0.05	0.000137375890232447\\
0.1	0.000191382280215158\\
0.15	0.000140832563258178\\
0.2	0.000190854642381839\\
0.25	0.000118267216072087\\
0.3	0.000137475512606969\\
0.35	0.000189227167924039\\
0.4	9.97915588317049e-05\\
0.45	0.000136241984674141\\
0.5	0.000181006595710245\\
0.55	9.17645315758103e-05\\
0.6	0.000121543662397303\\
0.65	0.000151883519109255\\
0.7	0.000156731217388305\\
0.75	7.25378444677971e-05\\
0.8	0.000103001258508353\\
0.85	0.000146700053303006\\
0.9	0.000165159862575493\\
0.95	9.91835589363989e-05\\
1	7.78454011298667e-05\\
1.05	0.000124751094187165\\
1.1	0.000175658186973609\\
1.15	0.000147451011134713\\
1.2	5.40757699469116e-05\\
1.25	6.6934004443241e-05\\
1.3	0.000155429875649865\\
1.35	0.000184968120404673\\
1.4	0.000140994059550125\\
1.45	4.14623012614544e-05\\
1.5	7.51677888290116e-05\\
1.55	0.000165269060534907\\
1.6	0.000196141858110566\\
1.65	0.000157163427705335\\
1.7	6.25548673566196e-05\\
1.75	5.45451033090788e-05\\
1.8	0.000154113455180391\\
1.85	0.00020297003408991\\
1.9	0.000185528232266591\\
1.95	0.000108204550571914\\
2	1.12124877776475e-05\\
2.05	0.000114157154180108\\
2.1	0.000189490464353828\\
2.15	0.000207336098437835\\
2.2	0.000163551097818603\\
2.25	7.23910869712596e-05\\
2.3	3.86472668844939e-05\\
2.35	0.000137516314916374\\
2.4	0.000197031981421783\\
2.45	0.000202250005563504\\
2.5	0.00015387294236793\\
2.55	8.02700424916928e-05\\
2.6	3.71246112309673e-05\\
2.65	0.000121577955180947\\
2.7	0.000175014007256509\\
2.75	0.000182556947828414\\
2.8	0.000145146883859753\\
2.85	0.000108084055929654\\
2.9	7.50681520748767e-05\\
2.95	8.65234537216815e-05\\
3	0.000139189011021372\\
3.05	0.000157254034585331\\
3.1	0.000139523534013501\\
3.15	0.000125361531828316\\
3.2	0.000111317685143043\\
3.25	7.38961744248069e-05\\
3.3	8.37291513460943e-05\\
3.35	0.000113431339472404\\
3.4	0.000118692427444641\\
3.45	0.000130674384468929\\
3.5	0.000143046099091136\\
3.55	0.000127790776111482\\
3.6	8.75561131784543e-05\\
3.65	3.01813214159518e-05\\
3.7	4.1269878475979e-05\\
3.75	8.87776885578859e-05\\
3.8	0.000126372012284398\\
3.85	0.00013839064776523\\
3.9	0.000123275982164676\\
3.95	8.55524528756262e-05\\
4	3.48096274330436e-05\\
4.05	1.68640453320393e-05\\
4.1	5.62335241873843e-05\\
4.15	7.50868237921314e-05\\
4.2	6.92133363682072e-05\\
4.25	7.04741196493072e-05\\
4.3	9.84625278146312e-05\\
4.35	0.000117123019280813\\
4.4	0.000120281954428249\\
4.45	0.00011089693866532\\
4.5	0.000112050750974636\\
4.55	9.25585896383199e-05\\
4.6	5.77432081455284e-05\\
4.65	7.72833116174819e-05\\
4.7	0.000113721260654351\\
4.75	0.000132347419142876\\
4.8	0.000131128439558537\\
4.85	0.000111738425751618\\
4.9	7.88143924098129e-05\\
4.95	3.86900040385205e-05\\
5	1.96178224084811e-05\\
5.05	6.00992607246165e-05\\
5.1	9.83292272758451e-05\\
5.15	0.000130453031084908\\
5.2	0.000153726152406372\\
5.25	0.000166533837462013\\
5.3	0.000168287753889509\\
5.35	0.000159269794004879\\
5.4	0.000140469531365928\\
5.45	0.000113436027017577\\
5.5	8.01445800288902e-05\\
5.55	8.00361432064245e-05\\
5.6	9.05187368887839e-05\\
5.65	0.00010142728628465\\
5.7	0.000110187248701904\\
5.75	0.000114314018916761\\
5.8	0.000120757654339609\\
5.85	0.000135608430763712\\
5.9	0.000142192327167889\\
5.95	0.000140679846349828\\
6	0.00013179486009135\\
6.05	0.000116717869589278\\
6.1	9.69559875095283e-05\\
6.15	0.000101264385976491\\
6.2	0.000141823304510208\\
6.25	0.000178714641515809\\
6.3	0.00021035015288362\\
6.35	0.000235605694346677\\
6.4	0.000253842893021342\\
6.45	0.00026488590507423\\
6.5	0.000268964272688757\\
6.55	0.000266633964583124\\
6.6	0.000258688141896113\\
6.65	0.00024606755056103\\
6.7	0.000229778161574878\\
6.75	0.00021082119785773\\
6.8	0.000190138317960192\\
6.85	0.000168572691162015\\
6.9	0.000146845127799514\\
6.95	0.000125543347502151\\
7	0.000105121867439655\\
7.05	8.59097691585653e-05\\
7.1	6.8123702319442e-05\\
7.15	5.18838138360775e-05\\
7.2	4.70204457758243e-05\\
7.25	5.20803910908883e-05\\
7.3	5.63731113934762e-05\\
7.35	6.00085857435546e-05\\
7.4	6.30981450892645e-05\\
7.45	6.57451680724924e-05\\
7.5	6.80388136825427e-05\\
7.55	7.00506433767307e-05\\
7.6	7.18338084678711e-05\\
7.65	7.34243578270656e-05\\
7.7	7.48441370760217e-05\\
7.75	7.61047271992683e-05\\
7.8	7.72118758318166e-05\\
7.85	7.81699175587732e-05\\
7.9	7.8985748246654e-05\\
7.95	7.96720018075301e-05\\
8	8.02491735023819e-05\\
8.05	8.07465368388e-05\\
8.1	8.12018123290585e-05\\
8.15	8.16596314235324e-05\\
8.2	8.21689384282141e-05\\
8.25	8.27795419021137e-05\\
8.3	8.35380842063255e-05\\
8.35	8.44837374594323e-05\\
8.4	8.56439431848384e-05\\
8.45	8.70305149571857e-05\\
8.5	8.86363897963888e-05\\
8.55	9.04332811737252e-05\\
8.6	9.23704251563804e-05\\
8.65	9.43745497258195e-05\\
8.7	9.63511262125853e-05\\
8.75	9.81868860509257e-05\\
8.8	9.97535213295403e-05\\
8.85	0.000100912419612746\\
8.9	0.000101520232761805\\
8.95	0.00010143503342024\\
9	0.000100522783910197\\
9.05	9.86638244783692e-05\\
9.1	9.57590872696746e-05\\
9.15	9.17357532682969e-05\\
9.2	8.65520938931445e-05\\
9.25	8.02012730663706e-05\\
9.3	7.27139287636083e-05\\
9.35	6.41594036992799e-05\\
9.4	5.4645549089547e-05\\
9.45	4.43170836488259e-05\\
9.5	3.33525454017203e-05\\
9.55	2.20431967659239e-05\\
9.6	2.14439380428433e-05\\
9.65	2.14878135142944e-05\\
9.7	2.22342559524991e-05\\
9.75	2.3741810458979e-05\\
9.8	3.2855308477066e-05\\
9.85	4.16015214217837e-05\\
9.9	4.90543469005829e-05\\
9.95	5.50213237801054e-05\\
10	5.93446534044195e-05\\
};
\addlegendentry{Intrusive-POD}

\addplot [color=mycolor3, dashed, line width=2.0pt]
  table[row sep=crcr]{%
0	0\\
0.05	0.000155128347891632\\
0.1	0.000214265715905759\\
0.15	0.000214274181108273\\
0.2	0.000260874331109727\\
0.25	0.000198424557335657\\
0.3	0.000200868360359139\\
0.35	0.000244145913961129\\
0.4	0.000153734041921223\\
0.45	0.000184006515356003\\
0.5	0.000233854437917641\\
0.55	0.000110119905620312\\
0.6	0.000171326990614817\\
0.65	0.00020947489963272\\
0.7	0.000209150888282132\\
0.75	0.000102337035930324\\
0.8	0.000144964374866423\\
0.85	0.00020133977053422\\
0.9	0.000220185231866278\\
0.95	0.000127600518116021\\
1	0.00011116148708609\\
1.05	0.000168450208260286\\
1.1	0.000231321700719178\\
1.15	0.00018992541576515\\
1.2	6.50452930369179e-05\\
1.25	9.44705176057982e-05\\
1.3	0.000205460327235136\\
1.35	0.000239977859582877\\
1.4	0.000179582313136561\\
1.45	4.92857492827323e-05\\
1.5	0.000100493739722256\\
1.55	0.000214035029592145\\
1.6	0.000250742198827387\\
1.65	0.000198723032922993\\
1.7	7.74606662272663e-05\\
1.75	7.04143628254489e-05\\
1.8	0.000194646270933684\\
1.85	0.000254469465374155\\
1.9	0.000231555300116376\\
1.95	0.000134862715567347\\
2	3.57390942441186e-06\\
2.05	0.000139884911453853\\
2.1	0.000232549472988574\\
2.15	0.000254897647261635\\
2.2	0.000202261204294187\\
2.25	9.20794036557649e-05\\
2.3	4.24714935051232e-05\\
2.35	0.000162917224890141\\
2.4	0.000236649839669672\\
2.45	0.000245594978129276\\
2.5	0.000190177254832674\\
2.55	8.79079583593175e-05\\
2.6	3.88746619145148e-05\\
2.65	0.00013962267487236\\
2.7	0.00020745338251695\\
2.75	0.000221768105421252\\
2.8	0.000182565293871746\\
2.85	0.000114631512131208\\
2.9	8.21307841615606e-05\\
2.95	8.9524734558816e-05\\
3	0.000157104474989033\\
3.05	0.000185722969595075\\
3.1	0.000172406699903947\\
3.15	0.000132462593987109\\
3.2	0.000122958101557258\\
3.25	8.88228740967188e-05\\
3.3	8.70970207579763e-05\\
3.35	0.000129807482557853\\
3.4	0.000145077276336282\\
3.45	0.000134062267282542\\
3.5	0.000146077193788188\\
3.55	0.000137756795094061\\
3.6	0.000102661923085216\\
3.65	4.73864479013628e-05\\
3.7	4.11716474515835e-05\\
3.75	7.90806890905268e-05\\
3.8	0.000126026210371994\\
3.85	0.000149964153599172\\
3.9	0.000147942391041931\\
3.95	0.000122955519381477\\
4	8.31889102649998e-05\\
4.05	4.00011450619759e-05\\
4.1	3.36150601362103e-05\\
4.15	6.67461883881829e-05\\
4.2	0.000105506428725646\\
4.25	0.000145628652940319\\
4.3	0.000180631863526615\\
4.35	0.000203267182260953\\
4.4	0.000207468704951144\\
4.45	0.00019019940268703\\
4.5	0.000161989322739927\\
4.55	0.000146447197275627\\
4.6	0.000117228494652753\\
4.65	8.19682437786544e-05\\
4.7	5.97952906605768e-05\\
4.75	8.62953939461588e-05\\
4.8	9.25817574787217e-05\\
4.85	8.00832635056252e-05\\
4.9	5.32773932358649e-05\\
4.95	7.11716022434999e-05\\
5	0.000106947574179193\\
5.05	0.000143495750905213\\
5.1	0.000176383516915804\\
5.15	0.000202126047335403\\
5.2	0.000218372161677902\\
5.25	0.000223893307653628\\
5.3	0.000218456431814562\\
5.35	0.00020265079428039\\
5.4	0.00017771456840654\\
5.45	0.000145380977231102\\
5.5	0.000107743592527548\\
5.55	6.71298587393565e-05\\
5.6	6.15033347855787e-05\\
5.65	6.66367280563462e-05\\
5.7	6.96021917003797e-05\\
5.75	7.79106540781394e-05\\
5.8	9.98765399944368e-05\\
5.85	0.000113625768658782\\
5.9	0.000118894470712422\\
5.95	0.000116012729592052\\
6	0.000105845710686697\\
6.05	8.96905383337143e-05\\
6.1	0.000126634016668542\\
6.15	0.000168348304297933\\
6.2	0.000207500594367971\\
6.25	0.000242207294349761\\
6.3	0.00027099593396844\\
6.35	0.000292871798175609\\
6.4	0.000307331786649379\\
6.45	0.000314334033060373\\
6.5	0.000314234702106883\\
6.55	0.000307704291924754\\
6.6	0.000295635070111835\\
6.65	0.000279049499678913\\
6.7	0.000259017133778627\\
6.75	0.000236584917213227\\
6.8	0.000212723441983358\\
6.85	0.000188289672664509\\
6.9	0.000164005113283502\\
6.95	0.000140447335052273\\
7	0.000118052219277603\\
7.05	9.71240719681021e-05\\
7.1	7.78508860763e-05\\
7.15	6.03223695951824e-05\\
7.2	5.13936090008396e-05\\
7.25	5.79579613295793e-05\\
7.3	6.36729212013208e-05\\
7.35	6.86446734259188e-05\\
7.4	7.2981774272677e-05\\
7.45	7.67859120737908e-05\\
7.5	8.01455733495018e-05\\
7.55	8.31324591776036e-05\\
7.6	8.5800333922607e-05\\
7.65	8.81858801548459e-05\\
7.7	9.03110622798922e-05\\
7.75	9.21864886122923e-05\\
7.8	9.38152760857833e-05\\
7.85	9.51969709005246e-05\\
7.9	9.63311494729131e-05\\
7.95	9.72204050196431e-05\\
8	9.78725131516398e-05\\
8.05	9.83016597689997e-05\\
8.1	9.85287086066183e-05\\
8.15	9.85805496082084e-05\\
8.2	9.84886466173489e-05\\
8.25	9.82869492270207e-05\\
8.3	9.80093712096934e-05\\
8.35	9.76870615385119e-05\\
8.4	9.73456931255764e-05\\
8.45	9.70029913399458e-05\\
8.5	9.66666936267724e-05\\
8.55	9.63331051968123e-05\\
8.6	9.59863688201966e-05\\
8.65	9.55985203609558e-05\\
8.7	9.51303525688932e-05\\
8.75	9.45330539560064e-05\\
8.8	9.37505485386772e-05\\
8.85	9.27224155451874e-05\\
8.9	9.13872387846499e-05\\
8.95	8.96862076438747e-05\\
9	8.75667792304405e-05\\
9.05	8.49862059346975e-05\\
9.1	8.19147413461672e-05\\
9.15	7.83383538664965e-05\\
9.2	7.426080287743e-05\\
9.25	6.97049637899449e-05\\
9.3	6.47133248404694e-05\\
9.35	5.93476183637651e-05\\
9.4	5.36875895055536e-05\\
9.45	4.93002359451513e-05\\
9.5	4.67429818870231e-05\\
9.55	4.43970676865969e-05\\
9.6	4.23046238668467e-05\\
9.65	4.05072161030739e-05\\
9.7	3.90449729985675e-05\\
9.75	3.79554677347637e-05\\
9.8	3.72723989394584e-05\\
9.85	3.70241300943768e-05\\
9.9	3.72321712327151e-05\\
9.95	3.79096959372883e-05\\
10	3.90601931784806e-05\\
};
\addlegendentry{Learned ROM}

\end{axis}
\end{tikzpicture}%

%% file: BC_DiffFullState_approxResidual.tikz
%
%
\definecolor{mycolor1}{rgb}{0.00000,0.44700,0.74100}%
\definecolor{mycolor2}{rgb}{0.85000,0.32500,0.09800}%
\definecolor{mycolor3}{rgb}{0.92900,0.69400,0.12500}%
\begin{tikzpicture}

\begin{axis}[%
width=0.973\fwidth,
height=\fheight,
at={(0\fwidth,0\fheight)},
scale only axis,
xmin=20,
xmax=50,
xlabel style={font=\color{white!15!black}},
xlabel={reduced model dimension},
ymode=log,
ymin=2.30461281735256e-07,
ymax=0.00161216349900744,
yminorticks=true,
ylabel style={font=\color{white!15!black}},
ylabel={Error b/w FOM and ROM},
axis background/.style={fill=white},
legend style={legend cell align=left, align=left, draw=white!15!black,at = {(1,1.15), anchor = north east }}
]
\addplot [color=mycolor2, dashdotted, line width=2.0pt]
  table[row sep=crcr]{%
20	0.00161216349900744\\
22	0.000682993138380743\\
24	0.00034531935475376\\
26	0.000171483590582855\\
28	8.15863003277372e-05\\
30	4.1656043599392e-05\\
32	2.18925831453569e-05\\
34	1.14629612386843e-05\\
36	1.61571373468581e-05\\
38	9.88563677757842e-06\\
40	9.92354859464153e-06\\
42	1.12898283668336e-06\\
44	1.0429095204538e-06\\
46	2.90217171688026e-07\\
48	2.35767328008552e-07\\
50	2.30461281735256e-07\\
};
\addlegendentry{$\|\bS-\bV\hat\bS^{\text{POD}}\|_2/\|\bS\|_2$}

\addplot [color=mycolor3, dashed, line width=2.0pt]
  table[row sep=crcr]{%
20	0.00128705430574921\\
22	0.000652998495733864\\
24	0.000329971950106206\\
26	0.00016434924914329\\
28	0.000118239683361719\\
30	4.10835399915565e-05\\
32	0.000163735171822294\\
34	1.10411684749263e-05\\
36	6.43789794284248e-05\\
38	1.01953832516015e-05\\
40	0.000197632852065158\\
42	1.54090591900028e-06\\
44	1.15211188433786e-06\\
46	7.06692021944402e-07\\
48	2.45620426506554e-07\\
50	2.33678408260174e-07\\
};
\addlegendentry{$\|\bS-\bV\hat\bS^{\text{NI}}\|_2/\|\bS\|_2$}

\end{axis}
\end{tikzpicture}%